\documentclass[11pt]{amsart}
\usepackage{amsmath}
\usepackage{amsthm}
\usepackage{amsfonts}
\usepackage{amssymb}
\usepackage{verbatim}
\usepackage{graphicx}
\usepackage[margin=1.0in]{geometry}
\usepackage{url}
\usepackage{enumerate}
\usepackage[pdftex,bookmarks=true]{hyperref}
\newcommand{\norm}[1]{\left\lVert #1 \right\rVert}
\newcommand{\ab}[1]{\left| #1 \right|}

\usepackage{color}

\newcommand\R{{\mathbb{R}}}
\newcommand\C{{\mathbb{C}}}
\newcommand\Z{{\mathbf{Z}}}
\newcommand\D{{\mathbf{D}}}

\renewcommand\P{{\mathbf{P}}}
\newcommand\E{{\mathbf{E}}}

\newcommand\Var{\mathbf{Var}}
\renewcommand\Im{{\operatorname{Im}}}
\renewcommand\Re{{\operatorname{Re}}}
\newcommand\eps{{\varepsilon}}

\newcommand\BI{{\mathbf I}}

\newcommand\CE{{\mathcal E}}

\newcommand\CT{{\mathcal T}}

\newcommand\ep{{\epsilon}}

\theoremstyle{plain}
  \newtheorem{theorem}{Theorem}[section]

  \newtheorem{proposition}[theorem]{Proposition}
  
  \newtheorem{lemma}[theorem]{Lemma}
  \newtheorem{corollary}[theorem]{Corollary}
  \newtheorem{claim}[theorem]{Claim}

\theoremstyle{definition}
  \newtheorem{definition}[theorem]{Definition}
  
  \newtheorem{remark}[theorem]{Remark}

\begin{document}

\title {Roots of   random polynomials  with coefficients having  polynomial growth}

\author{Yen Do} 
\address{(Yen Do) Department of Mathematics, The University of Virginia, Charlottesville, VA 22904, USA}
\email{yendo@virginia.edu}
\thanks{Y.D. partially supported by NSF Grants DMS--1201456 and DMS--1521293.}
\author{Oanh Nguyen}  
\address{(Oanh Nguyen and Van Vu) Department of Mathematics, Yale University, New Haven, CT 06520, USA}
\email{oanh.nguyen@yale.edu}
\thanks{O. N and V. V. are partially supported by NSF Grant  DMS-1307797 and  AFORS Grant  FA9550-12-1-0083.  
Part of this work was done at VIASM (Hanoi), and the authors would like to thank the institute for the support and hospitality. }

\author{Van Vu}
\email{van.vu@yale.edu}

\begin{abstract} 
In this paper, we prove optimal local universality  for roots of random polynomials with arbitrary 
coefficients of polynomial growth. As an application, we derive, for the first time, sharp estimates for 
the number of real roots of these polynomials, even  when the coefficients are not explicit.
  Our results also hold for series; in particular, we prove local universality 
for random hyperbolic series. 

\end{abstract}

\maketitle
\tableofcontents

\section{A motivation: Real roots of random polynomials}

\vskip2mm 

Let us start by describing a natural and famous problem which serves as the motivation  of our studies, the main results of which will be 
discussed in Section \ref{s.correlation}. 

Finding real roots of a high degree polynomial is  among the most basic problems in  mathematics. 
From the algebraic point of view, it is classical that for   most polynomials of degree at least 5 
the roots cannot be computed in radicals, thanks to the fundamental works of Abel-Ruffini and Galois. There has been  a huge amount  of results on the number of real roots and also 
their locations  using information from the coefficients (for instance, one of the earliest results is Descartes' classical theorem concerning sign sequences), however most results
are often sharp for certain special classes of polynomials, but poor in many others.

It is   natural and important  to consider the root problem from  the statistical point of view. What can we say about   a {\it typical (i.e. random)} polynomial? 
Already in the seventeenth century, Waring considered   random cubic polynomials  and concluded  that the probability 
of having three real roots is at most $2/3$. This effort was discussed by Toddhunter in \cite{Todd}, one of the earliest books in probability theory, 
which also reported a similar effort  made by Sylvester. 
However,  the distribution of the polynomials was not explicitly defined at the time.  


In the last hundred years,  random polynomials have attracted the attention of many generations of mathematicians,
with most efforts directed to the following model
$$ P_{n, \xi} (x) :=    c_n  \xi_n x^n + \dots  + c_1 \xi_1 x + c_0 \xi_0 x^0, $$ where $\xi_i$ are iid copies of a random variables $\xi$ with zero mean and unit variance, and $c_i$ 
are deterministic coefficients which may depend on both $n$ and $i$. Different  definitions of $c_i$  give rise to  different classes of random polynomials, which have  different behaviors. 
When $c_i =1$ for all $i$, the polynomial $P_{n,\xi}$ is often referred to as  the Kac  polynomial.  Even for this special case, the literature is very rich (see \cite{BS1, F1} for surveys). 
 In the next few paragraphs, we  will discuss few seminal results which directly motivate our research. 

The first modern work on random polynomials was due to Bloch and Polya in 1932 \cite{BP1}, who  considered the Kac polynomial with   $\xi$ being 
Rademacher  (namely $\P( \xi =1) =\P(\xi =-1) =1/2 $), and showed that with high probability 
$$ N_{n ,\xi}  =O( \sqrt  n ) , $$ where $N_{n,\xi}$ denotes the number of real roots of the Kac polynomial associated with the random variable $\xi$.  Their key idea is simple and beautiful. Notice that if we apply
Descartes' rule of signs for $P_n$,  one could only obtain the trivial bound $O(n)$ for $N_{n,\xi}$ as the typical number of sign changes is around $n/2$.  Bloch and Polya's idea is to apply 
Descartes rule for $P_n Q$, where $Q$ is a deterministic polynomial  which does not have any real positive  roots. By choosing $Q$  properly, they reduced the 
 number of sign changes  significantly.

Next came the ground breaking  series of papers by Littlewood and Offord \cite{LO1, LO2, LO3} in the early 1940s, 
which, to the surprise of many mathematicians of their time,  showed that $N_{n, \xi}$ is typically polylogarithmic in $n$.

\begin{theorem} [Littlewood-Offord]
For $\xi$ being Rademacher, Gaussian, or uniform on $[-1,1]$, 

$$ \frac{\log n} {\log \log n}  \le N_{n, \xi} \le \log^2 n$$ with probability $1-o(1)$. 

\end{theorem} 

Littlewood-Offord's papers    and later works of Offord  \cite{LO1, LO2, LO3}  lay the foundation for the theory of random functions, which is an important part of modern 
probability and analysis, see for instance \cite{PV1, NS1}.

During more or less the same time, Kac \cite{K1} discovered his famous formula for the density function  $\rho(t)$  of $N_{n, \xi}$
\begin{equation} \label{Kac0} \rho(t) =   \int_{- \infty} ^{\infty} |y| p(t,0,y) dy, \end{equation}   where 
$p(t,x,y)$ is the joint probability density for $P_{n, \xi}  (t) =x$ and the derivative $P'_{n, \xi} (t) =y$.

Consequently, 
 \begin{equation} \label{Kacformula}     \E N_{n ,\xi}  = \int_{-\infty}^{\infty}   dt \int_{- \infty} ^{\infty} |y| p(t,0,y) dy.
\end{equation}

In the Gaussian case ($\xi$ is Gaussian), the joint distribution  of $P_{n, \xi} (t)$ and $P'_{n, \xi}(t)$ can be explicitly computed. Using this fact, Kac showed in \cite{K1} that
\begin{equation}  \label{Kacformula2}   \E N_{n,  Gauss}  = \frac{1}{\pi} \int_{-\infty} ^{\infty} \sqrt { \frac{1}{(t^2-1) ^2} + \frac{(n+1)^2 t^{2n}}{ (t^{2n+2} -1)^2}  }  dt       =   (\frac{2}{\pi} +o(1)) \log n. \end{equation} 

A more careful evaluation  by Wilkins \cite{W1}  and also Edelman and Kostlan  \cite{EK1} gives 
\begin{equation} \label{Kacformula3}     \E N_{n ,Gauss}   = \frac{2}{\pi} \log n + C +o(1), 
\end{equation} where $C$ is an explicit constant. As a matter of fact, Wilkins \cite{W1} computed all terms in the Taylor expansion of the integration in \eqref{Kacformula2}.

 In his original paper \cite{K1}, Kac  thought that his formula would lead to the same estimate for   $\E N_{n, \xi}$ for all other random variables $\xi$. It has turned  out to be not the case, as the right-hand side of 
 \eqref{Kacformula} is often hard to compute, especially  when $\xi$ is discrete (Rademacher for instance).
 Technically, the  joint distribution  of $P_{n, \xi} (t)$ and $P'_{n, \xi}(t)$ is easy to determine in the Gaussian case, thanks to special properties of the Gaussian distribution, 
  but can pose a great challenge 
 in the general one.  Kac admitted this  in  a later paper \cite{K2}, in which he  managed to push  his method to 
  treat the case  $\xi$ being uniform in $[-1,1]$. A further extension was  made by  Stevens \cite{S2}, who evaluated Kac's formula for   a large class  of  $\xi$ having continuous and smooth  distributions
with certain regularity properties (see \cite[page 457]{S2} for details).

The treatment of $\E N_{n, \xi}$ for discrete random variables $\xi$ required considerable effort.  More than 10 years after Kac's paper \cite{K1}, 
Erd\H{o}s and Offord   \cite{EO} found a completely new approach  to handle the Rademacher  case, proving 
 that with probability $1- o(\frac{1} {\sqrt {\log \log n} } )$
\begin{equation} \label{e.erdos-offord}
N_{n, \xi}  = \frac{2}{\pi} \log n + o(\log^{2/3} n \log \log n). 
\end{equation}

The argument of Erd\"os and Offord is combinatorial and very delicate, even by today's standard. Their main idea is to approximate the number of roots by the number of sign changes 
in $P_{n, \xi} (x_1) , \dots, P_{n, \xi}(x_k)$ where $x_1, \dots, x_k$  is a carefully defined deterministic  sequence of points of length  $k = (\frac{2}{\pi} +o(1)) \log n$.  The authors showed that with high probability, almost every interval 
$(x_i, x_{i+1} )$ contains exactly one root, and used this fact to prove \eqref{e.erdos-offord}.

It took another ten years until  Ibragimov and Maslova \cite{IM1, IM2}  successfully extended  the method of Erd\"os-Offord  to treat
the Kac polynomials associated with more general distributions of $\xi$.

\begin{theorem} \label{theorem:IM}  For  any $\xi$ with mean zero which belongs to the domain of attraction of the normal law, 
\begin{equation}  \label{IM-1} 
\E  N_{n, \xi}  = \frac{2}{\pi} \log n + o(\log n) .
\end{equation}  \end{theorem} 

 For related results, see also \cite{IM3, IM4}. 
Few years later, Maslova  \cite{M1, M2} showed that if $\xi$ has mean zero and variance one and $\P (\xi =0) =0$, then the variance of $N_{n, \xi}$ is 
$(\frac{4}{\pi} (1- \frac{2}{\pi} )+o(1) ) \log n$, and $N_{n, \xi}$ satisfies the central limit theorem.  We also like to mention a very recent result of Soze that gives a strong upper bound 
for the number of real roots in the case when $\xi$ has arbitrary distribution \cite{Soz}.


So, after more than three decades of continuous research, a satisfactory answer for the Kac polynomial  (the base case when all $c_i=1$)  was obtained. 
Apparently, the next question is what happens with  {\it more general }  sets of coefficients ?

This general problem is very hard and  still far from being settled. Let us recall  that  Kac's formula for the density function  \eqref{Kacformula} 
applies  for all random polynomials. However, in practice one can only evaluate this formula in   the  Gaussian case and some other very nice 
continuous distributions.  On the other hand,   Erd\"os-Offord's argument seems too delicate and  relies heavily on the fact that all $c_i=1$.
 For a  long time, no analogue of  Theorem \ref{theorem:IM} 
was available for general   sets of coefficients  $c_i$ with respect  to non-Gaussian random variables $\xi$. 

\subsection{Description of the new results for  coefficients with zero means}
In this paper, we  prove  universality results for  general random polynomials where the coefficients $c_i$ have polynomial growth. 
These universality results show that, among other,  the expectation of the number of real roots depend only on the mean and variance of the coefficients $\xi$
 (two moment theorems). Thus, the problem of finding the expectation of real roots reduces to the gaussian case, which we can handle using an analytic argument (see the last paragraph of 
 Section \ref{section:ideas}). 
 
As the reader will see in the next section, our universality results show much more than just the expectation. They completely describe the 
local behavior of the roots (both complex and real). More generally, we can also control the number of intersection of the graph of the 
random  polynomial with any deterministic curve of given degree. (The number of real roots is the number of intersections with the $x$-axis.)

Thanks to new and powerful tools,   our method does not require an explicit expression for the deterministic coefficients $c_i$. 
As a corollary,  we obtain  the following extension (and refinement) of Theorem \ref{theorem:IM}.  
To formulate this result (see Theorem~\ref{theorem:new1}), we first introduce a definition. 

\begin{definition}\label{d.genpoly}
We say that $h(k)$ is a generalized polynomial  if there exists a finite sequence $0< L_0<\dots <L_d<\infty$  such that for some $\alpha_0,\dots, \alpha_d\in \R$ with $\alpha_d\ne 0$ it holds that  
$$h(k) = \sum_{j=0}^d \alpha_j \frac{L_j(L_j+1)\dots (L_j+k-1)}{k!}\qquad \text{for every $k=0,1,\dots, n$.}$$
Here, we understand that $L\dots (L+k-1)/k! \equiv 1$ if $k=0$. We will say that the degree of $h$ is $L_d-1$ in this case. We say that $h$ is a real generalized polynomial if the coefficients $\alpha_j$'s are real.
\end{definition}

It is clear that  any classical polynomial is also a generalized polynomial with the same degree: if $h(k)$ is a classical polynomial with degree $d$ then it  could be written as a linear combination of the binomial polynomials $\frac{L_j(L_j+1)\dots (L_j+k-1)}{k!}$ with $L_j=j+1$ for $j=0,1,\dots, d$, we also have $\alpha_d\ne 0$ because it is a nonzero multiple of the leading coefficient of $h$; therefore the degree in the generalized sense of Definition~\ref{d.genpoly} is also $d$. On the other hand, the class of generalized polynomials is much richer  as the degree of $h$ could be fractional.

For a polynomial $P$ and a subset $S\subset \C$, denote by $N_P(S)$ the number of zeros of $P$ in $S$. As usual, our random polynomials have the form 

$$P_{n}(z) = \sum_{i=0}^n c_i \xi_i z^i . $$ 

\begin{theorem} \label{theorem:new1}  Let $N_0$ be a nonnegative constant. Let $\xi_0,\dots, \xi_n$ be independent (but not necessarily iid) real-valued random variables with variance 1 and $\sup_{j=0, \dots, n} \E |\xi_j|^{2+\epsilon}< C_0$ for some constant $C_0 >0$ and  $\xi_i$ has mean 0 for all $i\ge N_0$.  Let $h$ be a
fixed  generalized polynomial  with a positive leading coefficient.  Assume that there are positive constants,  $M,m,C_1$   such that the real deterministic coefficients  $c_0,\dots, c_n\in \R$ satisfy
$$\begin{cases}
m h(k) \le  c_k^2  \le Mh(k), & N_0\le k\le n\\
c_k^2 \le C_1 M, & 0\le k<N_0
\end{cases}$$
Then
\begin{equation}\label{new2}
\frac{m^2}{M^2}\Big[\frac{1+\sqrt{\deg(h)+1}}{\pi} \log n + O(1)\Big] \le \E N_{P_{n}}(\R)  \le \frac{M^2}{m^2}\Big[\frac{1+\sqrt{\deg(h)+1}}{\pi} \log n + O(1)\Big]
\end{equation}
The implicit constants  in $O(1)$ depend on $\epsilon$, $C_0$, $C_1$, $N_0$, $h$, and the ratio $M/m$.  In particular, if  $c_k^2 = h(k)$ for some real (generalized)  polynomial $h$ of degree $d$ then
\begin{equation} \label{new3} \E N_{P_{n}}(\R) = \frac{1+\sqrt{d+1}}{\pi}\log n  + O(1). \end{equation} 
\end{theorem}

Notice that  the zeros of $P_{n}$ is invariant under  the scaling of $c_j$'s,  this explains why we only need dependence on the ratio $M/m$ instead of both $M$ and $m$. In the proof we may assume $M=1$ without loss of generality. The first few $\xi_i, i < N_0$ can have arbitrary means.

Theorem \ref{theorem:new1} is a corollary of our main local university result discussed in the next section.  This result (formulated in term of correlation functions) proves universality 
for not only the expectation, but higher moments of the number of roots  (complex or real) in any small region of microscopic scale. 
We delay the discussion of universality to the next section and  make a few comments on Theorem \ref{theorem:new1}.

First,   the error term in \eqref{new3} is only $O(1)$, which is best possible, as showed in  \eqref{Kacformula3}.
Even in the well-studied case of  Kac polynomials (all $c_i=1$), this 
gives a improvement 
\begin{equation} \label{new4} \E N_{n, \xi} = \frac{2}{ \pi} \log n +O(1) \end{equation} 
upon  the 
estimate  $ \frac{2}{ \pi} \log n +o (\log n)$  from Theorem \ref{theorem:IM} by Ibragimov and Maslova. 
We believe that the method used by Erdos and Offord  and also Ibragimov and Maslova cannot lead to error term better than  
$O(\sqrt {\log n })$.  \eqref{new4} was also proved by H. Nguyen and the last two authors in \cite{NNV} by other means, but the method there does not go beyond the Kac polynomials; see also
\cite{DNV}.

Second,  there are many natural families of random polynomials which satisfy the assumptions in Theorem \ref{theorem:new1}.  Here are a  few examples: 

{\it  Derivatives of the Kac polynomial .}  The roots of the derivatives of a function have strong analytic and geometric meanings, and  thus are of particular interests.
For the $d$th derivative of the Kac polynomial (any fixed $d\ge 0$) our result implies 
$$\E N_{P_{n}}(\R)=  \frac{1 +\sqrt {2d +1} } {\pi} \log n + O(1) . $$ 
Prior to this, for derivatives of the Kac polynomial only  weaker estimates (with  error terms $o(\log^{1/2} n)$) are available for the Gaussian case, see  the works of Das \cite{D1, D2} for $d=1,2$ and the extension in \cite{sambandham1979, FGK} to the setting when $\xi_j$'s are weakly correlated Gaussian random variables. For the first derivative ($d=1$), Maslova \cite{M2} considered 
non-Gaussian polynomials and obtained an asymptotic bound with worse error term $o(\log n)$.

{\it Hyperbolic  
polynomials.} Random hyperbolic  polynomials are defined by 

$$c_i :=  \sqrt {\frac{ L(L+1) \cdots (L+i-1)}{i!} },$$ for a constant $L >0$.  This class of random polynomials includes 
the Kac polynomials as a sub case ($L=1$) and has became very popular recently due to the invariance of the zeros of the corresponding infinite series under hyperbolic transformations; see \cite{HKPV1} for more discussion.  By Theorem~\ref{theorem:new1}, we have

$$ \E N_{P_{n}}(\R)=  \frac{1 + \sqrt {L}} {\pi} \log n +O(1)  \ \ .$$


{\it Logarithmic expectation.}  Another  immediate corollary of Theorem~\ref{theorem:new1} is  that $\E N_{P_{n}}(\R)$ grows  logarithmically  if the deterministic coefficients $c_j$ have polynomial growth:

\begin{corollary} \label{c.new1}  Consider $\xi_i$ as in Theorem \ref{theorem:new1}. 
 Assume that there are positive constants,  $C_0,C_1$ and some constant $\rho >-1/2$ such that the real deterministic coefficients  $c_0,\dots, c_n\in \R$ satisfy
$$\begin{cases}
C_0 k^\rho   \le  |c_k|  \le C_1  k^\rho , & N_0\le k\le n\\
c_k^2 \le C_1 , & 0\le k<N_0
\end{cases}$$ 
Then there are positive constants $C_2$, $C_3$ such that 
\begin{equation}\label{new5}
C_2 \log n \le \E N_{P_{n}}(\R)  \le C_3 \log n \ \ .
\end{equation}
Here $C_2, C_3$ depend only on $C_0,C_1,\rho, N_0$, and $\ep$.
\end{corollary} 

To deduce this result from Theorem~\ref{theorem:new1}, simply let $L=2\rho+1$ and notice that the binomial coefficient
$$h_0(k)=\frac{L(L+1)\dots (L+k-1)}{k!}$$ 
is about the size of $k^{L-1}=k^{2\rho}$ for $k$ large, therefore   the desired conclusion follows from Theorem~\ref{theorem:new1} via comparing $|c_k|$ with $\sqrt{h(k)}$.

\begin{corollary} \label{c.new2}  Consider $\xi_i$ as in Theorem \ref{theorem:new1}. 
 Assume that there are positive constants,  $C_0,C_1$ and some constant $\rho >-1/2$ such that the real deterministic coefficients  $c_0,\dots, c_n\in \R$ satisfy
$$\begin{cases}
|c_k| = C_0  k^\rho (1+o(1)), & N_0\le k\le n\\
c_k^2 \le C_1 , & 0\le k<N_0
\end{cases}$$   
 Then 
\begin{equation}\label{new6}
\E N_{P_{n}}(\R) = \frac{1+\sqrt{2\rho+1}}{\pi} \log n + o(\log n)  \ \ .
\end{equation}
\end{corollary} 
The Gaussian setting of Corollary~\ref{c.new2} in the special case $c_k=k^\rho, \rho \ge 0$ was considered by \cite{D1, D2}, see also the extension in \cite{sambandham1979,FGK}.

 To see Corollary~\ref{c.new2}, we need to show that given any $\delta>0$ it holds that 
$$ (1-\delta)\frac{1+\sqrt{2\rho+1}}{2\pi} \log n \quad \le \quad \E N_{P_{n}}(\R) \quad \le \quad (1+\delta)\frac{1+\sqrt{2\rho+1}}{2\pi} \log n $$
 for all $n$ sufficiently large. Again by comparing with $\sqrt{h(k)}=\sqrt{(2\rho+1)\dots (2\rho+k)/k!}$  and rescaling all $c_j$ if necessary we may assume that
$$(1+\delta)^{-1/10}h(k) \le |c_k|^2 \le (1+\delta)^{1/10}h(k)$$
for $k\le n$ sufficiently large (the threshold now depends on $\rho$ and (polynomially) on $\delta$). Applying Theorem~\ref{theorem:new1} we obtain the desired conclusion.


The reader can also notice that by Definition \ref{d.genpoly}, our generalized polynomials always have degree greater than -1. This corresponds to the assumption that $\rho >-1/2$ in Corollaries \ref{c.new1} and \ref{c.new2}. This assumption is important for our results. For example, consider the model when $c_i = i^{\rho}$ with $\rho <-1/2$ and $\Var \xi_i = 1$ for all $i$, then $\Var P_{n}(\pm 1) = \sum_{i=0}^{n} i^{2\rho}$ converges as $n\to \infty$. Intuitively, this says that the contribution of the first few terms becomes important and one may not expect to see universality around $\pm 1$ which is where most of the real roots locate. 

{\it Number of crossings.} The number of real roots is the number of intersections of the graph of $P(z)$ (over the real) with the line $y=0$. What about an arbitrary line  ? (The line $y=T$ is of particular interest, as it corresponds to 
the important notation of level sets.) For Kac polynomials, this question was considered  (see \cite{F1} for a survey) in the Gaussian case, and it was showed that the number of crossing (in expectation) is 
asymptotically  $(\frac{2}{\pi} +o(1)) \log n$. 

Theorem \ref{theorem:new1} allows us to prove a more precise result  in much more general setting, where we can consider the number of intersection with any polynomial curve of constant degree. 

\begin{corollary} \label{cor:crossing} 
 Consider $\xi_i$ as in Theorem \ref{theorem:new1}. Assume that  $c_k^2 = h(k)$ for some real (generalized)  polynomial $h$ of degree $d$. Let $f$ be a deterministic  polynomial of 
 degree $l$ and $\Gamma$ be its graph over the real. 
 Let $ N_{P_{n}, \Gamma}(\R) $ be the number of intersections of the graph of $P_{n}$ (over the real) with $\Gamma$. Then 
\begin{equation} \label{new03} \E N_{P_{n}, \Gamma}(\R) = \frac{1+\sqrt{d+1}}{\pi}\log n  + O(1), \end{equation}  where the constant in $O(1)$ depends on $\epsilon, N_0, h$ and $f$.
\end{corollary} 

Corollary \ref{cor:crossing} can be derived by applying Theorem \ref{theorem:new1} to the random polynomial $P_n - f$.

{\it The Gaussian case.}  The strategy of the proof of Theorem \ref{theorem:new1} is to reduce 
to the Gaussian case, using universality results presented in the next section (which are the main results of this paper). Let us emphasize 
 that even in the Gaussian setting,  Theorem~\ref{theorem:new1} (and Theorem~\ref{theorem:new2}  below) are substantially new and the method of proof is novel compared to previous works.  For more details, see the last paragraph of Section \ref{section:ideas}.

\subsection{Polynomials with coefficients having  non-zero means}

To conclude this section,  let us mention that our method could also  be used to handle polynomials with non-zero means.  For instance, we have the following analogue of Theorem \ref{theorem:new1}.

\begin{theorem} \label{theorem:new2}  Let $N_0$ be a positive constant and $h$ be a deterministic classical polynomial with real coefficients. Let $\xi_0,\dots, \xi_n$ be independent real-valued random variables with variance 1 and $\sup_{j=0, \dots, n} \E |\xi_j|^{2+\epsilon}< C_0$ for some 
constant $C_0 >0$. Assume that $\xi_i$ has mean $\mu \ne 0$ and $c_i= h(i)$ for $i\ge N_0$ and that $|c_i| \le C_1$ for $i<N_0$.  Let $P_{n} (z)= \sum_{i=0}^{n} c_i\xi_iz^{i}$.  
Then
\begin{equation}\label{new2-nonzeromean}
\E N_{P_{n}}(\R)  = \frac{1+\sqrt{2\deg(h)+1}}{2\pi} \log n + O(1) 
\end{equation}
The implicit constant  in $O(1)$ depends on $\epsilon$, $C_0$, $C_1$, $N_0$, $h$, and $\mu$. 
\end{theorem} 
 
The key feature of this result is that the number of real roots reduces by a factor of 2, compared to Theorem \ref{theorem:new1}. 
To our best knowledge, such a result was available only for Kac polynomials.
Farahmand \cite{F1} showed that when $\xi$ is gaussian with nonzero mean, $\E N_{P_{n}}(\R) = (1+o(1))\frac{1}{\pi}\log n$. 
 Ibragimov and Maslova in \cite{IM3} proved the same estimate if $\xi$ belongs to the domain of attraction of normal law. 
Even for Kac polynomials, our result improves upon these as it achieves the optimal error term $O(1)$. The analogue of Corollary \ref{cor:crossing} holds for this model.

Similarly to Theorem \ref{theorem:new1}, Theorem \ref{theorem:new2} will be derived from a general universality result provided in the next section. 
These results can also be used to treat higher moments (such as the variance) of the number of real roots. Details will appear elsewhere.

\subsection{Outline of the paper}

In the next section \ref{s.correlation}, we will present our main  results regarding the  local universality of the joint distribution of the zeros of $P_{n,\xi}$ when the deterministic coefficients $c_j$ have polynomial growth. 
 Special cases of these results will be used to reduce the proof of Theorems~\ref{theorem:new1} and \ref{theorem:new2} to the Gaussian setting (see the discussion near the end of Subsection~\ref{s.real} for details). 
 In Section~\ref{s.series},  we will also discuss  several extensions  regarding universality for the zeros of random power series. Among others, we achieve local universality of hyperbolic series under 
very general  assumptions. A sketch of our proofs for these results is presented is Section \ref{section:ideas}, followed by the detailed proofs in Sections~\ref{proof-complex}, \ref{proof-real}, \ref{proof-mean}, \ref{proof-complex-series}. In the rest of the paper (from Section~\ref{s.Gaussian} to the end), we prove the Gaussian case of Theorems~\ref{theorem:new1} and \ref{theorem:new2}. 

The current  paper uses and  further develops the method introduced in  \cite{TVpoly} by Tao and the third author. On the other hand, the arguments given in this paper are self-contained and the reader does not need to be familiar with \cite{TVpoly}. 

\section{ Correlation functions and Universality}  \label{s.correlation}

Correlation functions are effective tools to study random point processes. 
To define correlation functions, let us first consider the complex case in which the coefficients $c_i$ and the atom distribution $\xi$ are not required to be real valued.  In this case the point process $\{\zeta_1,\ldots,\zeta_n\}$ of zeroes of a random polynomial $P=P_{n}$ can be described using the (complex) \emph{$k$-point correlation functions} $\rho^{(k)} = \rho^{(k)}_P: \C^k \to \R^+$, defined for any fixed natural number $k$ by requiring that
\begin{equation}\label{cord}
 \E \sum_{i_1,\ldots,i_k \hbox{ distinct}} \varphi(\zeta_{i_1},\ldots,\zeta_{i_k}) = \int_{\C^k} \varphi(z_1,\ldots,z_k) \rho^{(k)}(z_1,\ldots,z_k)\ dz_1\ldots dz_k
 \end{equation}
for any continuous, compactly supported, test function $\varphi: \C^k \to \C$, with the convention that $\varphi(\infty)=0$;
see e.g. \cite{ AGZ,HKPV1}.  This definition of $\rho^{(k)}$ is clearly independent of the choice of ordering $\zeta_1,\ldots,\zeta_n$ of the zeroes. 
Note that if the random polynomial $P$ has a discrete law rather than a continuous one, then $\rho^{(k)}_{P}dz_1\dots dz_k$ needs to be interpreted as a general
measure.  

\begin{remark} When $\xi$ has a continuous complex distribution and when the coefficients $c_i$ are non-zero, then the zeroes are almost surely simple.  In this case if  $z_1,\ldots,z_k$ are distinct fixed complex numbers then one can interpret $\rho^{(k)}(z_1,\ldots,z_k)$ as the unique quantity such that the following holds: the probability that there is a zero in each of the disks $B(z_i,\eps)$ for $i=1,\ldots,k$ is $(\pi \eps^2)^k(\rho^{(k)}(z_1,\ldots,z_k)+o(1)) $ in the limit $\eps \to 0$.
\end{remark}

When the random polynomial $P$ have real coefficients,  the zeroes $\zeta_1,\ldots,\zeta_n$ are symmetric with respect to the real axis, and one expects several of the zeroes to lie on this axis.  Because of this possibility, the situation is more complicated.
 It is no longer  natural to work with the complex $k$-point correlation functions $\rho^{(k)}_P$, as they are likely to become singular on the real axis. Instead, we divide the complex plane $\C$ into three pieces $\C = \R \cup \C_+ \cup \C_-$, with $\C_+ := \{ z \in \C: \Im(z)>0\}$ being the upper half-plane and $\C_- := \{z \in \C: \Im(z)<0\}$ being the lower half-plane.  By the aforementioned symmetry, we may restrict our attention to the zeroes in $\R$ and $\C_+$ only.  For any natural numbers $k,l \geq 0$, we define the \emph{mixed $(k,l)$-correlation function} $\rho^{(k,l)} = \rho^{(k,l)}_{P}: \R^k \times (\C_+ \cup \C_-)^l \to \R^+$ of a random polynomial $P$ to be the function defined by the formula

\begin{eqnarray} \label{cord1}
& \E \sum_{i_1,\ldots,i_k \hbox{ distinct}} \sum_{j_1,\ldots,j_l \hbox{ distinct}} \varphi(\zeta_{i_1,\R},\ldots,\zeta_{i_k,\R},\zeta_{j_1,\C_+},\ldots,\zeta_{j_l,\C_+})  \\  \nonumber
&\quad = \int_{\R^k} \int_{\C_+^l} \varphi(x_1,\ldots,x_k,z_1,\ldots,z_l) \rho^{(k,l)}_{P}(x_1,\ldots,x_k,z_1, \ldots,z_l)\ dz_1\ldots dz_l dx_1 \ldots dx_k 
\end{eqnarray}	
for any continuous compactly supported test function $\varphi: \R^k \times \C^l \to \C$ (note that we do not require $\varphi$ to vanish at the boundary of $\C_+^l$), $\zeta_{i,\R}$ runs over an arbitrary enumeration of the real zeroes of $P_n$, and $\zeta_{j,\C_+}$ runs over an arbitrary enumeration of the zeroes of $P_n$ in $\C_+$.  This defines $\rho^{(k,l)}$ (in the sense of distributions, at least) for $x_1,\ldots,x_k \in \R$ and $z_1,\ldots,z_l \in \C_+$; we then extend $\rho^{(k,l)}(x_1,\ldots,x_k,z_1,\ldots,z_l)$ to all other values of $x_1,\ldots,x_k \in \R$ and $z_1,\ldots,z_l \in \C_+ \cup \C_-$ by requiring that $\rho^{(k,l)}$ is symmetric with respect to conjugation of any or all of the $z_1,\ldots,z_l$ parameters.  Again, we permit $\rho^{(k,l)}$ to be a measure\footnote{As in the complex case, we allow the real zeros $\zeta_{i_1,\R},\ldots,\zeta_{i_k,\R}$ or the complex zeroes $\zeta_{j_1,\C_+},\ldots,\zeta_{j_l,\C_+}$ to have multiplicity; it is only the indices $i_1,\ldots,i_k,j_1,\ldots,j_l$ that are required to be distinct.  In particular, in the discrete case it is possible for $\rho^{(0,2)}(z_1,z_2)$ (say) to have non-zero mass on the diagonal $z_1=z_2$ or the conjugate diagonal $z_1 = \overline{z_2}$, if $P$ has a repeated complex eigenvalue with positive probability.} instead of a function when the random polynomial $P_n$ has a discrete distribution.

In the case $l=0$, the correlation functions $\rho^{(k,0)}$ for $k \geq 1$ provide (in principle, at least) all the essential information about the distribution of the real zeroes. For instance, 
\begin{equation}\label{enr}
 \E N_P(\R) = \int_\R \rho^{(1,0)}(x)\ dx
\end{equation}
and similarly
\begin{equation}\label{varr}
 \Var N_P(\R) = \int_\R \int_\R \rho^{(2,0)}(x,y) -\rho^{(1,0)}(x) \rho^{(1,0)}(y) \ dx \ dy  + \int_\R \rho^{(1,0)}(x)\ dx
\end{equation}

We refer the reader to \cite{AGZ, HKPV1} for a thorough  discussion of correlation functions.

\subsection{Universality}  

The correlation functions give us a lot of information at  finer scales. Given the story of 
real roots in the previous section (which corresponds to the special case  \eqref{enr}), 
it is natural to expect  that their computation is extremely hard. 

The situation is roughly as follows. 
 We have an explicit  formula (Kac-Rice formula) to compute correlation functions. This formula is a generalization of Kac's formula in the previous section
 and involves joint distributions.  In principle, one can evaluate it in the Gaussian case (as Kac did). But technically,  for various sets of  coefficients $c_i$, this is already a significant challenge. 

In \cite{NS2}, Nazarov and Sodin considered the random series $f(z)=\sum_{j=0}^\infty \frac 1{\sqrt {j!}}\xi_j z^j$ where $\xi_j$ are iid normalized complex Gaussian and used the Kac-Rice formula to prove repulsion properties of its complex zeros, more specifically they proved that the $k$-correlation function is locally comparable to the square modulus of the complex Vandermonde product
$$C^{-1} \prod_{i<j} |z_i-z_j|^2 \le \rho^{(k)}_f(z_1,\dots, z_k) \le C \prod_{i<j} |z_i-z_j|^2 \ \ .$$The method of \cite{NS2} extends to more general settings. In \cite{NS2} the authors proved the same type of estimates for the complex $k$-point correlation function of the so-called $2k$-nondegenerate Gaussian analytic functions, which include (among others) $P(z)=\sum_{j=0}^\infty c_j \xi_j z^j$ with  $c_0,\dots, c_{2k-1}\ne 0$ such that $\sum_j |c_j|^2 |z|^{2j}$ converges in the domain where estimates for $\rho_f^{(k)}$ are needed (see \cite{NS2} for technical details). These certainly include random polynomials of finite degrees (at least $2k-1$) whose first $2k$ coefficients are nonzero, however the implicit constants $C$ in the estimates depend also on $f$ (and $k$ and the domain) and thus could be a very large function of the degree.

Similar to the Kac formula, a direct evaluation of the Kac-Rice formula is not feasible when $\xi$ is a  general non-Gaussian random variable. On the other hand, it has been conjectured that 
 the value of the formula, at least in the asymptotic sense, 
 should not depend on the fine details of the atom variable $\xi$. This is commonly referred to in the literature as the   \emph{universality phenomenon}. 

  Bleher and Di proved universality  for  
  elliptic polynomials in which the atom distribution $\xi$ was real-valued and sufficiently smooth and rapidly decaying (see \cite[Theorem 7.2]{BD} for the precise technical conditions and statement).
    With these hypotheses, they showed that the pointwise limit of the normalized correlation function $n^{-k/2} \rho^{(k,0)}(a+\frac{x_1}{\sqrt{n}},\ldots,a+\frac{x_k}{\sqrt{n}})$ for any fixed $k,a,x_1,\ldots,x_k$ (with $a \neq 0$) as $n \to \infty$ was independent of the choice of $\xi$ (with an explicit formula for the limiting distribution).  Their method is based on the Kac-Rice formula.

In a recent paper \cite{TVpoly}, Tao and the third author  introduced a  new method to prove universality, which  we will refer to as ``universality by sampling" (see Section \ref{section:ideas}).   
This method makes no distinction between continuous and discrete random variables and  the authors used it 
to derive  universality for flat, elliptic and Kac polynomials in certain domains. 

\begin{definition}\label{matching}
Two complex random variables $\xi$ and $\xi'$ are said to \textit{match moments to order $m$} if 
\begin{equation}
\E \Re(\xi)^{a}\Im(\xi)^{b} = \E \Re(\xi')^{a}\Im(\xi')^{b}\nonumber
\end{equation}
for all natural numbers $a, b\ge 0$ with $a+b\le m$.

\end{definition}


\subsection{Coefficients with polynomial growth}

We consider random polynomials  
\begin{equation}
P_{n}(z) = \sum_{i=0}^{n}c_i\xi_iz^i,\qquad z\in \C,\label{poly}
\end{equation}

\noindent with the following condition

{\bf Condition 1.} 

\begin{enumerate}\label{assumption1} 

\item \label{cond1i}$\xi_i$'s are independent  (real or complex) random variables with unit 
variance and  bounded $2+\epsilon$ moment, namely $\E|\xi_i|^{2+\epsilon} \le \tau_2$ for an arbitrarily small positive  constant $\epsilon$.

\item \label{cond1ii}$c_i$'s are deterministic complex numbers with  \begin{equation}\label{cond-c}
\tau_1 i^{\rho}\le |c_i|\le \tau_2 i^{\rho} \qquad\text{for all } i\ge N_0,
\,\,{\rm and} \,\,\,  |c_i|\le \tau_2\qquad\text{for all }  0\le i<N_0,
 \end{equation} 

 where $N_0, \tau_1, \tau_2, \epsilon$ are positive constants and $\rho > -1/2$.
 \end{enumerate}

Notice that we do not require the $\xi_i$ to be identically distributed. They are also allowed to have 
different means. However, by H\H{o}lder's inequality, our condition on the uniform boundedness of $2+\ep$ moments implies that the means should be bounded $\E|\xi_i|\le \tau_2^{1/(2+\ep)}$ for all $i$. 

An  essential point here is that we do not need to know the values of the  coefficients $c_i$ precisely, only their growth. 
We do not know of any result which is  applicable at this level of generality. 

In the next two subsections, we state our universality theorems for complex and mixed correlation functions. 


\subsection{Complex local universality for polynomials}

For a polynomial $P = P_n$ of the form \eqref{poly}, let $(\zeta_i^{P})_{i=1}^{n}$ be the zeros of $P$. 
 We  use the convention that if $P_n$ vanishes identically then it has a zero of order $n$ at $\infty$, and similarly, if $P_n$ has degree $m < n $ then it has a zero of order $n-m$ at $\infty$.

Let $\delta$ be a small positive number. 
Define 

\begin{center}
$I(\delta) = \begin{cases} [1 - 2\delta, 1-\delta] &\mbox{if } \frac{1}{10n}\le \delta<1, \\ 
[1 - 1/n, 1+1/n] & \mbox{if } 0<\delta< \frac{1}{10n}, \end{cases}.$
\end{center}

and
\begin{equation}
J(\delta) =  \left [\frac{1}{1-\delta}, \frac{1}{1-2\delta}\right ] \quad\mbox{ if $\frac{1}{10n}\le \delta<1$}.
\end{equation}

Note that $\left (\bigcup_{\frac{1}{20n}\le \delta\le \frac{1}{C}} I(\delta)\right )\cup \left (\bigcup_{\frac{1}{10n}\le \delta\le \frac{1}{C}} J(\delta)\right ) = \left [1-\frac{2}{C},\frac{C}{C-2}\right ]\supset\left [1-\frac{1}{C},1+\frac{1}{C}\right ]$.

Our goal is to prove universality in the annulus $A\left (0, 1-\frac{1}{C},1+\frac{1}{C}\right )$ for some large constant $C$ and we shall break it into annuli with radii given by  $I(\delta)$ and $J(\delta)$. When proving universality on the annulus $\{z: |z|\in I(\delta)\}$, for convenience of notations we will consider the following rescaled version
\begin{equation}
\check{P}(\check{z}) = P(z), \text{where } \check{z} = \frac{z}{10^{-3}\delta}.\label{rescale}
\end{equation}
The term ``local universality" can be thought of as universality on balls that contain $\Theta(1)$ zeros on average. It is more or less proven throughout the paper that for such $z$ as above, there are an average of $\Theta(1)$ zeros in the ball $B(z, 10^{-3}\delta)$. The rescaled factor $10^{-3}\delta$ in \eqref{rescale} plays the simple role of making this ball have the unit radius. The factor $10^{-3}$ is artificial and can be replaced by any sufficient small constant that allows the ball to grow under various approximation steps in our proofs while keeping distance $\Theta(\delta)$ away from the unit circle. Observe that by the change of variables formula, we have
\begin{equation}
\rho_{\check P}^{(k)}(w_1, \dots, w_k) = (10^{-3}\delta)^{2k}\rho_P^{(k)}(10^{-3}\delta w_1, \dots, 10^{-3}\delta w_k). \nonumber
\end{equation}

When working with the annulus $\{z: |z|\in J(\delta)\}$, we first consider $Q(z) = \frac{z^{n}}{c_n}P\left(\frac{1}{z}\right )$ to transform the domain $|z|\ge 1$ into $|z|\le 1$, and in particular, $J(\delta)$ into $I(\delta)$.  Note that $Q=\sum_{i=0}^{n}\frac{d_i}{d_0}\xi_{n-i}z^{i}$ where $d_i = c_{n-i}$. For notational convenience, sometimes we also think about $Q$ as $Q = \sum_{i=0}^{n}\frac{d_i}{d_0}\xi_{i}z^{i}$. And then, we use the same rescaling
\[\check{Q}(\check{z}) = Q(z), \text{where } \check{z} = \frac{z}{10^{-3}\delta}.
\]

Let $\rho^{(k)}_{\check P}$ and $\rho^{(k, l)}_{\check P}$ be the corresponding correlation functions of $\check P$.
Note that they depend on $\delta$ because the rescaling factor does.

Here, for a function $F:\C ^m\to \C$, we think of it as a function from $\R^{2m}\to \C$ and denote by $\left |{\triangledown^a F(x)}\right |$ the Euclidean norm of $\nabla^a F(x)$
\[\left |{\triangledown^a F(x)}\right | = \left (\sum_{1\le i_1, \dots, i_a\le 2m}\left |\frac{\partial ^{a}F}{\partial x_{i_1}\dots \partial x_{i_a}} (x)\right |^{2}\right )^{1/2}.
\]

\begin{theorem}\label{complex} Let $k\ge 1$ be an integer constant. 
Let $P_n = \sum_{i=0}^{n} c_i\xi_iz^{i}$ and $\tilde P_n = \sum_{i=0}^{n} c_i\tilde \xi_i z^{i}$ be two random polynomials satisfying Condition 1. Assume that $\xi_i$ and $\tilde \xi_i$ match moments to second order for all $N_0\le i\le n$ where $N_0$ is the constant in Condition 1. 

Then there exist constants $C, C', c$ depending only on $k$ and the constants in Condition 1 such that for every $\frac{1}{20n}\le\delta\le \frac{1}{C}$ and complex numbers $z_1, \dots, z_k$ with $|z_j|\in I(\delta)$ for all $0\le j\le k$, and for every smooth function $G: \mathbb{C}^{k}\to \mathbb{C}$ supported on $B(0, 10^{-3})^{k}$ with $|{\triangledown^aG}(z)|\le 1$ for all $0\le a\le 2k+4$ and $z\in \C^{k}$, we have
\begin{eqnarray}\label{h1}
\bigg|&&\int_{\mathbb{C}^{k}}G(w_1,\dots, w_k)\rho_{\check P}^{(k)}(\check z_1+ w_1,\dots, \check z_k+ w_k)\text{d}w_1\dots\text{d}w_k\nonumber\\
&&-\int_{\mathbb{C}^{k}}G(w_1,\dots, w_k)\rho_{\check {\tilde P}}^{(k)}(\check z_1+ w_1,\dots, \check z_k+ w_k)\text{d}w_1\dots\text{d}w_k\bigg|\le C'\delta^c.
\end{eqnarray}
Furthermore, if $\frac{1}{10n}\le\delta\le\frac{1}{C}$, 
\begin{eqnarray}\label{h1-1}
\bigg|&&\int_{\mathbb{C}^{k}}G(w_1,\dots, w_k)\rho_{\check Q}^{(k)}(\check z_1+ w_1,\dots, \check z_k+ w_k)\text{d}w_1\dots\text{d}w_k\nonumber\\
&&-\int_{\mathbb{C}^{k}}G(w_1,\dots, w_k)\rho_{\check {\tilde Q}}^{(k)}(\check z_1+ w_1,\dots, \check z_k+ w_k)\text{d}w_1\dots\text{d}w_k\bigg|\le C'\delta^c.
\end{eqnarray}
\end{theorem}


\subsection{Real local universality}\label{s.real}

For real universality, we require the following additional condition on $\xi_i$'s and $c_i$'s

{\bf Condition 2.} 
\begin{enumerate}
\item \label{cond-real} The random variables $\xi_i$'s and the coefficients $c_i$'s are real.
\item \label{cond-real'} One of the following holds
\begin{enumerate}
 \item \label{cond-real'-1} $\E\xi_i = 0$ for all $i\ge N_0$,

\item \label{cond-real'-2} $\E \xi_i = \mu$ for all $i\ge N_0$, where $\mu$ is any constant, and there exists a classical polynomial $\mathfrak P$ (independent of $n$) with degree $\rho\in \mathbb N$ such that $c_i = \mathfrak P(i)$ for all $i\ge N_0$. \footnote{For instance, $P$ is Kac polynomial or its derivatives.}  
\end{enumerate}

\end{enumerate}

Notice that when Condition 2 \eqref{cond-real'-2} is satisfied, by replacing $c_i$ by $-c_i$ if needed, we can also assume that $c_i = \mathfrak P(i)>0$ for all $i$ larger than some constant because the (fixed) polynomial $\mathfrak P(x)$ keeps the same sign when $x$ is sufficiently large.

\begin{theorem}\label{real} Let $k, l\ge 0$ be integer constants with $k+l\ge 1$. Let $P_n = \sum_{i=0}^{n} c_i\xi_iz^{i}$ and $\tilde P_n = \sum_{i=0}^{n} c_i\tilde \xi_i z^{i}$ be two random polynomials satisfying Conditions 1 and 2. Assume that $\xi_i$ and $\tilde \xi_i$ match moments to second order for all $N_0\le i\le n$ where $N_0$ is the constant in Condition 1. 

Then there exist constants $C, c$ depending only on $k, l$ and the constants and the polynomial $\mathfrak P$ in Conditions 1 and 2 such that for every $\frac{1}{20n}\le\delta\le \frac{1}{C}$, real numbers $x_1,\dots, x_k$, and complex numbers $z_1, \dots, z_l$ such that $|x_i|, |z_j|\in I(\delta)$ for all $i=1,\dots, k, j=1,\dots, l$, and for every smooth function $G: \mathbb{R}^{k}\times\mathbb{C}^{l}\to \mathbb{C}$ supported on $[-10^{-3}, 10^{-3}]^{k}\times B(0, 10^{-3})^{l}$ such that $|{\triangledown^aG}(z)|\le 1$ for all $0\le a\le 2(k+l)+4$ and $z\in \mathbb{R}^{k}\times\mathbb{C}^{l}$, we have
\begin{eqnarray}\label{h6}
\bigg|&&\int_{\mathbb{R}^{k}}\int_{\mathbb{C}^{l}}G(y_1,\dots, y_k, w_1,\dots, w_l)\nonumber\\
&&\qquad\rho_{\check P}^{(k,l)}(\check x_1+ y_1,\dots, \check x_k+ y_k,\check z_1+ w_1,\dots, \check z_l+ w_l)\text{d}y_1\dots\text{d}y_k\text{d}w_1\dots\text{d}w_l\nonumber\\
&&-\int_{\mathbb{R}^{k}}\int_{\mathbb{C}^{l}}G(y_1,\dots, y_k, w_1,\dots, w_l)\\
&&\qquad\rho_{\check {\tilde P}}^{(k,l)}(\check x_1+ y_1,\dots, \check x_k+ y_k,\check z_1+ w_1,\dots, \check z_l+ w_l)\text{d}y_1\dots\text{d}y_k\text{d}w_1\dots\text{d}w_l\bigg|\nonumber\\
&\le& C\delta^c.\nonumber
\end{eqnarray}
Furthermore, if $\frac{1}{10n}\le\delta\le\frac{1}{C}$, we have the same inequality \eqref{h6} with $Q$ in place of $P$. 
\end{theorem}

Now, in order to derive Theorems \ref{theorem:new1} and \ref{theorem:new2} from Theorem~\ref{real} it suffices to show that 
the number of real roots in the Gaussian case satisfies the claimed bounds and  that 
the expectation of real roots (in the general case) outside the universality annulus is bounded. More specifically, we will show that

\begin{lemma}\label{boundedness}
Under the conditions of Theorem \ref{real}, for each constant $C>0$, there exists a constant $M(C)$ such that
$$\E N_{P_n}\left (\R \setminus A(0, 1 - \frac{1}{C}, 1+\frac{1}{C})\right )\le M(C),$$
for every $ n\ge 1$.
\end{lemma}

which together with Theorem \ref{real} give
\begin{corollary}\label{mean}
Under conditions of  Theorem   \ref{real}, there exists a constant $C$ such that for every $n\ge 1$, one has
$$|\E N_{P_n}(\R) - \E  N_{\tilde P_n}(\R)|\le C.$$
\end{corollary}

\begin{remark}\label{rmkbdd}
To get an intuition for Lemma \ref{boundedness}, let $i$ be the smallest index for which $c_{i_0}\ne 0$. Assume that $c_{i_0} = \Omega(1)$, $\E \log|\xi_{i_0}|=O(1)$, and $\E \log|\xi_{n}|=O(1)$. Under condition \eqref{cond-c}, $i_0 = O(1)$. Then by Jensen's inequality for the function $P(z)/z^{i_0}$ and concavity of the function $\log$, one has the easy bound
\begin{eqnarray}
\E N_P(B(0,1 - 1/C))&\le& i_0 + \E\frac{\log\frac{M}{|c_{i_0}\xi_{i_0}|}}{\log\frac{1-1/2C}{1 - 1/C}}\nonumber\\
&=& i_0 + O_C(1) + O_C(\E\log M)\le O_C(1) + O_C(\log\E M)\nonumber\\
&=& O_C(1) + O_C(\log \sum_{i=0}^{\infty}i^{\rho}(1-1/2C)^{i}) = O_C(1).
\end{eqnarray} 
where $M = \max_{|z|\le 1 - 1/2C}|P(z)/z^{i_0}|$.
Similarly, $\E N_Q(B(0,1 - 1/C))=O_C(1)$. And hence, $\E N_P(\C\setminus A(0, 1-1/C, 1+1/C)) = O_C(1)$. In other words, for a large class of polynomials of the form \eqref{poly}, one expects to see only a few zeros outside the annulus of universality. 
\end{remark}
By Corollary \ref{mean}, to verify Theorem \ref{theorem:new1} and Theorem~\ref{theorem:new2}  it suffices to consider the Gaussian case.
\begin{theorem} \label{t.Gaussian}
The statement of Theorem \ref{theorem:new1}  holds for $\xi_i$ being  standard Gaussian for all $i=0, \dots, n$. 
\end{theorem} 

\begin{theorem} \label{t.Gaussian2}
The statement of Theorem \ref{theorem:new2}  holds for $\xi_i$ being  Gaussian with mean $\mu$ and variance $1$ for all $i=0, \dots, n$. 
\end{theorem}

We are going to prove these theorems in Section \ref{s.Gaussian} and Section~\ref{s.Gaussian2}. The evaluation of Kac's formula under the general setting of Theorem \ref{theorem:new1} is fairly involved, and as mentioned in the discussion leading to Corollary~\ref{c.new1}
it is somewhat surprising that the growth of the coefficients alone already determines the number of real roots.

\subsection{Local universality for series}\label{s.series}

Our method could also be used to extend the previous results to random series. 
Let us first extend Theorem \ref{complex}. 

We consider a  random series $P_{PS}$ of the form  
\begin{equation}
P_{PS}(z) = \sum_{i=0}^{\infty}c_i\xi_iz^i,\qquad z\in \D\label{series}
\end{equation}
where $\D$ is the open unit disk in the complex plane, and the $c_i$'s and $\xi_i$'s satisfy  Condition 1.


\begin{theorem}\label{complex-series} Let $k\ge 1$ be an integer constant. 
Let $P_{PS} = \sum_{i=0}^{\infty} c_i\xi_i z^{i}$ and $\tilde P_{PS} = \sum_{i=0}^{\infty}c_i\tilde \xi_i z^{i}$ be two random power series satisfying Condition 1 (with $n$ being replaced by $\infty$). Assume that $\xi_i$ and $\tilde \xi_i$ match moments to second order for all $i\ge N_0$ where $N_0$ is the constant in Condition 1. 

Then there exist constants $C, c$ depending only on $k$ and the constants in Condition 1 such that for every $0<\delta\le \frac{1}{C}$ and complex numbers $z_1, \dots, z_k$ with $|z_j|\in [1-2\delta, 1-\delta]$ for all $0\le j\le k$, and for every smooth function $G: \mathbb{C}^{k}\to \mathbb{C}$ supported on $B(0, 10^{-3})^{k}$ with $|{\triangledown^aG}(z)|\le 1, \forall 0\le a\le 2k+4$ and $z\in \C^{k}$, we have

\begin{eqnarray}\label{com}
\bigg|&&\int_{\mathbb{C}^{k}}G(w_1,\dots, w_k)\rho_{\check P_{PS}}^{(k)}(\check z_1+ w_1,\dots, \check z_k+ w_k)\text{d}w_1\dots\text{d}w_k\nonumber\\
&&-\int_{\mathbb{C}^{k}}G(w_1,\dots, w_k)\rho_{\check {\tilde {P}}_{PS}}^{(k)}(\check z_1+ w_1,\dots, \check z_k+ w_k)\text{d}w_1\dots\text{d}w_k\bigg|\le C\delta^c.
\end{eqnarray}
\end{theorem}

Notice that when all $\xi_i$ are (complex) standard Gaussian, the distribution of the zeroes is invariant with respect  to rotation. 
As a corollary of Theorem \ref{complex-series}, this invariance is preserved (in the asymptotic sense)  if $\xi_i$ matches the moments of standard Gaussian up to second order.

\begin{corollary}\label{series-rotation} Let $k\ge 1$ be an integer constant. Let $P_{PS}$ be the random series of the form \eqref{series} satisfying  Condition 1. Assume furthermore that $\E(\mbox{Re}(\xi_i))=\E(\mbox{Im}(\xi_i))=Cov(\mbox{Re}(\xi_i),\mbox{Im}(\xi_i)) = 0$ and $\Var( \mbox{Re}(\xi_i)) = \Var( \mbox{Im}(\xi_i)) =1/2$ for all $i\ge N_0$. 

Then there exist constants $C, c$ such that for every $0<\delta\le \frac{1}{C}$ and complex numbers $z_1, \dots, z_k$ with $|z_j|\in [1-2\delta, 1-\delta]$ for all $0\le j\le k$ and $0\le \theta< 2\pi$, and for every smooth function $G: \mathbb{C}^{k}\to \mathbb{C}$ supported on $B(0, 10^{-3})^{k}$ with $|{\triangledown^aG}(z)|\le 1, \forall 0\le a\le 2k+4$ and $z\in \C^{k}$, we have
\begin{eqnarray}
\bigg|&&\int_{\mathbb{C}^{k}}G(w_1,\dots, w_k)\rho_{\check P_{PS}}^{(k)}(\check z_1+ w_1,\dots, \check z_k+ w_k)\text{d}w_1\dots\text{d}w_k\nonumber\\
&&-\int_{\mathbb{C}^{k}}H(w_1,\dots, w_k)\rho_{\check { {P}}_{PS}}^{(k)}(e^{\sqrt{-1}\theta}\check z_1+ w_1,\dots, e^{\sqrt{-1}\theta}\check z_k+ w_k)\text{d}w_1\dots\text{d}w_k\bigg|\le C\delta^c,
\end{eqnarray}
where $H(w_1, \dots, w_k) = G(e^{-\sqrt{-1}\theta}w_1, \dots, e^{-\sqrt{-1}\theta}w_k)$.
\end{corollary}

In case  that $P_{PS}$ is hyperbolic and the $\xi_i$ are complex Gaussian,  the distribution of the zeros of $P_{PS}$ is  invariant 
 under hyperbolic transformations of the disk $\D$(see \cite{HKPV1}). A hyperbolic transformation on $\D$ is a transformation of the form 
\begin{equation}
\phi(z) = \frac{az+b}{\bar b z+\bar a},\nonumber
\end{equation}
where $a,b\in \C$ and $|a|^{2}-|b|^{2} = 1$. A holomorphic function on $\D$ is bijective if and only if it is a hyperbolic transformation (see, for instance, \cite[Theorems 12.4, 12.6]{Ru}).

As another immediate corollary of Theorem \ref{complex-series}, this  invariance is preserved (in the asymptotic sense) again if  $\xi_i$ matches the moments of standard Gaussian up to order 2 and if the hyperbolic transformation preserves our universality domain.

\begin{corollary}\label{series-isom} Let $k\ge 1$ be an integer constant. Let $P$ be the random \textbf{hyperbolic} series of the form \eqref{series} satisfying 
Condition 1.  Assume furthermore that $\E(\mbox{Re}(\xi_i))=\E(\mbox{Im}(\xi_i))=Cov(\mbox{Re}(\xi_i),\mbox{Im}(\xi_i)) = 0$ and $\Var( \mbox{Re}(\xi_i)) = \Var( \mbox{Im}(\xi_i))=1/2$ for all $i\ge N_0$. 

Then there exist constants $C, c$ such that the following holds true. Let $0<\delta_0\le \frac{1}{C}$ and complex numbers $z_1, \dots, z_k$ with $|z_j|\in [1-2\delta_0, 1-\delta_0]$ for all $0\le j\le k$ and $0\le \theta< 2\pi$. Let $\phi$ be a hyperbolic transformation that maps $z_j$ to $t_j$ with $|t_j|\in [1-2\delta_1, 1-\delta_1]$ for all $j$ and for some $0<\delta_1\le \frac{1}{C}$. Then for every smooth function $G: \mathbb{C}^{k}\to \mathbb{C}$ supported on $B(0, 10^{-4})^{k}$ with $|{\triangledown^aG}(z)|\le 1, \forall 0\le a\le 2k+4$ and $z\in \C^{k}$, we have

\begin{eqnarray}
\bigg|&&\int_{\mathbb{C}^{k}}G(w_1,\dots, w_k)(10^{-3}\delta_0)^{2k}\rho_{ P_{PS}}^{(k)}( z_1+ 10^{-3}\delta_0 w_1,\dots,  z_k+ 10^{-3}\delta_0 w_k)\text{d}w_1\dots\text{d}w_k\nonumber\\
&&-\int_{\mathbb{C}^{k}}H(w_1,\dots, w_k)(10^{-3}\delta_1)^{2k}\rho_{ { {P}}_{PS}}^{(k)}( {t_1}+10^{-3}\delta_1 w_1,\dots,  t_k+10^{-3}\delta_1 w_k)\text{d}w_1\dots\text{d}w_k\bigg|\nonumber\\
&\le& C\max\{\delta_0, \delta_1\}^c,\nonumber
\end{eqnarray}

where 
$$H(w_1, \dots, w_k) = G\left (\frac{1}{10^{-3}\delta_0}\left (\phi^{-1}\left (t_1 + 10^{-3}\delta_1w_1\right )-z_1\right ), \dots, \frac{1}{10^{-3}\delta_0}\left (\phi^{-1}\left (t_k + 10^{-3}\delta_1w_k\right )-z_k\right )\right ).$$
\end{corollary}

Similar to the complex case, real universality also follows from our arguments for polynomials.
\begin{theorem}
\label{real-series} Let $k, l\ge 0$ be integer constants with $k+l\ge 1$. Let $P_{PS} = \sum_{i=0}^{\infty} c_i\xi_i z^{i}$ and $\tilde P_{PS} = \sum_{i=0}^{\infty}c_i\tilde \xi_i z^{i}$ be two random power series satisfying Conditions 1 and 2(with $n$ being replaced by $\infty$). Assume that $\xi_i$ and $\tilde \xi_i$ match moments to second order for all $i\ge N_0$ where $N_0$ is the constant in Conditions 1 and 2. 

Then there exist constants $C, c$ depending only on $k, l$ and the constants and the polynomial $\mathfrak P$ in Conditions 1 and 2 such that for every $0<\delta\le \frac{1}{C}$, real numbers $x_1,\dots, x_k$, and complex numbers $z_1, \dots, z_l$ such that $|x_i|, |z_j|\in [1-2\delta, 1-\delta]$ for all $i=1,\dots, k, j=1,\dots, l$, and for every smooth function $G: \mathbb{R}^{k}\times\mathbb{C}^{l}\to \mathbb{C}$ supported on $[-10^{-3}, 10^{-3}]^{k}\times B(0, 10^{-3})^{l}$ such that $|{\triangledown^aG}(z)|\le 1, \forall 0\le a\le 2(k+l)+4$ and $z\in \mathbb{R}^{k}\times\mathbb{C}^{l}$, we have
\begin{eqnarray}\label{h6-series}
\bigg|&&\int_{\mathbb{R}^{k}}\int_{\mathbb{C}^{l}}G(y_1,\dots, y_k, w_1,\dots, w_l)\nonumber\\
&&\qquad\rho_{\check P_{PS}}^{(k,l)}(\check x_1+ y_1,\dots, \check x_k+ y_k,\check z_1+ w_1,\dots, \check z_l+ w_l)\text{d}y_1\dots\text{d}y_k\text{d}w_1\dots\text{d}w_l\nonumber\\
&&-\int_{\mathbb{R}^{k}}\int_{\mathbb{C}^{l}}G(y_1,\dots, y_k, w_1,\dots, w_l)\\
&&\qquad\rho_{\check {\tilde P}_{PS}}^{(k,l)}(\check x_1+ y_1,\dots, \check x_k+ y_k,\check z_1+ w_1,\dots, \check z_l+ w_l)\text{d}y_1\dots\text{d}y_k\text{d}w_1\dots\text{d}w_l\bigg|\le C\delta^c.\nonumber
\end{eqnarray}
\end{theorem}
We will prove these results in Section \ref{proof-complex-series}.

  \section{Sketch of the proof and the main technical ideas}  \label{section:ideas}

To start, we make use of   the ``universality by sampling''  method from \cite{TVpoly}, which is based on the Lindeberg swapping technique. To give the reader a quick introduction on  this method, let us discuss the simplest correlation function
$\rho^{(0,1)} $, which is the density function of the complex roots.  Consider two polynomials $P_{n, \xi}$ and $P_{n, \tilde \xi}$ and a (nice)  test function $G(x)$.  We would like  to show

$$ \int_{\C}  G(x) \rho_{P_{n, \xi} } ^{(0,1)} (x) dx =  \int_{\C}  G(x) \rho_{P_{n, \tilde \xi} } ^{(0,1)} (x) dx  +o(1). $$

\noindent Recall that by definition  

$$  \int_{\C}  G(x) \rho_{P_{n, \xi} } ^{(0,1)} (x) dx= \sum_{i=1}^n  \E_{\xi}  G ( \zeta_i ), $$  and 

$$  \int_{\C}  G(x) \rho_{P_{n, \tilde  \xi} } ^{(0,1)} (x) dx= \sum_{i=1}^n  \E_{\tilde  \xi}  G (\tilde  \zeta_i ), $$  where $\zeta_i$ ($\tilde \zeta_i $) are the roots of $P_{n, \xi}$ ($P_{n, \tilde \xi}$). 

\noindent  We are going to prove universality of the right-hand side, namely 

$$\sum_{i=1}^n \E_{\xi}  G ( \zeta_i )=  \sum_{i=1}^n \E_{\tilde \xi}  G (\tilde \zeta_i ) +o(1) . $$

\noindent Our starting point is Green's formula, which asserts that 

$$\log G(0) = -\frac{1}{2\pi}\int_{\C} \log |z| \Delta G(z) dz. $$

\noindent By change of variables, this implies that for all $i$,  

$$\log G(\zeta_i) = -\frac{1}{2\pi}\int_{\C} \log |z -\zeta_i | \Delta G(z) dz, $$ which, in turn, yields 

$$\sum_i \E_{\xi}  G ( \zeta_i ) = -\frac{1}{2\pi}\E_{\xi}  \int_{\C} \log | \prod_{i=1}^n (z -\zeta_i ) | \Delta G(z) dz = -\frac{1}{2\pi}\E_{\xi}  \int_{\C} \log | P_{n, \xi} (z)  | \Delta G(z) dz. $$

(Notice that the leading coefficient of $P$ does not matter here, as $ \int_{\C} \Delta G(z) dz =0$.) 
We estimate the integration  $ \int_{\C} \log | P_{n, \xi} (z)  | \Delta G(z) dz$ by {\it sampling}. The intuition is that if $S$ is the average  of (say) $N$ numbers
$S:= \frac{a_1+ \dots +a_N}{N}$ where $N$ is large integer, then (hopefully) we can estimate $S$  accurately by a much shorter random partial sum $S' =\frac{ a_{i_1} + \dots + a_{i_m}} {m} $, where 
the indices $i_1, \dots, t_m$ are chosen randomly from the index set $\{1, \dots, N \}$, with  $m$ being  a parameter much smaller  than $N$. Thinking of $a_1, \dots, a_N$ as terms in the Riemann sum approximation of  $\int_{\C} \log | P_{n, \xi} (z)  | \Delta G(z) dz$, we want to  approximate this integral by 
$$\frac{1}{m} ( H_{\xi} (z_1) + \dots + H_{\xi} (z_m )) , $$ where $H_{\xi} (z):= C  \log | P_{n, \xi} (z)  | \Delta G(z) $ with $C$ being a normalization constant, $z_1, \dots, z_m$ are random sample points, and $m$ is a properly chosen parameter which tends to infinity slowly with  $n$.  (The magnitude of $m$ determines the quality of the approximation.) 

Now assume, for a moment, that $\frac{1}{m} ( H_{\xi} (z_1) + \dots + H_{\xi} (z_m ))$ is indeed a good approximation of  $\int_{\C} \log | P_{n, \xi} (z)  | \Delta G(z) dz$  and similarly, $\frac{1}{m}(  H_{\tilde \xi} (z_1) + \dots + H_{\tilde \xi} (z_m )$ is  a good approximation of  $\int_{\C} \log | P_{n, \xi} (z)  | \Delta G(z) dz$, with overwhelming probability.  In this case, the  problem  reduces to showing 
$$ \E_{\xi} \frac{1}{m} (  H_{\xi} (z_1) + \dots + H_{\xi} (z_m )) =  \E_{\tilde \xi} \frac{1}{m}(  H_{\tilde \xi} (z_1) + \dots + H_{\tilde \xi} (z_m ) ) +o(1) . $$

We can apply the Lindeberg swapping  method to prove this estimate. In fact, we can use this method  to show that the joint distribution of  $m$ variables  $H_{\xi} (z_1), \dots,  H_{\xi} (z_m )$
and that of \newline $H_{\tilde \xi} (z_1), \dots ,H_{\tilde \xi} (z_m ) $ are approximately the same.   This can be done by defining\newline $Z:= (H_{\xi} (z_1), \dots,  H_{\xi} (z_m ))$ 
and showing 

\begin{equation} \label{exp0} \E_{\xi}  F (Z) = \E_{\tilde \xi}  F(\tilde Z) +o(1) \end{equation} for any nice test function $F$.

An application of the Linderberg method often requires estimates on the derivatives of the function in question, and a decisive advantage  here is that the function $H$ is explicit, and it is not too hard to 
bound its derivatives.  Generalizing the whole scheme to the general case of $\rho^{k,l}$ requires several additional technical steps, but the spirit of the method remains the same.


The   critical point of this scheme is to show that   the random sum  indeed approximates the integral.  In order to do so, we need to bound from above  the second moment 
$$ \int_{\C} | \log | P_{n, \xi} (z)  | \Delta G(z)  |^2  dz  =   \int_{D} | \log | P_{n, \xi} (z)  | \Delta G(z) | ^2  dz, $$ where $D$ is the support of $G$; see Lemma \ref{f5}.

Our strategy has two steps. We first define a {\it good} event $\CT$ (which holds with high probability) in the space generated by the $\xi_i$. 
Among others, this event guarantees that the number of roots in  $D$ is at most $n^c$, where $c$ is a sufficiently small positive constant. 
($D$ was actually chosen so that the expectation of the number of roots in $D$ is $O(1)$.)  When $\CT$ holds, we split 
$P=  R Q $, where $ R  : =\prod _{\zeta_i \in D} (z -\zeta_i )  $ and $Q :=  \prod _{\zeta_i \notin D} (z -\zeta_i )  $. Then

$$   \int_{D} |\log | P_{n, \xi} (z)  | \Delta G(z) | ^2  dz \le  2 ( \int_{D} | \log | R_{n, \xi} (z)  | \Delta G(z) | ^2  dz +  \int_{D} | \log | Q_{n, \xi} (z)  | \Delta G(z) | ^2  dz). $$

The first integral on the RHS is easy to bound, as the number of roots in $R$ is small, and $\log |R|$ can be split into sum of few terms. 
  To bound the second one, we show that  $| \log | Q_{n, \xi} (z)  | \Delta G(z | $ is small for {\it every point } in $D$. Typically, in order to prove that an event $\CE (z) $ holds for every point  $z$ in some domain $D$
one makes use of the $\epsilon$-net argument.  We   put an $\epsilon$-net on $D$ and prove that $\CE (z)$ holds for all points in the net, and then use some analytic argument to extend the net to the whole domain.
  If the net has size $N$, then 
   by the union bound, we need to show that for each $z$ in the net $\P( \CE(z) \,\,{\rm holds} ) \ge 1 -o(1/N)$.  The proof of this usually requires sophisticated 
 anti-concentration inequalities; furthermore, sometimes the bound itself is not true (which does not contradict  the correctness of the final statement we want to prove). 
 In our situation, we make a novel use of  Harnack's inequality, which allows us to  reduce the statement to one point, instead of to the whole $\epsilon$-net,  which 
 completely avoids the use of union bound argument.  This way, we obtain a sufficiently strong bound on the second moment so that the sampling procedure goes through. See Section \ref{bounded-variance} for more details.

The trickier part is when $\CT$ does not hold.  In this case, it is possible that sampling does not provide a good approximation.  We are going to avoid this problem by directly showing that 
the contribution coming from the complement $\CT^c$  of $\CT$ towards the expectations in \eqref{exp0} is small, namely
$ \E_{\xi}  F (Z)  \BI_{\CT^c}   =o(1)$  (and the same for the $\tilde \xi$ version).

The main difficulty here is that the logarithm function has a pole at zero.   If  $|P_{n, \xi}(z) |$ is very close to zero in some region, then the value of $\log | P_{n, \xi} (z)  |$ could be very large. 
(Another type of danger is that $|P_{n, \xi}(z)|$ is large, but this is easy to deal with, even by elementary method such as the moment method.)
To overcome this problem, one needs to 
 show that with high probability, $|P_{n, \xi}(z)|$ is bounded away from 0. Technically speaking, we need to show  
 $$\P ( | c_n \xi_n  z^n + \dots + c_0 \xi_0 | \le \epsilon (n) )  $$  is sufficiently small, for most value of $z$  and a properly chosen parameter $\epsilon (n)$. 
 This type of estimates is called anti-concentration (or small ball) 
 inequality  in the literature; see \cite{NVsmallball} for an introduction. 
This part is the most delicate part of our proof, and unlike prior works (see e.g. \cite{TVpoly}), our method could treat the general set of coefficients considered here.

  In this paper, we introduce a completely different way to obtain the desired anti-concentration bound, which makes use of  
  of various \emph{a priori}  etimates for $P_{n, \xi}$ and a recent powerful result of Nazarov--Nishry--Sodin \cite{NNS} about the log-integrability of random Rademacher series.  As a matter of fact, Nazarov et al. result only holds 
 for random Rademacher variables (and  may fail for others). We use a couple of symmetrization arguments to handle the general case. See Section \ref{proof-complex} and in particular, \ref{non-clustering} and \ref{negligible-set} for details.

By completing the above scheme, we obtain universality results for the complex roots. 
 The handling of real roots also requires extra care. In order to prove the universality of the correlation functions among real roots (including the universality  of the density function which yields new results on the expectation discussed in the introduction) we need to show that there is no complex root near the real line, with high probability.  This, at the intuition level at least, would allow us to translate results for complex roots  near the real line to results for real roots, as once a root is sufficiently near the real line it has to be real.

 One way to obtain this is 
via the  so-called weak level repulsion property, relying on explicit estimates of the Kac-Rice formula  for Kac polynomials with gaussian coefficients. However, it is 
very difficult, if not impossible, to obtain similar  estimates for the general polynomials considered in this paper,
 particularly in the case when the means of the coefficients  are nonzero. 
We handle  this problem by  a novel  argument,  based on  Rouch\'e's theorem  following an  ideas from a paper of Peres and Virag \cite{PV1} 
and the monograph  by Hough et al. \cite{HKPV1}. Apparently, the repulsion property  is interesting on its own right, and there is a chance that the argument can be 
applied for other settings.

To illustrate the idea, let us consider a disk $B(x_0, r) $ center at a point $x_0$ on the real line. We want to show that if $r$ is sufficiently small, then  with high probability $B(x_0,r)$ contains at most one root. This excludes the complex roots as they come in conjugated pairs. 
Define $g(z)= P_{n, \xi }(x_0)  + (z-x_0) P' _{n, \xi} (x_0) $.  By Rouch\'e's theorem, if we can show that (with high probability), $| P_{n,\xi}(z)- g(z) | < |g(z) |$ for all $z$ on the boundary of $B(x_0, r)$, then $P_{n,\xi}(z)$ and $g(z)$ have the same number of roots inside the disk. Note that 
$g(z)$ is linear, so it has at most one root. The verification of  $| P_{n,\xi}(z)- g(z) | < |g(z) |$  makes use of the Cauchy's integral formula and an anti-concentration result. (One can also use an $\epsilon$-net argument here, but the details are more involved.) See Section \ref{proof-real} for details.

Finally, let us discuss the treatment of polynomials with Gaussian coefficients. The strategy of the proof of Theorem \ref{theorem:new1} (and other results in the Introduction) is to reduce 
to the Gaussian case, using universality results. In fact, the Gaussian setting of Theorem~\ref{theorem:new1} and Theorem~\ref{theorem:new2} are already substantially new, and furthermore  our method of proof is novel compared to previous works. For example, the only case we know where the optimal error term $O(1)$ in our results was obtained is Kac polynomials, thanks to the very explicit formula \eqref{Kacformula2}.  In our general setting while some version of \eqref{Kacformula2} is available, evaluation of such formula turns out to be fairly delicate: in many other previous works for the mean-zero coefficients setting \cite{D1, D2, sambandham1979, FGK}  (see also \cite{BS1}), researchers used the method of Logan and Shepp \cite{LS1, LS2}, but this could not lead to the error term $O(1)$, and for coefficients with nonzero means (see below), the analysis from Farahmand's and Ibragimov-Maslova's paper  \cite{IM3, F1} do  not lead to the error bound $O(1)$, even for the Kac polynomial. 
 Finally, none of the above mentioned analysis can be reproduced to yield an asymptotic result for our general setting, where only the order of magnitude of the coefficients  $c_i$ is known. 
 
Now, let us discuss briefly the main new ideas in our the treatment of the Gaussian case. Via the Kac-Rice formula, the analysis of the nonzero mean case relies on several key estimates from the zero mean setting. In the zero mean case, our  new idea is to develop a reformulation of the Edelman--Kostlan formula for the density function of the distribution of real zeros \cite{EK1}, so that the density could be computed  using \emph{only} the  variance function $\Var[P_{n,Gauss}]$ and its first few derivatives.  This enables us to reduce the analysis of the density function to a careful study of the large $n$ asymptotics of   $\Var[P_{n,Gauss}]$ and its derivatives.  This novel approach  allows us to get the $O(1)$ estimate for the error terms, which can not be obtained using the Logan-Shepp methods. The analysis of the large $n$ behavior of $\Var[P_{n,Gauss}]$ and its derivatives involves fairly technical estimates and occupies the last few sections of the paper.  Unlike the Kac polynomials, in our setting the distribution of the real zeros is not invariant under the map $x\mapsto 1/x$, leading to extra difficulty in the analysis.


\section{Proof of complex local universality for polynomials}\label{proof-complex}

Throughout the paper, $L := \frac{1}{\delta}$. 

In this section, we prove Theorem \ref{complex}. In particular, we will prove \eqref{h1}. The same proof works for \eqref{h1-1} by replacing $P$ by $Q$, unless otherwise noted. Notice that we only consider $Q$ when talking about $\delta\ge \frac{1}{10n}$.

We can assume without loss of generality that $\tilde \xi_i$ has Gaussian distribution for all $i$.

By standard arguments using the Fourier analysis, using the assumption that the test function $G$ is sufficiently smooth, one gets that $G$ equals its Fourier series on its support with the Fourier coefficients growing sufficiently slowly. Therefore, if the desired statement is proven for each term (which is smoothly truncated on the support of $G$) in the Fourier expansion, it extends automatically to $G$. In other words, the problem reduces to proving (\ref{h1}) for 
\begin{equation}\label{h2}
 G(w_1,\dots, w_m) = G_1(w_1)\dots G_k(w_k)
 \end{equation} where for each $1\le i\le k$, $G_i:\mathbb{C}\to \mathbb{C}$ is a smooth function supported in $B(0, 10^{-2})$ and $|{\triangledown^aG_{i}}|\le 1$ for all $0\le a\le 3$.

When $G$ is of that form, we have 
\begin{eqnarray}
&&\int_{\mathbb{C}^{k}}G(w_1,\dots, w_k)\rho_{\check P}^{(k)}(\check z_1+ w_1,\dots, \check z_k+ w_k)\text{d}w_1\dots\text{d}w_k\nonumber\\
& =& \E \sum_{i_1,\dots, i_k \text{distinct}}G_1({\zeta}_{i_1}^{\check P}-\check{z_1})\dots G_k({\zeta}_{i_k}^{\check P}-\check{z_k}).
\end{eqnarray}
Let $r_0 = 10^{-2}$.
By the inclusion-exclusion formula, we then can rewrite the later expression as 
\begin{equation}\label{h3}
 \E \prod_{j=1}^{k}X_{j}^{P}
 \end{equation} plus a bounded number of lower order terms which are of the form (\ref{h3}) for smaller values of $k$, where 
\begin{eqnarray}
X_j^{P} = X_{\check{z}_j, G_j}^{P} = \sum_{i=1}^{n} G_j({\zeta}_i^{\check P}-\check z_j).\label{moon1}
\end{eqnarray} 
Hence, by induction on $k$, it suffices to show that 
\begin{eqnarray}
\ab{\E \prod_{j=1}^{k}X_j^{P}-\E \prod_{j=1}^{k} X_j^{\tilde P}}\le C\delta^{c}.\label{du5}
\end{eqnarray}

If $P$ does not vanish on the support of $H_j$, then by the Green formula we have
\begin{equation}
X_j^{P} =\sum_{i=1}^{n} G_j({\zeta}_i^{\check P}-\check z_j)= \int_{\mathbb C}\log |\check P(z)|H_j(z)dz  = \int_{B(\check z_j, r_0)}\log |\check P(z)|H_j(z)dz,\label{sat1}
\end{equation}
where $H_j(z) = -\frac{1}{2\pi}\triangle G_j(z-\check z_j)$. Note that $supp(H_j)\subset B(\check z_j, r_0)$.

%
%

Let $K_j^{P} = \log |\check P(z)|H_j(z)$. Let $c_1$ be a small positive constant to be chosen later. Let $\mathcal T = \mathcal T(\delta)$ be the event on which
\begin{enumerate}[(i)]
\item \label{Ti}$P \not\equiv 0$.
\item \label{Tii}$N_P \left (B\left ( z_j, \frac{\delta}{10}\right ) \right )\le  L^{c_1}$ for all $1\le j\le k$.
\item \label{Tiii}$\log|P(z_j)|\ge -\frac{1}{2} L^{c_1}$ for all $1\le j\le k$.
\item \label{Tiv}$\log|P(z)|\le \frac{1}{2}L^{c_1}$ for all $z$ such that $|z|\in I(\delta) + (-\delta/2, \delta/2)$.
\end{enumerate}
And if $\delta\ge \frac{1}{10n}$, we also require that on the event $\mathcal{T}$
\begin{enumerate}[(i)]
\setcounter{enumi}{4}
\item \label{Tv}$N_Q \left (B\left ( z_j, \frac{\delta}{10}\right )\right ) \le  L^{c_1}$ for all $1\le j\le k$.
\item \label{Tvi}$\log|Q(z_j)|\ge -\frac{1}{2} L^{c_1}$ for all $1\le j\le k$.
\item \label{Tvii}$\log|Q(z)|\le \frac{1}{2}L^{c_1}$ for all $z$ such that $|z|\in I(\delta) + (-\delta/2, \delta/2)$.
\end{enumerate}

The rest of the proof consists of several parts. In Section \ref{non-clustering}, we will show that the event $\mathcal T$ occurs with high probability. Then in Section \ref{bounded-variance}, we will show that $\norm{K_j^{P}}_{L^2(z)}$ is small on $\mathcal T$ for all $1\le j\le k$. This allows us to approximate $X_j^{P}$ by a finite sum $\frac{1}{m}\sum_{i=1}^m \log |\check P(w_i)|H_j(w_i)$ using Monte Carlo sampling method. After the approximation step, in Section \ref{log-comparability}, we show that the two approximating expressions for $P$ and $\tilde P$ are close using the Lindeberg swapping technique. Next, in Section \ref{negligible-set}, we show that the tail event $\mathcal T^{c}$ does not contribute significantly to the picture, i.e., $\E \left (\left |\prod_{j=1}^{k}X_j^{P}\right |\textbf{1}_{ \mathcal T^{c}}\right )$ is small. This is the key step of our proof. Finally, we wrap up the proof in Section \ref{finishing}.

\subsection{The event $\mathcal{T}$ occurs with high probability}\label{non-clustering}
Let $A$ be a large constant, say $A = k+2$. And set 

\begin{center}
$\gamma(\delta) = \begin{cases} \delta^{A} &\mbox{if } \frac{\log ^{2}n}{n}\le \delta \le \frac{1}{C}, \\ 
n^{-1/2} & \mbox{if } 0\le \delta < \frac{\log^{2}n}{n}. \end{cases} $
\end{center}

In this section, we show that $\P (\mathcal{T})\ge 1 - C\gamma(\delta)$ for some constant $C$. To show that \eqref{Tiii} and \eqref{Tvi} occur with high probability, we will need two Littlewood-Offord type anti-concentration bounds. The first bound for $\xi_i$ being Rademacher is known as Erd\H{o}s' lemma. We reduce the general case to the Rademacher case and then include a proof of the Erd\H{o}s' lemma. 

\begin{lemma}\label{LO}
If the $\xi_i$'s satisfy Condition 1, there exists a constant $D$ such that for any integer $n\ge 1$, real number  $a>0$, and complex numbers $a_1, \dots, a_n$ with $|a_i|\ge a$ for all i, and for any $z\in \C$, we have
\begin{equation}
 \P\left (\left |\sum _{i=1}^{n}a_i{\xi}_i-z\right |\le \frac{a}{D} \right )\le \frac{D}{\sqrt n}.\nonumber
 \end{equation} 
\end{lemma}

\begin{proof}[Proof of Lemma \ref{LO}]
By translation, we can assume that $\E\xi_i = 0$ for all $i$. It then suffices to show the lemma when the $\xi_i$'s and $a_i$'s are real. Indeed, assume that the statement on the real line holds true. In the general case, assume without loss of generality (wlog) that $a = 1$. Since $|a_i|\ge 1$, either $|\Re(a_i)|\ge \max\{|\Im(a_i)|, \frac{1}{\sqrt{2}}\}$ or $|\Im(a_i)|\ge \max\{|\Re(a_i)|,\frac{1}{\sqrt{2}}\}$. By the pigeonhole principle, we can assume wlog that there are at least $n/2$ indices $i$ such that $|\Re(a_i)|\ge \max\{|\Im(a_i)|, \frac{1}{\sqrt{2}}\}$. For such $i$, set $X_i = \Re(\xi_i)-\frac{\Im(a_i)}{\Re(a_i)}\Im(\xi_i)$, $Y_i = \Im(\xi_i)+\frac{\Im(a_i)}{\Re(a_i)}\Re(\xi_i)$, then $a_i\xi_i = \Re(a_i)(X_i + \sqrt{-1}Y_i)$, $\Var(X_i) + \Var(Y_i) = 1+\frac{\Im^{2}(a_i)}{\Re^{2}(a_i)}\in [1,2]$, and $\E|X_i|^{2+\ep}, \E|Y_i|^{2+\ep}\le 2^{2+\ep}\E|\xi|^{2+\ep}\le 2^{2+\ep}\tau_2$. By the pigeonhole principle, we can then assume wlog that there are at least $n/4$ indices $i$ such that $|\Re (a_i)|\ge 1/\sqrt{2}$ and $\Var(Y_i)\in [1/2,2]$. Now, for such $i$, $\Re (a_i)Y_i = \Re (a_i)\sqrt{\Var(Y_i)}\frac{Y_i}{\sqrt{\Var(Y_i)}}$ with $|\Re (a_i)|\sqrt{\Var(Y_i)}\ge \frac{1}{2}$. This allows one to use the result for the reals (after conditioning on the rest $Y_j$'s) with coefficients $\Re (a_i)\sqrt{\Var(Y_i)}$ and random variables $\frac{Y_i}{\sqrt{\Var(Y_i)}}$ and obtain a constant $D$ such that for any $y\in \R$, 
\begin{equation}
 \P\left (\left |\sum _{i=1}^{n}\Re(a_i)Y_i-y\right |\le \frac{1}{2 D} \right )\le \frac{D}{\sqrt n}.\nonumber
 \end{equation} 
This implies that for all $z\in \C$, $\P\left (\left |\sum _{i=1}^{n}a_i\xi_i-z\right |\le \frac{1}{2 D} \right )\le \frac{D}{\sqrt n}\le \frac{2 D}{\sqrt{n}}$.

Thus, we can assume that the $\xi_i$'s and $a_i$'s are real. We can further assume that the $a_i$'s have the same sign. Indeed, by the pigeonhole principle again, there are at least $n/2$ numbers $a_i$ having the same sign, say, positive. By conditioning on the $\xi_i$'s with $a_i$ negative, we can reduce the problem to the case $a_i>0$ for all $i$. Thus, the assumption becomes $a_i\ge 1$, $\forall i$.

Since the $\xi_i$'s satisfy Condition 1, there exist constants $D$ and $q>0$  such that $\P(D\ge \xi_i - \xi_i'\ge \frac{1}{D})\ge q$, where $\xi_i'$ is an independent copy of $\xi_i$. 
Let $\epsilon_1, \dots, \epsilon_n$ be independent Rademacher random variables which are independent of all previous random variables. Let 
\begin{center}
$\tilde \xi_i = \begin{cases} \xi_i &\mbox{if } \epsilon_i = 1, \\ 
\xi_i' & \mbox{if } \epsilon_i = -1. \end{cases} $
\end{center}

Then $\tilde{\xi}_1,\dots, \tilde{\xi}_n$ are independent random variables having the same distribution as $\xi_1, \dots, \xi_n$. Hence, it suffices to show that for all $x\in \R$,
\begin{equation}
 \P\left (\left |\sum _{i=1}^{n}a_i\tilde{\xi}_i-x\right |\le \frac{1}{3D} \right )= O(\frac{1}{\sqrt n}).\nonumber
\end{equation}

Let $J$ be the set of indices $j$ such that $\xi_j - \xi_j'\ge \frac{1}{D}$. Since $\P(\xi_j - \xi_j'\ge \frac{1}{D})\ge q$, $\E|J| \ge nq$. By Chernoff's bound (see, for instance, \cite[Theorem 1.1]{DP1}), 
\begin{equation}
\P\left (|J|\le \frac{nq}{2	}\right )\le \P\left (|J|\le \frac{\E |J|}{2}\right )\le 2e^{-\frac{\E |J|}{8}}\le 2e^{-\frac{nq}{8}} \le \frac{1}{\sqrt{n}}.
\end{equation}

Conditioning on the event that $|J|\ge \frac{nq}{2}$, and fixing $\tilde \xi_k$'s for all $k\notin J$ as well as $\xi_j$'s, $\xi_j'$'s for all $j\in J$, the only source of randomness left is from $\ep_j$'s with $j\in J$. It suffices to show that for all $x$, $\P(|\sum _{j\in J}a_j\tilde{\xi}_j-x|\le \frac{1}{3D}) = O\left (\frac{1}{\sqrt{n}}\right )$. 

Let $\mathcal F$ be the collection of all subsets $\{j\in J: \epsilon_j = 1\}$ as $\epsilon_j$ run over all possible values such that $|\sum _{j\in J}a_j\tilde{\xi}_j-x|\le\frac{1}{3D}$. Observe that $\mathcal F$ is an anti-chain. Indeed, suppose that $F \subset F'$ be two elements of $\mathcal F$ which correspond to $\epsilon_j = x_j$ and $\epsilon_j = x'_j$ respectively ($x_j, x'_j\in \{\pm 1\}$). For  $\epsilon_j = x_j$,
\begin{eqnarray}
\sum _{j\in J}a_j\tilde{\xi}_j = \sum _{j\in F}a_j{\xi}_j + \sum _{j\in J\setminus F}a_j{\xi'}_j,\label{ss4}
\end{eqnarray}

and for $\epsilon_j = x'_j$,
\begin{eqnarray}
\sum _{j\in J}a_j\tilde{\xi}_j = \sum _{j\in F'}a_j{\xi}_j + \sum _{j\in J\setminus F'}a_j{\xi'}_j,\label{ss5}
\end{eqnarray}

The difference of the expressions in \eqref{ss4} and \eqref{ss5} is 
$$\sum _{j\in F'\setminus F}a_j({\xi}_j-\xi_j')\ge \frac{1}{D}$$
which contradicts the assumption that they both lie in an interval of length at most $\frac{2}{3D}$. Hence, $\mathcal{F}$ is an anti-chain. And so, $|\mathcal{F}|\le {|J|\choose \lfloor |J|/2\rfloor}$ by Sperner's theorem \cite[Chapter 12]{AS}. It follows that of all $2^{|J|}$ choices of the values of $\ep_j$, there are at most ${|J|\choose \lfloor |J|/2\rfloor}$ of them can make $|\sum _{j\in J}a_j\tilde{\xi}_j-x|\le \frac{1}{3D}$. By Stirling's formula, 
\begin{equation}
\P\left (\left |\sum _{j\in J}a_j\tilde{\xi}_j-x\right |\le \frac{1}{3D}\right ) \le \frac{{|J|\choose \lfloor |J|/2\rfloor}}{2^{|J|}} =O\left (\frac{1}{\sqrt {|J|}}\right ) = O\left (\frac{1}{\sqrt{n}}\right ).\nonumber
\end{equation}
This completes the proof.
\end{proof}

The next bound is proven in \cite[Lemma 9.2]{TVpoly}. We include a short proof for the convenience of the reader. 
\begin{lemma}\label{LO2} 
Let $(\xi_i)_{i=1}^{n}$ be independent random variables satisfying Condition 1. There exist positive constants $C'$ and $\alpha$  such that for any complex number $z$, any integer $m$, and any sequence of complex numbers $e_0, \dots, e_n$ containing a lacunary subsequence $|e_{i_1}|\ge 2|e_{i_2}|\ge\dots\ge 2^{m}|e_{i_m}|$, we have
\begin{equation}
\P\left (\left |\sum_{i=0}^{n}e_i\xi_i -z\right|\le |e_{i_m} |\right )\le C'\exp(-\alpha m).\label{lacunary}
\end{equation}
\end{lemma}
\begin{proof} [Proof of Lemma \ref{LO2}]
As in the proof of Lemma \ref{LO}, we can assume that $\E\xi = 0$. Consider $\xi_i'$, $\tilde \xi_i$, $D$ and $q$ as in that proof. Without loss of generality, assume that $D\ge 10$. We can choose a subsubsequence $(e_{i_{j_{k}}})_{k=1}^{\tilde m}$ of the lacunary sequence $(e_{i_j})_{j=1}^{m}$ with $\tilde m = \Theta(m)$ such that $|e_{i_{j_1}}|\ge D^{3}|e_{i_{j_2}}|\ge\dots\ge D^{3\tilde m}|e_{i_{j_k}}|$. By conditioning on the random variables $\xi_l$ with $l$ not equal any $i_{j_{k}}$, we can assume that the subsubsequence equals the original sequence; in other words, $i_{j_{k}} = k$ for all $k$, and $\tilde m = m = n$. Let $J$ be the set of indices $j<m$ such that $D\ge \xi_j - \xi_j'\ge \frac{1}{D}$. By the same argument with Chernoff's bound as before, we have $|J|\ge \frac{mq}{2}$ with probability at least $1-\exp(-\alpha m)$. Conditioning on the event that $|J|\ge \frac{mq}{2}$, and fixing $\tilde \xi_k$'s for all $k\notin J$ as well as $\xi_j$'s, $\xi_j'$'s for all $j\in J$, the only source of randomness left is from $\ep_j$'s with $j\in J$. It suffices to show that for all $z$, $\P(|\sum _{j\in J}e_j\tilde{\xi}_j-z|\le |e_m|) = O\left (\exp(-\alpha m)\right )$. By triangle inequality, we can show that for any two instances of $(\ep_j)_{j\in J}$, the difference of the two sums $\sum _{j\in J}e_j\tilde{\xi}$ has magnitude at least $4|e_m|$. And so, $\P(|\sum _{j\in J}e_j\tilde{\xi}_j-z|\le |e_m|) \le 2^{-|J|}= O\left (\exp(-\alpha m)\right )$.
\end{proof}

We are now ready to show that \eqref{Tiii} and \eqref{Tvi} occur with the desired probability. For $\delta\in [\frac{\log ^{2}n}{n},\frac{1}{C}]$, we prove the following.
\begin{lemma} \label{f1-1} For any constants $A>0$ and $c>0$, there exists a constant $C$ such that for any $\delta\in [\frac{\log ^{2}n}{n},\frac{1}{C}]$, complex number $z$ such that $|z|\in I(\delta)$, and $1\le \lambda\le \frac{n\delta}{\log^{2}\delta}$, one has 
\begin{eqnarray}
\label{moon2}\P\left (\log|P(z)|\ge -\frac{1}{2}\lambda\delta^{-c}\right )\ge 1 - C\frac{\delta^{A}}{\lambda^A},
\end{eqnarray}
and
\begin{eqnarray}
\label{moon3}\P\left (\log|Q(z)|\ge -\frac{1}{2}\lambda\delta^{-c}\right )\ge 1 - C\frac{\delta^{A}}{\lambda^A}.
\end{eqnarray}
\end{lemma}

\begin{proof} 

Since $L= \frac{1}{\delta}\le \frac{n}{\log^{2}n}$, we have $L\log ^{2}L\le \frac{n}{\log^{2}n}\log^{2}n = n$. Thus, there exists some $\lambda$ such that $1\le \lambda\le \frac{n}{L\log^{2}L}=\frac{n\delta}{\log^{2}\delta}$.
Set ${m} = \lceil \frac{\log C' + A\log(\lambda L)}{\alpha}\rceil$, then $C'\exp(-\alpha m)\le \frac{1}{\lambda^{A}L^{A}}$. We obtain a lacunary sequence $|c_{j_0}z^{j_0}|\ge 2|c_{2j_0}z^{2j_0}|\ge\dots\ge 2^{m}|c_{i_0}z^{i_0}|$ where $j_0 = \lceil BL\rceil$, $B$ is a large enough constant, and $i_0 = (m+1)j_0$. 


%

Observe that $i_0\le \frac{n}{2}$ and $c_{i_0}|z|^{i_0}\ge e^{-1/2\lambda L^{c}}$. 
Thus, by applying inequality \eqref{lacunary} to this lacunary sequence, we get \eqref{moon2}. 

Similarly, to prove \eqref{moon3}, we apply \eqref{lacunary} to the lacunary sequence $|\frac{d_{j_0}}{d_0}z^{j_0}|\ge 2|\frac{d_{2j_0}}{d_0}z^{2j_0}|\ge \dots\ge 2^{m}|\frac{d_{i_0}}{d_0}z^{i_0}|$.
\end{proof}

For $\delta\in [\frac{1}{20n},\frac{\log ^{2}n}{n}]$, we prove the following.
\begin{lemma}
 \label{f2-1} 
%
 For any positive constant $c$, there exists a constant $C$ such that for all $\delta\in [\frac{1}{20n},\frac{\log ^{2}n}{n}]$, and complex number $z$ such that $|z|\in I(\delta)$,  it holds that 
$\log|P(z)|\ge -\frac{1}{2}\delta^{-c}$ with probability at least $1 - Cn^{-1/2}$.

If $\delta\ge \frac{1}{10n}$, the same statement holds for $Q$ in place of $P$.
\end{lemma}

\begin{proof}
If $\delta\in [\frac{1}{10n},\frac{\log ^{2}n}{n}]$, $|z|\in [1 - 2\delta, 1-\delta]$, and $N_0\le i\le n$, then  
\begin{eqnarray}
|c_i z^{i}|\ge\tau_1 n^{-|\rho|}(1-2\delta)^{n}\ge \tau_1  n^{-|\rho|}\left (1 - \frac{2\log ^{2}n}{n}\right )^{n}\ge e^{-8\log^{2}n}\ge 2D e^{-\frac{1}{2}L^{c}},\nonumber
\end{eqnarray}
where $D$ is the constant in Lemma \ref{LO}.

By Lemma \ref{LO}, we have $\P\left (|P(z)|\le e^{-\frac{1}{2}L^{c}}\right )\le Cn^{-1/2}$. Note that we may not have $|c_i z^{i}|\ge 2D e^{-\frac{1}{2}L^{c}}$ for $i<N_0$, but by first conditioning on $\xi_0, \dots, \xi_{N_0}$, Lemma \ref{LO} still gives us the desired result.

The same argument holds for $\delta \le \frac{1}{10n}$ and for $Q$ in place of $P$.
%
\end{proof}

In the following lemma, we show that the events \eqref{Tiv} and \eqref{Tvii} occur with high probability.

\begin{lemma}\label{f6-1} For any constants $A>1$ and $c>0$, there exists a constant $C$ such that for any $\frac{1}{10n}\le\delta\le \frac{1}{C}$ and $\lambda\ge 1$, we have $$\log M\le \frac{1}{2}\lambda L^{c}$$ with probability at least $1 - \frac{\delta^{A}}{\lambda^A}$, where $M = \max\{|P(z)|, |Q(z)|:|z|\le 1 - \delta/2\}$.

And if $\frac{1}{20n}\le \delta\le \frac{1}{10n}$ then $$\log M\le \frac{1}{2} L^{c}$$ with probability at least $1 - n^{-1/2}$, where $M = \max\{|P(z)|:|z|\le 1+\frac{4}{n}\}$.
\end{lemma}
\begin{proof}
Assume that $\frac{1}{10n}\le \delta\le \frac{1}{C}$. Let $X = \frac{2-\delta/2}{2-\delta} \in (1, 2)$ and $a_i = \lambda ^{A}L^{A}X^{i}$. Let \[\Omega^\prime = \{\omega: |\xi_i|\le a_i, \forall i = 0, \dots, n\}.\] The probability of the complement of $\Omega^\prime$ is 
\begin{eqnarray}
\P(\Omega^{\prime c})&=& \P\bigg(\exists i \in \{0, 1, \dots, n\}: |\xi_i|> a_i\bigg)\le\sum_{i=0}^n \frac{\tau_2}{a_i^{2}}\le \frac{1}{\lambda^{A}L^{A}}\nonumber.
\end{eqnarray}
For every $\omega\in \Omega^\prime$, we have
\begin{eqnarray}
\max_{|z|\le 1 - \delta/2}|P(z)|\le \sum_{i=0}^n |c_i\xi_i|\left(1-\frac{\delta}{2}\right)^i\le \sum_{i=0}^na_i|c_i| \left(1-\frac{\delta}{2}\right)^i
\le C'' \lambda^{A}L^{A} \left (\frac{4}{\delta}\right )^{\lceil \rho \rceil+1}\le e^{\frac{1}{2}\lambda L^{c}}.\nonumber
\end{eqnarray}

A similar bound holds for $Q$.

When $\frac{1}{20n}\le \delta\le \frac{1}{10n}$, we set $\Omega' = \{\omega: |\xi_i|\le n, \forall 0\le i\le n\}$ and argue similarly. 
%
%
\end{proof}

Combining Lemmas \ref{f1-1}, \ref{f2-1}, and \ref{f6-1}, we obtain that the events \eqref{Tii} and \eqref{Tv} occur with high probability.

\begin{proposition}\textbf{(Non-clustering)}\label{nonclustering-1}
 For any constants $A>1$ and $c>0$, there exists a constant $C$ such that 
 \begin{enumerate}[(i)]
 \item \label{st1}  For any $\delta\in [\frac{\log ^{2}n}{n},\frac{1}{C}]$, $1\le \lambda\le \frac{n\delta}{\log^{2}\delta}$, and complex number $z$ such that $|z|\in I(\delta)$, we have
\begin{equation}
N_{P}(B(z, \delta/9))\le \lambda\delta^{-c} \mbox{ and } N_{Q}({B(z, \delta/9)})\le \lambda\delta^{-c}\nonumber
\end{equation}
with probability at least $1 - C\frac{\delta^{A}}{\lambda^A}$.

\item \label{st2} For any $\delta\in [\frac{1}{20n},\frac{\log ^{2}n}{n}]$ and complex number $z$ such that $|z|\in I(\delta)$, one has 
\begin{equation}
N_P({B(z, \delta/9)})\le \delta^{-c} \nonumber
\end{equation} 
with probability at least $1 - Cn^{-1/2}$.
The same statement holds for $Q$ in place of $P$ when $\delta\in [\frac{1}{10n},\frac{\log ^{2}n}{n}]$.
%
%
%

\end{enumerate}
 
\end{proposition}

\begin{proof}

By our convention, we only need to work on the event that $P$ and $Q$ do not vanish identically. In the following, we prove for $P$. The same argument works for $Q$ equally well.

We first prove \eqref{st1}. By Jensen's inequality, we have
\[N_P({B(z, s)})\le \frac{\log M - \log|P(z)|}{\log \frac{R}{s}} \le \log M - \log|P(z)|,
\]
where $R = \frac{\delta}{3}, s = \frac{\delta}{9}, M = \max_{|w - z| = R} |P(w)|\le \max_{|w| \le 1 - \delta/2} |P(w)|$. 

Claim \eqref{st1} follows from Lemmas \ref{f1-1} and \ref{f6-1}. Similarly, claim \eqref{st2} follows from Lemmas \ref{f2-1} and \ref{f6-1}.
\end{proof}

From the above proposition, we obtain
\begin{proposition} \label{bigT} For any constants $A>1$ and $c_1>0$, there exists a constant $C$ such that for any $\frac{1}{20n}\le \delta\le\frac{1}{C}$, we have	
\begin{equation}
\P(\mathcal{T}(\delta))\ge 1 - C\gamma(\delta).\nonumber
\end{equation}
\end{proposition}

\begin{proof}
By H\H{o}lder's inequality,
\begin{equation}
1=\Var(\xi_i)\le \E|\xi_i|^{2}\le (\E|\xi_i|^{2+\ep})^{2/(2+\ep)}\P(|\xi_i|>0)^{\ep/(2+\ep)}\le \tau_{4}^{2/(2+\ep)}\P(|\xi_i|>0)^{\ep/(2+\ep)}.\nonumber
\end{equation}
Thus, for all $i$, $\P(\xi_i = 0)\le 1 - \frac{1}{C'}$ for some constant $C'$. This gives 
\begin{equation}
\P(P\equiv 0) =\left ( 1-\frac{1}{C}\right )^{n}\le Cn^{-A}\le C\gamma(\delta).\nonumber
\end{equation}
The proposition then follows from Lemmas \ref{f1-1}, \ref{f2-1}, \ref{f6-1}, and \ref{nonclustering-1}, and the union bound.
\end{proof}

\subsection{Approximation of integrals by finite sums}\label{bounded-variance}

Fix $\delta\in \left [\frac{1}{20n}, \frac{1}{C}\right ]$. In this section, we show that on the event $\mathcal T$, the norms $\norm{K_j^{P}}_{L^2(z)}$ are small for all $1\le j\le k$. At the end of the section, this bound allows us to use the Monte Carlo sampling lemma to approximate $X_j^P$ with finite (sample) sums, on which we will apply the Lindeberg swapping argument. The crucial tool in this section is Harnack's inequality which allows us to show that property \eqref{Tiii} in the definition of $\mathcal T$ basically holds for every $z\in B( z_j, 10^{-5}\delta)$.

Recall that $K_j^{P} = \log |\check P(z)|H_j(z)$ and 
\begin{eqnarray}
\norm{K^{P}_j}_{L^2(z)}^2&=&\int_{B(\check z_j, r_0)}\ab{\log |\check P(z)|H_j(z)}^2dz\le\int_{B(\check z_j, r_0)}\ab{\log |\check P(z)|}^2dz \label{sat2}\\
&=& \frac{1}{10^{-6}\delta^2}\int_{B( z_j, 10^{-5}\delta)}{\log ^2| P(z)|}dz \nonumber\\
&&\text{(by change of variables formula for integral on the plane.)}\nonumber
\end{eqnarray}

\begin{lemma}\label{f5} On $\mathcal T$, one has the bound
\begin{equation}\label{norm2}
\norm{\log |P(z)|}_{L^2(B( z_j, 10^{-5}\delta))}\le L^{4c_1-1}.
\end{equation}
\end{lemma}
Note that this is a deterministic statement.

\begin{proof}
Fix $\omega\in\mathcal T$.
Consider $I := [10^{-5}\delta, 10^{-1}\delta]$, we have $\ab{I} \ge \frac{\delta}{20}$. There exists an $r\in I$ such that $P$ does not have zeros in the (closed) annulus $A(z_j, r-\eta, r+\eta)$ where $\eta = \frac{1}{80}\delta^{1+c_1}$. Indeed, assume such an $r$ does not exist, then 
\[N_P{B(z_j, \delta/10)}\ge \frac{\ab{I}}{3\eta}> \delta^{-c_1}
\]
which contradicts the condition \eqref{Tii} in the definition of $\mathcal{T}$.

Now, fix that $r$, we have $\int_{B( z_j, 10^{-5}\delta )}{\log ^2| P(z)|}dz \le \int_{B( z_j, r )}{\log ^2| P(z)|}dz $. Let $\zeta_1,\dots, \zeta_m$ be all zeros of $P$ in $B(z_j, r-\eta)$, then $m \le L^{c_1}$ and $P(z) = (z-\zeta_1)\dots(z-\zeta_m)g(z)$ where $g$ is a polynomial having no zeros on the closed ball $B(z_j, r+\eta)$.
We have
\begin{eqnarray}
\norm{\log |P(z)|}_{L^2(B( z_j, r))}&\le& \sum_{i=1}^m\norm{\log |z-\zeta_i|}_{L^2(B( z_j, r))} + \norm{\log |g(z)|}_{L^2(B( z_j, r))}\nonumber\\
&\le& m\delta^{1 - c_1}+ \norm{\log |g(z)|}_{L^2(B( z_j, r ))},\nonumber
\end{eqnarray}
where the last inequality is because 
\begin{eqnarray}
\int_{B( z_j, r)}\log^2 |z-\zeta_i|\text{d}z \le \int_{B(0,\delta)} \log ^2 |z|\text{d}z 
\le \delta^{2-2c_1}. \nonumber
\end{eqnarray}
Thus, \begin{equation}\label{a1}
\norm{\log |P(z)|}_{L^2(B( z_j, r))}\le L^{2c_1 - 1} + \norm{\log |g(z)|}_{L^2(B( z_j, r))}.
\end{equation}

Next, we will estimate $\int_{B( z_j, r)}{\log ^2| g(z)|}dz$. Since $\log |g(z)|$ is harmonic in $B(z_j, r)$, it attains its extrema on the boundary. Thus, 
\begin{equation}\label{a2}
\norm{\log |g(z)|}_{L^2(B( z_j, r))} = \left (\int_{B( z_j, r)}{\log ^2| g(z)|}dz \right )^{1/2}\le \delta\max_{z\in \partial B(z_j, r)}|\log|g(z)||.
\end{equation} 
Notice that $\log|g(z)|$ is also harmonic on the ball $B(z_j, r+\eta)$. 
\begin{claim}\label{f3}
For every $z$ in $B(z_j, r+\eta)$, we have
\begin{eqnarray}
\log |g(z)| \le L^{2c_1}\nonumber.
\end{eqnarray}
\end{claim}
\begin{proof}
Since a harmonic function attains its extrema on the boundary, we can assume that $z\in\partial B(z_j, r+\eta)$.
Since $\ab{z}< |z_j| + \delta/2$, $|z|\in I(\delta)+(-\delta/2, \delta/2)$. So, by condition \eqref{Tiv} in the definition of $\mathcal T$, $\log|P(z)|\le L^{c_1}$. Additionally, by noticing that $|z-\zeta_i|\ge 2\eta$ for all $1\le i\le m$, we get
\begin{eqnarray}
\log |g(z)| = \log|P(z)| - \sum_{i=1}^m \log \ab{z-\zeta_i}\le L^{c_1} - m\log(2\eta) \le L^{2c_1}
\end{eqnarray}
as desired.
\end{proof}

Now, let $u(z) = L^{2c_1} - \log |g(z)|$, then $u$ is a non-negative harmonic function on the ball $B(z_j, r+\eta)$. By Harnack's inequality (see \cite[Chapter 11]{Ru}) for the subset $B(z_j, r)$ of the above ball, we have that for every $z\in B(z_j, r)$,
\[\alpha u(z_j)\le u(z)\le \frac{1}{\alpha} u(z_j),
\]
where $\alpha =\frac{\eta}{2r+\eta}\ge \frac{\delta^{c_1}}{160}$.
Hence,
\[\alpha \left(L^{2c_1} - \log |g(z_j)|\right)\le L^{2c_1} - \log |g(z)|\le \frac{1}{\alpha} \left(L^{2c_1} - \log |g(z_j)|\right).
\]

And so,
\begin{equation}\label{a3}
\ab{\log |g(z)|}\le \frac{1}{\alpha}\ab{\log |g(z_j)|}+ \frac{1}{\alpha }L^{2c_1} \le 160 L^{c_1}\ab{\log |g(z_j)|} + 160 L^{3c_1}.
\end{equation}
Thus, we reduce the problem to bounding $\ab{\log |g(z_j)|}$.
From Claim \ref{f3} and the condition \eqref{Tiii} in the definition of $\mathcal{T}$, we have
\begin{eqnarray}
L^{2c_1}\ge \log \ab{g(z_j)} &=& \log \ab{P(z_j)} - \sum_{i=1}^m\log \ab{z_j - \zeta_i}\ge \log \ab{P(z_j)} \ge -\frac{1}{2}L^{c_1}.\nonumber
\end{eqnarray}
And so, $\ab{\log|g(z_j)|}\le L^{2c_1}$, which together with 
%
%
\eqref{a3} give
\begin{equation}\label{a3-1}
\ab{\log |g(z)|}\le 320L^{3c_1}.
\end{equation}

From \eqref{a1}, \eqref{a2}, and \eqref{a3-1}, we obtain
\begin{eqnarray}
\norm{\log |P(z)|}_{L^2(B( z_j, r))}&\le&  L^{4c_1-1}.\nonumber
\end{eqnarray}
Lemma \ref{f5} is proved. 
\end{proof}

From this lemma, we conclude that on the event $\mathcal T$,
\begin{equation}\label{a5}
\norm{K_j^{P}}_{L^2(z)} \le 10^{3}.L\norm{\log |P(z)|}_{L^2(B( z_j, 10^{-5}\delta))}\le 10^6L^{4c_1}.
\end{equation}

Having bounded the $2$-norm, we now use the following sampling lemma.
\begin{lemma}[Monte Carlo sampling Lemma] (\cite[Lemma 38]{V17})
Let $(X, \mu)$ be a probability space, and $F: X\to \C$ be a square integrable function. Let $m\ge 1$, let $x_1, \dots, x_m$ be drawn independently at random from $X$ with distribution $\mu$, and let $S$ be the empirical average 
$$S: =  \frac{1}{m}\left (F(x_1)+\dots+ F(x_m)\right ).$$
Then $S$ has mean $\int_{X} Fd\mu$ and variance $\frac{1}{m}\int_{X} \left (F-\int_{X} Fd\mu\right )^{2}d\mu$. In particular, by Chebyshev's inequality, we have
$$\P\left (\left |S-\int_{X} Fd\mu\right |\ge \lambda\right )\le \frac{1}{m\lambda^{2}}\int_{X} \left (F-\int_{X} Fd\mu\right )^{2}d\mu.$$
\end{lemma}

Conditioning on $\mathcal T$ and applying this sampling lemma, we have for large $m_0>0$ and small $\gamma_0>0$ to be chosen later,
\begin{equation}\label{rc1}
\bigg|X_j^{P} - \frac{\pi r_0^2}{m_0}\sum_{i=1}^{m_0} K_j^{P}(\check w_{j,i})\bigg|\le \frac{2\sqrt{\pi r_0^2}}{\sqrt {m_0\gamma_0}} 10^6L^{4c_1} \le \frac{CL^{4c_1}}{\sqrt{m_0\gamma_0}}
\end{equation}
with probability at least $1 -\gamma_0$, where $\check w_{j,i}$ are chosen independently at random from $B(\check z_j, r_0)$ with uniform distribution and are independent from all previous random variables. By exactly the same argument, \eqref{a5} also holds for $Q$ when $\frac{1}{10n}\le \delta\le \frac{1}{C}$.

\subsection{Log comparability}\label{log-comparability}

We shall show in Section \ref{finishing} that \eqref{rc1} allows us to reduce the problem to comparing $F\big(\log|P(z_1)|, \dots, \log|P(z_{m})|\big)$ and $F\big(\log|\tilde P(z_1)|, \dots, \log|\tilde P(z_{m})|\big)$ for some smooth function $F$. This is done by making use of the beautiful Lindeberg swapping trick. The following result is from \cite{TVpoly}, we include a proof in the appendix for the convenience of the reader.

\begin{theorem}[Comparability of log-magnitude]\label{log-com16-1}
Let $P$ be the random polynomial of the form \eqref{poly} satisfying Condition 1 \eqref{cond1i}. And let $\tilde P = \sum_{i=0}^{n}c_i\tilde \xi_i z^{i}$ be the corresponding polynomial with Gaussian random variables $\tilde \xi_i$. Assume that $\tilde \xi_i$ matches moments to second order with $\xi_i$ for every $i\in \{0, \dots, n\}\setminus I_0$ for some (deterministic) set $I_0$ (may depend on $n$) of size at most $N_0$ and that $\sup_{i\ge 0} \E|\tilde \xi_i|^{2+\ep}\le \tau_2$ where $N_0$ and $\tau_2$ are constants in Condition 1 \eqref{cond1i}. 

Then there exists a constant $C_2$ such that the following holds true. Let $\alpha_1\ge C_2 \alpha_0>0$ and $C>0$ be any constants. Let $\delta \in (0, 1)$ and $m\le \delta^{-\alpha_0}$ and $z_1,\dots, z_m\in \mathbb{C}$ be complex numbers such that 
\begin{equation}\label{log-com-e1}
\frac{|c_i||z_j|^i}{\sqrt{V (z_j)}}\le C\delta^{\alpha_1}, \forall i=0,\dots, n, j=1, \dots, m,
\end{equation} 
where $V(z_j)  =\sum_{i\in \{0, \dots, n\}\setminus I_0}|c_i|^{2}|z_j|^{2i}$.

Let $F:\C ^m\to \C$ be any smooth function such that $\left |{\triangledown^a F(w)}\right | \le C\delta^{-\alpha_0}$ for all $0\le a\le 3$ and $w\in \C^{m}$, then 
\[\big|\E F\big(\log|P(z_1)|, \dots, \log|P(z_{m})|\big)-\E F\big(\log|\tilde {P}(z_1)|, \dots, \log|\tilde {P}(z_{m})|\big) \big|\le \tilde C\delta^{\alpha_0},
 \]
where $\tilde C$ is a constant depending only on $\alpha_0, \alpha_1,  C$ and not on $\delta$.
\end{theorem}

Now, we show that condition \eqref{log-com-e1} holds for $P$ and $Q$.

\begin{lemma}\label{log-comp-11}
Under the assumptions Theorem \ref{complex}, there exist constants $\alpha_1>0$ and $C>0$ such that for every $\delta\in [\frac{1}{20 n}, \frac{1}{C}]$ and for every $z$ such that $|z|\in I(\delta) + [-\delta/2,\delta/2]$, 
\begin{equation}
\frac{|c_i||z|^{i}}{\sqrt{\Var P(z)}}\le C\delta^{\alpha_1}, \qquad \forall 0\le i\le n.\label{log-compa1-P}
\end{equation}

and if $\delta\in [\frac{1}{10 n}, \frac{1}{C}]$,
\begin{equation}
\frac{\frac{|d_i|}{|d_0|}|z|^{i}}{\sqrt{\Var Q(z)}}\le C\delta^{\alpha_1}, \qquad \forall 0\le i\le n.\label{log-compa1-Q}
\end{equation}
\end{lemma}

Notice that once \eqref{log-compa1-P} holds, say, the contribution of a few terms in $\Var P(z) = \sum_{i = 0}^{n}|c_i|^{2}|z|^{2i}$ is negligible and hence for any set $I_0$ of size at most $N_0$, we have $\frac{|c_i||z|^{i}}{\sqrt{\sum_{i \in \{0, \dots, n\}\setminus I_0}|c_i|^{2}|z|^{2i}}}\le C\delta^{\alpha_1}$ as required in \eqref{log-com-e1}. 

\begin{proof} Let $\alpha_1 = \min(\rho+1/2, 1/2)>0$. We prove \eqref{log-compa1-P} when $\delta\in [\frac{1}{20 n}, \frac{1}{C}]$. The other parts of the statement are similar. Recall that $L\le 20n$. We have from \eqref{cond-c}, 
\begin{eqnarray}
\Var P(z) &=& \sum_{i=0}^{n}c_i^{2} |z|^{2i}\ge \frac{\tau_1^{2}}{40^{2\rho}}\sum_{i=\lfloor L/40\rfloor}^{\lceil L/20\rceil}L^{2\rho} \left(1 - \frac{5}{2L}\right)^{L}\ge \frac{1}{C}L^{2\rho +1},\label{fri2}
\end{eqnarray}

and
\begin{eqnarray}
c_i^{2}|z|^{2i} \le Ci^{2\rho} \left(1-\frac{1}{2L}\right)^{2i}\le Ci^{2\rho}e^{-i/L}\le C \max(1, L^{2\rho})\le CL^{-2\alpha_1}\Var P(z)\quad\forall 0\le i\le n,\nonumber
\end{eqnarray}
where the next to last inequality follows from the boundedness of the function $x\mapsto x^{2\rho}e^{-x}$ on $[0, \infty)$ whenever $\rho\ge 0$ and is trivial when $\rho<0$. 
\end{proof}

Combining Theorem \ref{log-com16-1} and Lemma \ref{log-comp-11}, we obtain

\begin{proposition}\textbf{(Log-comparability)}\label{logg-1}
There exist constants $\alpha_0>0$ and $C>0$ such that for every $\delta\in [\frac{1}{20 n}, \frac{1}{C}]$, $1\le m\le \delta^{-\alpha_0}$, $\ab{z_1}, \dots, \ab{z_m}\in I(\delta) + [-\delta/2, \delta/2]$, and smooth function $F:\mathbb{C}^m\to \mathbb{C}$ with $\norm{\triangledown^a F}\le \delta^{-\alpha_0},\forall 0\le a\le 3$, we have
\[\big|\E F\big(\log|P(z_1)|, \dots, \log|P(z_{m})|\big)-\E F\big(\log|\tilde {P}(z_1)|, \dots, \log|\tilde {P}(z_{m})|\big) \big|\le {C}\delta^{\alpha_0}.
 \]
and if $\delta\in [\frac{1}{20 n}, \frac{1}{C}]$, we have
\[\big|\E F\big(\log|Q(z_1)|, \dots, \log|Q(z_{m})|\big)-\E F\big(\log|\tilde {Q}(z_1)|, \dots, \log|\tilde {Q}(z_{m})|\big) \big|\le {C}\delta^{\alpha_0}.
 \]
\end{proposition}

\subsection{On the tail event $\mathcal T^{c}$}\label{negligible-set}

In this section, we show that if $\mathcal T$ is any event such that $\P(\mathcal T^{c})\le C\gamma(\delta)$ then the contribution from $\mathcal{T}^{c}$ is negligible. We make use of the powerful result in \cite{NNS} to deal with the case when the $\xi_i$'s are symmetric. The general case requires some additional tricks in the end.

\begin{lemma}\label{exceptionset-2.1}
There exists some constant $C$ such that for all $\delta\in [\frac{1}{20n}, \frac{1}{C}]$, one has
\begin{equation}
\E\left  (\left|\prod_{j=1}^{k}X_j^{P}\right |\textbf{1}_{\mathcal T^{c}}\right )\le C\delta^{1/22},\label{exceptionset}
\end{equation}
and when $\delta\in [\frac{1}{10n}, \frac{1}{C}]$, one has
\begin{equation}
\E\left  (\left|\prod_{j=1}^{k}X_j^{Q}\right |\textbf{1}_{\mathcal T^{c}}\right )\le C\delta^{1/22}.\label{exceptionset-Q}
\end{equation}
\end{lemma}

\begin{proof} We will consider two cases.

\textbf{Case 1.} $\frac{\log ^{2}n}{n}\le \delta\le \frac{1}{C}$. We have $\P(\mathcal T^{c})\le C\gamma(\delta) = C\delta^{A}$.

By Proposition \ref{nonclustering-1}, there exists a constant $C_1$ such that for any $1\le \lambda\le \frac{n}{L\log ^{2}L}$,
\[N_{\check P}(B(\check z_j, r_0)) = N_P(B(z_j, 10^{-5}\delta))\le \lambda\delta^{-c_1},
\]

with probability at least $1 - C_1\frac{1}{\lambda^AL^{A}}$. Hence, $|X_j^{P}|\le \delta^{-c_1}$ with probability at least $1 - C_1\frac{1}{\lambda^{A}L^{A}}$.

For each $i$ such that $i_0>i\ge 1$, where $2^{i_0-1}\le \frac{n}{L\log ^{2}L}< 2^{i_0}$, let 
$$\Omega_i = \{\omega\in \mathcal T^{c}: 2^{i-1}\delta^{-c_1}< N_{\check P}{B(\check z_j, r_0)}\mbox{ for some $1\le j\le k$, and } N_{\check P}{B(\check z_j, r_0)}\le 2^i\delta^{-c_1}\mbox{, }\forall 1\le j\le k\}.$$ 

Let $\Omega_0 = \{\omega\in \mathcal T^{c}: N_{\check P}{B(\check z_j, r_0)}\le \delta^{-c_1}\quad\forall 1\le j\le k\}$ and
$$\Omega_{i_0} =  \{\omega\in \mathcal T^{c}: 2^{i_0-1}\delta^{-c_1}< N_{\check P}{B(\check z_j, r_0)} \quad\mbox{for some $1\le j\le k$}\}.$$ 

Then $\mathcal T^{c} = \cup_{i=0}^{i_0}\Omega_i$, $\P (\Omega_i)\le \frac{C_1\delta^A}{2^{(i-1)A}}$ for all $i\le i_0$ and $|X_j^{P}|\le 2^i\delta^{-c_1}$ on $\Omega_i$ for all $i<i_0$, and $|X_j^{P}|\le n$ on $\Omega_{i_0}$.

Using the assumption that $\frac{\log^{2}n}{n}\le \delta$ and $A\ge k+2$, we have
\begin{eqnarray}
\E\left  (\left|\prod_{j=1}^{k}X_j^{P}\right |\textbf{1}_{\mathcal T^{c}}\right )&\le&\sum_{i=0}^{i_0-1} \E\left  (\left|\prod_{j=1}^{k}X_j^{P}\right |\textbf{1}_{\Omega_{i}}\right ) + \E\left  (\left|\prod_{j=1}^{k}X_j^{P}\right |\textbf{1}_{\Omega_{i_0}}\right )\nonumber\\
&\le& \sum_{i=0}^\infty \frac{C_1\delta^A}{2^{(i-1)A}}\left (2^i\delta^{-c_1}\right )^{k} + \frac{C_1n^k\delta^A}{2^{(i_0-1)A}}\nonumber\\
&\le & C_1\delta^{A-kc_1}\sum_{i=1}^\infty \frac{1}{\left(2^{A-k}\right)^i} + \frac{C_1n^k\delta^A}{\left (n/(2L\log^{2}L)\right )^A}\le C\delta^{1/22}\nonumber.
\end{eqnarray}

\textbf{Case 2.} $\frac{1}{20n}\le \delta\le \frac{\log^{2}n}{n}$. Then we have $\P(\mathcal T^{c})\le C\gamma(\delta) = Cn^{-1/2}$ and $|z_j|\in [1-\frac{2\log^{2}n}{n}, 1+\frac{1}{n}]$ for all $1\le j\le k$.

Since $\xi_i$'s satisfy Condition 1 \eqref{cond1i}, there exist positive constants $d$ and $q$ such that $\P(|\xi_i|< d)\le q< 1$. Indeed, if for some $d>0$, $\P(|\xi_i|< d)> 1-d$, then by H\H{o}lder's inequality
\begin{equation}\label{dq}
1\le \E|\xi_i|^{2} = \E|\xi_i|^{2} \textbf{1}_{|\xi_i|< d} + \E|\xi_i|^{2} \textbf{1}_{|\xi_i|\ge d}\le d^{2} + d^{\ep/(2+\ep)}(\E|\xi_i|^{2+\ep})^{2/(2+\ep)}\le d^{2} + d^{\ep/(2+\ep)}\tau_2^{2/(2+\ep)}.
\end{equation}
Thus, one can choose $d$ small enough (depending on $\tau_2$ and $\ep$), and $q = 1-d$ to have $\P(|\xi_i|< d)\le q< 1$.

\textbf{Subcase 2.1.} We first consider the case when the random variables $\xi_i$'s are symmetric. In other words, $\xi_i$ and $-\xi_i$ have the same distribution.

Let 
\begin{equation}
 \mathcal{V} = \{\omega\in \mathcal{T}^{c}: |\xi_i|\ge d \mbox{ for some } i\in [N_0, n] \}.\label{def-V}
 \end{equation} 
Since $|X_j^{P}|\le n$, one has
\begin{equation}
 \E\left (\left|\prod_{j=1}^{k}X_j^{P}\right |\textbf{1}_{ \mathcal{T}^{c}\setminus \mathcal{V}}\right )\le n^{k}\P(|\xi_i|< d, \forall i\in [N_0, n])\le n^{k}q^{n-N_0}\le \frac{1}{20n}\le \delta,\label{exception-set1}
 \end{equation} 
when $n$ is sufficiently large. Thus, it suffices to show that
\begin{equation}
\E\left  (\left|\prod_{j=1}^{k}X_j^{P}\right |\textbf{1}_{\mathcal V}\right )\le C'\delta^{1/22}.\label{exceptionset-2}
\end{equation}
By H\H{o}lder's inequality, we have
\begin{eqnarray}
\E\left  (\left|\prod_{j=1}^{k}X_j^{P}\right |\textbf{1}_{\mathcal V}\right )&\le& \prod_{j=1}^{k}
\E\left  (\left |X_j^{P}\right |^{k}\textbf{1}_{\mathcal V}\right )^{1/k}.
\label{sun1}
\end{eqnarray}
And so, we reduce the problem to showing that 
\begin{equation}\label{y1}
\E\left |X_j^{P}\right |^{k}\textbf{1}_{\mathcal V}\le C\delta^{1/22}, \quad\forall 1\le j\le k.
\end{equation}
From \eqref{sat1} and the change of variables formula, we obtain 
\begin{equation}
|X_j^{P}| \le C L^{2}\int_{B( z_j, 10^{-5}\delta)}\left |\log | P(z)|H_j\left (\frac{z}{10^{-3}\delta}\right )\right |\text{d}z\le CL^{2}\int_{B( z_j, 10^{-5}\delta)}|\log | P(z)||\text{d}z.\label{sat3}
\end{equation}
And from H\H{o}lder's inequality, we have
\begin{eqnarray}
\E\left |X_j^{P}\right |^{k}\textbf{1}_{\mathcal V}&\le& CL^{2k}\int_{\mathcal V}\left (\int_{B( z_j, 10^{-5}\delta)}|\log | P(z)||\text{d}z\right )^{k}\text{d}\P\nonumber\\
&\le&C L^{2k}\int_{\mathcal V}\left (\int_{B( z_j, 10^{-5}\delta)}|\log | P(z)||^{k}\text{d}z\right )|B( z_j, 10^{-5}\delta)|^{k-1}\text{d}\P\nonumber,\\
&\le&C L^{2}\int_{\mathcal V}\int_{B( z_j, 10^{-5}\delta)}|\log | P(z)||^{k}\text{d}z\text{d}\P\nonumber,\\
&\le&C L^{2}\left(\int_{\mathcal V}\int_{B( z_j, 10^{-5}\delta)}|\log | P(z)||^{kp}\text{d}z\text{d}\P\right)^{1/p}\left(\int_{\mathcal V}\int_{B( z_j, 10^{-5}\delta)}1\text{d}z\text{d}\P\right)^{1/q}\nonumber,
\end{eqnarray}
where $p$ and $q$ are positive constants to be chosen later so that $\frac{1}{p}+\frac{1}{q}=1$.

From this and the observation that $B(z_j, 10^{-5}\delta)\subset A\left (0, 1-\frac{3\log^{2}n}{n}, 1+\frac{2}{n}\right )=:\mathcal D$, we obtain
\begin{eqnarray}
\E\left |X_j^{P}\right |^{k}\textbf{1}_{\mathcal V}&\le& CL^{2}\left(\int_{\mathcal V}\int_{\mathcal D}|\log | P(z)||^{kp}\text{d}z\text{d}\P\right)^{1/p}\left(\frac{1}{\sqrt{n}L^{2}}\right)^{1/q}\label{tt1}.
\end{eqnarray}
For each $N_0\le i\le n$, let $\mathcal V_i = \{\omega\in \Omega: |\xi_i|\ge d\}$. By \eqref{def-V}, $\mathcal V \subset \bigcup_{i=N_0}^{n} \mathcal{V}_i$. Note that this bound is very generous because the measure of $\mathcal V_i$ can be very big. Let $I_i = \int_{\mathcal V_i}\int_{\mathcal D}|\log | P(z)||^{kp}\text{d}z\text{d}\P$. Then 
\begin{equation}
\int_{\mathcal V}\int_{\mathcal D}|\log | P(z)||^{kp}\text{d}z\text{d}\P\le \sum_{i=N_0}^{n}I_i=: I.\label{tt1-1}
\end{equation}
Fix $N_0\le i_0\le n$. We will upper bound $I_{i_0}$.

Let $(\ep_m)_{m\in\Z}$ be independent Rademacher random variables (independent of all previous random variables). In  \cite[Corollary 2.2]{NNS}, Nazarov, Nishry, and Sodin showed that 

\begin{theorem}\label{nns}
There exists a constant $C_1$ such that for any $g(\theta) = \sum_{j\in \Z} a_j\epsilon_j e^{\sqrt{-1}2\pi j\theta}$ with deterministic coefficients $a_j$'s satisfying $\sum_{j\in \Z}|a_j|^{2} = 1$, and any $p_0\ge 1$, one has
\begin{equation}
\E \int_0^{1}|\log|g(\theta)||^{p_0} \text{d}\theta\le (C_{1}p_0)^{6p_0}.\label{nns1}
\end{equation}
\end{theorem}

As a consequence, for any complex numbers $a_0, \dots, a_n$, by Minkowski's inequality for $L^{p}((\Omega\times [0, 1), \P\times m))$, we have
\begin{equation}
\E \int_0^{1}|\log|\sum_{j=0}^{n} a_j\epsilon_j e^{\sqrt{-1}2\pi j\theta}||^{p_0} \text{d}\theta\le \left ((C_{1}p_0)^{6} + \frac{1}{2}\left |\log \sum_{j=0}^{n}|a_j|^{2}\right |\right )^{p_0}\le  2^{p_0}(C_{1}p_0)^{6p_0} + \left |\log \sum_{j=0}^{n}|a_j|^{2}\right |^{p_0}\label{nns2}.
\end{equation}
Let $\hat{\xi}_k = \epsilon_k\xi_k$. Since $\xi_k$ is symmetric, $\hat{\xi}_k$ has the same distribution as $\xi_k$. And so, the random variables $\hat \xi_0, \dots, \hat \xi_n$ have the same joint distribution as $\xi_0, \dots, \xi_n$. Thus, from the definition of $I_{i_0}$, we have
\begin{eqnarray}
I_{i_0}& =&\int_{|\hat \xi_{i_0}| \ge d}\int_{\mathcal D}\left | \log\left | \sum_{j=0}^{n}c_j\hat{\xi}_jz^{j}\right |\right |^{kp}\text{d}z\text{d}\P\nonumber\\
&=& 2\pi\int_{1-\frac{3\log^{2}n}{n}}^{1+2/n}r\int_{| \xi_{i_0}| \ge d}\int_0^1 \left |\log \left | \sum_{j=0}^{n}c_j\xi_j \ep_j r^{j}e^{\sqrt{-1}2\pi j\theta}\right |\right |^{kp}\text{d}\theta\text{d}\P\text{d}r\nonumber.
\end{eqnarray}
Conditioning on the event $| \xi_{i_0}| \ge d$ and fixing the $\xi_i$'s, from \eqref{nns2}, we obtain
\begin{eqnarray}
\E_{\ep_0, \dots, \ep_n}\int_0^1 \left |\log \left | \sum_{j=0}^{n}c_j\xi_j \ep_j r^{j}e^{\sqrt{-1}2\pi j\theta}\right |\right |^{kp}\text{d}\theta \le  (2C_{1}kp)^{6kp} + \left |\log \sum_{j=0}^{n}|c_j\xi_jr^{j}|^{2}\right |^{kp}.\nonumber
\end{eqnarray}
Undoing the conditioning, we get
\begin{eqnarray}
I_{i_{0}}&\le&  2\pi\int_{1-\frac{3\log^{2}n}{n}}^{1+2/n}r\left( (C p)^{6kp}+ \int_{| \xi_{i_0}| \ge d}\left|\log\sum_{j=0}^{n}|c_j j\xi_jr^{j}|^{2}\right|^{kp}\text{d}\P\right)\text{d}r\nonumber\\
&\le & C + C\int_{1-\frac{3\log^{2}n}{n}}^{1+2/n}r \int_{| \xi_{i_0}| \ge d}\left|\log\sum_{j=0}^{n}|c_jr^{j}\xi_j|^{2}\right|^{kp}\text{d}\P\text{d}r\label{om1}.
\end{eqnarray}
By \eqref{cond-c}, for every $N_0\le j\le n$ and $1-\frac{3\log^{2}n}{n}\le r\le 1+\frac{2}{n}$, we have $\frac{1}{Cn^{2|\rho|}}r^{2j}\le c_j^{2}r^{2j}\le Cn^{2|\rho|}$. And hence, on the event $|\xi_{i_0}|\ge d$,
\begin{eqnarray}
\frac{1}{Cn^{2|\rho|}}r^{2i_{0}}d^{2}\le \sum_{j=0}^{n}|c_jr^{j}\xi_j|^{2}\le Cn^{2|\rho|+1}\sum_{j=0}^{n}\xi_j^{2}.\nonumber
\end{eqnarray}
Hence, $\left|\log \sum_{j=0}^{n}|c_jr^{j}\xi_j|^{2}\right|\le \max\{\left|\log\left ( \frac{1}{Cn^{2|\rho|}}r^{2i_{0}}d^{2}\right )\right|, \left|\log \left (Cn^{2|\rho|+1}\sum_{j=0}^{n}\xi_j^{2}\right )\right|\}$. And so, 
\begin{eqnarray}
&&\E\textbf{1}_{| \xi_{i_0}| \ge d}\left |\log\sum_{j=0}^{n}|c_jr^{j}\xi_j|^{2}\right |^{kp}\le C\E\textbf{1}_{| \xi_{i_0}| \ge d}\left |\log\sum_{j=0}^{n}\xi_j^{2}\right |^{kp} + C\log^{kp} n + C\left |\log |r^{2i_{0}}d^{2}|\right |^{kp}\nonumber\\
&=&  C\E\textbf{1}_{| \xi_{i_0}| \ge d, \sum_{j=0}^{n}\xi_j^{2}\ge 1}\left |\log\sum_{j=0}^{n}\xi_j^{2}\right |^{kp} + C\E\textbf{1}_{| \xi_{i_0}| \ge d, \sum_{j=0}^{n}\xi_j^{2}< 1}\left |\log\sum_{j=0}^{n}\xi_j^{2}\right |^{kp}\nonumber\\
&&\quad +C\log^{kp} n  + C\left |\log |r^{2i_{0}}d^{2}|\right |^{kp}\nonumber\\
&\le& C\log^{kp}n+ C\left |\log d^{2}\right |^{kp}+ C\left |\log |r^{2i_{0}}d^{2}|\right |^{kp}.\nonumber
\end{eqnarray}
Notice that under Condition 1, the coefficients $c_i$ of $P$, thanks to their polynomial growth, only contribute the term $\log^{kp}n$ in the above estimate. The same applies for $Q$ because $\frac{C}{n^{|\rho|}}\le \left |\frac{d_i}{d_0}\right |\le C n^{|\rho|}$.

%
%
%
%
Plugging in \eqref{om1} gives that for all $N_0\le i_0\le n$, one has
%
%
\begin{equation}\label{integral}
 I_{i_{0}}\le C\log^{kp} n  + C\int_{1-\frac{3\log^{2}n}{n}}^{1+2/n}r \left |\log r^{2i_{0}}d^{2}\right |^{kp}\text{d}r \le C\log^{kp} n,
 \end{equation} and so
 \begin{equation}\label{integral2}
 \int_{\exists i\in [N_0, n]: | \xi_{i}| \ge d}\int_{\mathcal D}\left | \log\left | \sum_{j=0}^{n}c_j{\xi}_jz^{j}\right |\right |^{kp}\text{d}z \le I = \sum_{i=N_0}^{n}I_i \le C_2n\log^{kp}n.
 \end{equation}
We now use a rude bound which will be convenient for the next subcase.  
\begin{equation}\label{integral3}
 \int_{\exists i\in [N_0, n]: | \xi_{i}| \ge d}\int_{\mathcal D}\left | \log\left | \sum_{j=0}^{n}c_j{\xi}_jz^{j}\right |\right |^{kp}\text{d}z \le I\le C_2n^{4/3}\log^{kp}n.
 \end{equation}
Combining this with \eqref{tt1} and \eqref{tt1-1}, we obtain
\begin{eqnarray}
\E\left |X_j^{P}\right |^{k}\textbf{1}_{\mathcal V}&\le& CL^{2}I^{1/p}\left(C\frac{1}{\sqrt{n}L^{2}}\right)^{1/q}\le CL^{2}n^{4/3p}\log^{k} n\left(\frac{1}{\sqrt{n}L^{2}}\right)^{1/q}\nonumber\\
&\le&  CL^{2}L^{8/3p}(\log L^{2})^{k} \left(\frac{1}{L^{5/2}}\right)^{1/q}\mbox{ since $\sqrt n\le \frac{n}{\log^{2}n}\le L\le 20n$} \nonumber\\
&=& C\frac{\log^{k} L}{ L^{\frac{5}{2q}-2-\frac{8}{3p}}}= C\frac{\log^{k}  L}{ L^{\frac{1}{2}-\frac{31}{6p}}}\le  \frac{C}{ L^{1/22}}\quad \mbox{by choosing $p = 12$, $q = \frac{12}{11}$}.\nonumber
\end{eqnarray}
This together with \eqref{sun1} complete the proof of \eqref{exceptionset-2}.

\textbf{Subcase 2.2.} Now let us consider the case when the $\xi_i$'s are not symmetric. The trick is to reduce to the symmetric case. For clarity, we write $P_{\xi}(z) = \sum_{l=0}^{n}c_l\xi_lz^{l}$.

Recall that $d$ and $q$ are constants such that $\P(|\xi_i|< d)\le q< 1$. Let $\xi_0', \dots, \xi_n'$ be independent copies of $\xi_0, \dots, \xi_n$ correspondingly. For this subcase, instead of \eqref{def-V}, we set
\begin{equation}
 \mathcal{V} = \{\omega\in \mathcal{T}^{c}: |\xi_i|\ge d, |\xi_i'|\ge d \mbox{ for some } i\in [N_0, n] \}.\label{def-V-asym}
 \end{equation}   
Correspondingly, 
\begin{equation}
 \mathcal{V}_i = \{\omega\in \Omega: |\xi_i|\ge d, |\xi_i'|\ge d\}.
 \end{equation}   
Let $\bar\xi_l = \frac{\xi_l - \xi'_l}{\sqrt{2}}$. Then the $\bar \xi_l$'s are symmetric and satisfy Condition 1 \eqref{cond1i} (with a different $\tau_2$). Let $\bar d<1$ and $\bar q$ be positive constants such that $\P(|\bar \xi_1|< \bar d)\le \bar q< 1$ for all $i$.
 
In the following, we will show that   
\begin{equation}\label{x3}
I_i = \int_{\mathcal V_i}\int_{\mathcal D}|\log | P_{\xi}(z)||^{kp}\text{d}z\text{d}\P\le 3n^{1/3}\log ^{10kp}n=:3K_0\quad\mbox{for all $N_0\le i\le n$},
\end{equation}
where $p=12$ and then, one can use the same argument as in the symmetric case to complete the proof. 

Let 
\begin{equation}\label{def_j0}
 j_0=\left\lceil \frac{1}{\bar q^{(n+1)/(4kp+8)}}\right\rceil.
 \end{equation} 
We will first show that  
\begin{equation}\label{int2}
\int_{\mathcal V_i}\int_{\mathcal D}|\log | P_{\xi}(z)||^{kp}\textbf{1}_{B_\xi}\text{d}z\text{d}\P\le K_0\quad\mbox{for all $N_0\le i\le n$},
\end{equation}
where $B_\xi = \{(\omega, z)\in \mathcal V_i\times \mathcal D: |\log|P_{\xi}(z)||\ge j_0\}$.
Indeed, by H\H{o}lder's inequality, one has
\begin{eqnarray}\label{int21}
\int_{\mathcal V_i}\int_{\mathcal D}|\log | P_{\xi}(z)||^{kp}\textbf{1}_{B_\xi}\text{d}z\text{d}\P\le \left(\int_{\mathcal V_i}\int_{\mathcal D}|\log | P_{\xi}(z)||^{2kp}\text{d}z\text{d}\P\right)^{1/2}\left(\int_{\mathcal V_i}\int_{\mathcal D}\textbf{1}_{B_\xi}\text{d}z\text{d}\P\right)^{1/2}.
\end{eqnarray}

To bound the first integral on the right, let $\ep_0', \dots, \ep_n'$ be independent Rademacher variables defined on $\left(\{\pm 1\}^{n+1}, \nu\right)$ where $\nu$ is the uniform probability measure on $\{\pm 1\}^{n+1}$. Let $(\hat \Omega, \hat \mu) = (\Omega\times \{\pm 1\}^{n+1}, \P\times \nu)$, and define the random variables $\hat \xi_i (\omega_1, \omega_2) = \xi_i(\omega_1) \ep_i'(\omega_2)$ for all $\omega_1\in \Omega$ and $\omega_2\in \{\pm 1\}^{n+1}$.

Observe that $\hat \xi_i'$ is symmetric and equal to $\xi_i$ when $\ep_i'=1$. Let $s>1$ be any constant such that $2^{1/s}\le \frac{1}{\bar q^{kp/(4kp+8)}}$. We have
\begin{eqnarray}
&&\int_{\mathcal V_i}\int_{\mathcal D}|\log | P_{\xi}(z)||^{2kp}\text{d}z\text{d}\P = 2^{n+1}\int_{\{1\}^{n+1}}\int_{\mathcal V_i}\int_{\mathcal D}|\log | P_{\xi}(z)||^{2kp}\text{d}z\text{d}\P\text{d}\nu\nonumber\\
&=&2^{n+1}\int_{\mathcal V_i\times\{1\}^{n+1}}\int_{\mathcal D}|\log | P_{\hat\xi}(z)||^{2kp}\text{d}z\text{d}\hat \mu\quad\mbox{by Fubini's theorem}\nonumber\\
&\le&2^{n+1}\hat\mu(\mathcal V_i\times\{1\}^{n+1})^{1-1/s}m(\mathcal{D})^{1-1/s}\left(\int_{\mathcal V_i\times\{1\}^{n+1}}\int_{\mathcal D}|\log | P_{\hat\xi}(z)||^{2kps}\text{d}z\text{d}\hat \mu\right)^{1/s}\nonumber\\
&&\qquad\quad\mbox{by H\H{o}lder's inequality}\nonumber\\
&\le&2^{(n+1)/s}\left(\int_{\mathcal V_i\times\{1\}^{n+1}}\int_{\mathcal D}|\log | P_{\hat\xi}(z)||^{2kps}\text{d}z\text{d}\hat \mu\right)^{1/s}\quad\mbox{because $\hat \mu (\mathcal V_i\times\{1\}^{n+1})\le 2^{-n-1}$}\nonumber\\
&\le& C2^{(n+1)/s}\log ^{2kp}n\quad\mbox{by \eqref{integral} for $\hat\xi_i$}\nonumber.
\end{eqnarray}
A bound for the second integral on the right of \eqref{int21} can also be derived from the above bound.
\begin{eqnarray}
\int_{\mathcal V_i}\int_{\mathcal D}\textbf{1}_{B_\xi}\text{d}z\text{d}\P &\le& \frac{1}{j_0^{2kp}}\int_{\mathcal V_i}\int_{\mathcal D}|\log | P_{\xi}(z)||^{2kp}\text{d}z\text{d}\P\nonumber\\
&\le& C\frac{2^{n/s}}{j_0^{2kp}}\log ^{2kp}n\nonumber.
\end{eqnarray}

Plugging into \eqref{int21} gives
\begin{eqnarray}\label{int22}
\int_{\mathcal V_i}\int_{\mathcal D}|\log | P_{\xi}(z)||^{kp}\textbf{1}_{B_\xi}\text{d}z\text{d}\P\le C\frac{2^{(n+1)/s}}{j_0^{kp}}\log ^{2kp}n\le  C\log ^{2kp}n<K_0,
\end{eqnarray}
where the next to last inequality follows directly from the way we set $s$ and $j_0$. This proves \eqref{int2}.

Now, assume to the contrary that \eqref{x3} failed, i.e., $I_i>3K_0$ for some $i$. Thanks to \eqref{int2}, one then has
\begin{eqnarray}\label{int23}
\int_{\mathcal V_i}\int_{\mathcal D}|\log | P_{\xi}(z)||^{kp}\textbf{1}_{|\log | P_{\xi}(z)||< j_0}\text{d}z\text{d}\P>2K_0.
\end{eqnarray}



For each $z\in \mathcal{D}$ and $1\le j\le j_0$, set $\mu_z(j) = \P(\mathcal V_i\cap (j-1\le|\log|P_{\xi}(z)||<j))$.  

Since $$2K_0< I_i \le \sum_{j=1}^{j_0}\int_{\mathcal D}j^{kp}\mu_\xi(j)\text{d}\P\le \sum_{j=1}^{j_0} j^{kp} m\left(\mathcal D\cap \{\omega: \mu_{\xi}(j)\ge \frac{1}{j^{kp+2}}\}\right) + \sum_{j=1}^{\infty} \frac{1}{j^{2}},$$ and $\sum_{j=1}^{\infty}\frac{1}{j^{2}}\le 2$, there exists a number $j\le j_0$ such that 
\begin{equation}\label{x5}
1\ge m(\mathcal D)\ge m(\mathcal D_0)\ge \frac{K_0}{2j^{kp+2}}.
\end{equation}
where $\mathcal D_0 = \{z\in \mathcal D: \mu_{\xi}(j)\ge \frac{1}{j^{kp+2}}\}$.
Since $j^{kp+2}\ge \frac{K_0}{2}\ge \frac{\log^{10kp}n}{2}$, we have $j\ge \log^{5}n$. Observe that by Markov's inequality and Condition 1, for any $z\in \mathcal D$, 
\begin{eqnarray}
\P(\log|P(z)|\ge j-1)\le \frac{\E|P(z)|^{2}}{e^{2j-2}}\le\frac{1}{e^{2j-2}}\left(\sum_{i=0}^{n}|c_i||z|^{i}(\E|\xi_i|^{2})^{1/2}\right)^{2}\le  \frac{n^{2\rho+2}}{e^{2j-2}}\le\frac{1}{e^{j}}\le \frac{1}{2j^{kp+2}}\nonumber.
\end{eqnarray}
Thus, for every $z\in \mathcal D_0$, $p_z:= \P(\mathcal V_i\cap (-j<\log|P_{\xi}(z)|\le -j +1))\ge \frac{1}{2j^{kp+2}}$.
%
%


On the set $\mathcal D_0$, 
\begin{eqnarray}
&&\P(\omega\in \mathcal V_i: -j<\log|P_{\xi}(z)|, \log|P_{\xi'}(z)|\le -j +1 \mbox{ and } \exists i'\in [N_0, n]: |\bar{\xi}_{i'}|\ge \bar d)\nonumber\\
&\ge&\P(\omega\in \mathcal V_i: -j<\log|P_{\xi}(z)|, \log|P_{\xi'}(z)|\le -j +1 )-\P(|\bar{\xi}_{i'}|< \bar d, \forall i') \nonumber\\
&\ge&  p_z^{2}- \bar q^{n+1} .\nonumber
\end{eqnarray}
From the definition \eqref{def_j0} of $j_0$, we have
\begin{eqnarray}
\bar q^{n+1}\le \frac{1}{(j_0-1)^{4kp+8}}\le \frac{1}{2}p_z^{2},
\end{eqnarray}
and thus, on $\mathcal D_0$,
\begin{eqnarray}
&&\P\left(\omega\in \mathcal V_i: -j<\log|P_{\xi}(z)|, \log|P_{\xi'}(z)|\le -j +1 \mbox{ and } \exists i'\in [N_0, n]: |\bar{\xi}_i|\ge \bar d\right)\ge\frac{1}{2} p_z^{2}.\nonumber
\end{eqnarray}
Hence, 
\begin{eqnarray}
\P\times m\left((\omega, z)\in \bar{ \mathcal U}\times \mathcal D_0: -j<\log|P_{\xi}(z)|, \log|P_{\xi'}(z)|\le -j +1\right) \ge \int_{\mathcal D_0} \frac{1}{2}p_z^{2}\text{d}z\ge\frac{1}{2} \left(\frac{K_0}{16j^{2kp+4}}\right)^{3}\nonumber
\end{eqnarray}
where $\bar{\mathcal U} = \{\omega: \exists i'\in [N_0, n]: |\bar{\xi}_{i'}|\ge \bar d\}$.
Note that when $-j<\log|P_{\xi}(z)|, \log|P_{\xi'}(z)|\le -j +1$, we have $|P_{\bar \xi}(z)|\le \sqrt{2}e^{-j+1}$, so $\log|P_{\bar \xi}(z)|\le -\frac{j}{2}$. This implies
\begin{eqnarray}
\int_{\bar{\mathcal U}}\int_{\mathcal D}|\log|P_{\bar \xi}(z)||^{6kp+12}\ge \frac{1}{2}\left (\frac{j}{2}\right )^{6kp+12}\left(\frac{K_0}{16j^{2kp+4}}\right)^{3} = \frac{K_0^{3}}{2^{6kp+25}} = \frac{n\log^{30kp}n}{2^{6kp+25}}.\label{x8}
\end{eqnarray}

Now, since $\bar \xi_i$'s are symmetric and satisfy Condition 1 \eqref{cond1i}, \eqref{integral2} holds for $\bar\xi_i$ with $\bar d$ in place of $d$ and $6kp+12$ in place of $kp$ and gives 
\begin{eqnarray}\label{x8'}
\int_{\bar{\mathcal U}}\int_{\mathcal D}|\log|P_{\bar \xi}(z)|^{6kp+12}\le Cn\log^{6kp+12}n.
\end{eqnarray}
Now as $p=12$, the bounds \eqref{x8} and \eqref{x8'} provide a contradiction which then completes the proof of Lemma \ref{exceptionset-2.1}.
\end{proof}

\subsection{Finishing}\label{finishing}
Finally, we will combine the previous results and complete the proof of Theorem \ref{complex}.

Let $\varphi_0$ be a smooth function on $\mathbb C^{k}$ such that $\varphi_0(z_1, \dots, z_k) = z_1\dots z_k$ on $B(0, \delta^{-c_1})^{k}$, $ = 0$ outside of  $B(0, 2\delta^{-c_1})^{k}$, $|\varphi_0(z_1, \dots, z_k)| \le |z_1|\dots|z_k|$ for all $(z_1, \dots, z_k)\in \C^{k}$, and $\big|\triangledown ^a \varphi_0(\omega)\big| \le C\delta^{-kc_1}$
 for all $0\le a\le 3$. For example, $\varphi_0(z_1, \dots, z_k) = \prod_{i=1}^{k}z_i\phi \big(\frac{|z_i|}{\delta^{-c_1}}\big)$ for some smooth function $\phi$ such that $\phi$ is a smooth function such that $\text{supp}(\phi)\subset[-2,2]$, $0\le \phi\le 1$, and $\phi = 1$ on $[-1,1]$.

Since $X_j^{P}\le \delta^{-c_1}$ on $\mathcal T$, we have
\begin{eqnarray}
&&\E_{\xi}\left |\prod_{j=1}^{k}X_j^{P}-\varphi_0(X_1^{P}, \dots, X_k^{P})\right |=\E_{\xi}\left |\prod_{j=1}^{k}X_j^{P}-\varphi_0(X_1^{P}, \dots, X_k^{P})\right |\textbf{1}_{\mathcal{T}^{c}}\nonumber\\
&\le& 2\E_{_\xi} \left (\left |\prod_{j=1}^{k}X_j^{P}\right |\textbf{1}_{\mathcal{T}^{c}}\right )\le C'\delta^{1/22}\qquad\text{by \eqref{exceptionset}},\nonumber 
\end{eqnarray}
where by $\E_{\xi}$, we mean the expectation with respect to the random variables $\xi_0, \dots, \xi_n$.

From Proposition \ref{bigT} and \eqref{rc1}, we deduce that on the product space generated by the random variables $\xi_0, \dots, \xi_n$ and the random points $\check{w}_{j, i}$, the bound \eqref{rc1} holds with probability at least $1-\gamma_0-C\gamma(\delta)$. Thus
\begin{eqnarray}
&&\E_{\xi, \check w}\left|\varphi_0(X_1^{P}, \dots, X_k^{P}) - \varphi_0\left(\frac{\pi r_0^2}{m_0}\sum_{i=1}^{m_0} K_1^{P}(\check w_{1,i}), \dots, \frac{\pi r_0^2}{m_0}\sum_{i=1}^{m_0} K_k^{P}(\check w_{k,i})\right)\right|\nonumber\\
&\le& C \bigg(  \delta^{-kc_1}\frac{L^{4c_1}}{\sqrt {m_0\gamma_0}} +\delta^{-kc_1}\big(\gamma(\delta) +k\gamma_0\big)\bigg)\nonumber,
\end{eqnarray}
where $\E_{\xi, \check w}$ is the expectation on the product space. The first term bounds the contribution of the good event when \eqref{rc1} holds, and follows from the bound on the first derivative of $\varphi_0$. The second term bounds the contribution of the bad event when \eqref{rc1} fails and follows from the bound on the infinity norm of $\varphi_0$.

Note that $\gamma(\delta)\le 10\delta^{1/2}$ for all $\frac{1}{20n}\le \delta\le \frac{1}{C}$.

Let $c$ be any constant such that $0<c\le\min\{\frac{1}{22},\frac{\alpha_0}{2(3k+11)}, \frac{1}{2(k+1)}\}$ where $\alpha_0$ is the constant in Proposition \ref{logg-1}. Let $c_1 = c$, $m_0=\lfloor \delta^{-(3k+11)c}\rfloor$, and $\gamma_0 = \delta^{(k+1)c}$, then the above error term is $C\delta ^c$, and so
\begin{eqnarray}
\E_{\xi, \check w}\left |\prod_{j=1}^{k}X_j^{P} - \varphi_0\left(\frac{\pi r_0^2}{m_0}\sum_{i=1}^{m_0} K_1^{P}(\check w_{1,i}), \dots, \frac{\pi r_0^2}{m_0}\sum_{i=1}^{m_0} K_k^{P}(\check w_{k,i})\right)\right|&\le&C\delta^c\nonumber.
\end{eqnarray}

Now, applying Proposition \ref{logg-1} by first conditioning on the points $\check w_{j, i}$, we obtain
\begin{eqnarray}
&&\bigg |\E_{\xi, \check w} \varphi_0\left(\frac{\pi r_0^2}{m_0}\sum_{i=1}^{m_0} K_1^{P}(\check w_{1,i}), \dots, \frac{\pi r_0^2}{m_0}\sum_{i=1}^{m_0} K_k^{P}(\check w_{k,i})\right)\nonumber\\
&&\qquad-\E_{\tilde \xi, \check w} \varphi_0\left(\frac{\pi r_0^2}{m_0}\sum_{i=1}^{m_0} K_1^{\tilde P}(\check w_{1,i}), \dots, \frac{\pi r_0^2}{m_0}\sum_{i=1}^{m_0} K_k^{\tilde P}(\check w_{k,i})\right)\bigg |
\le C\delta^c.\nonumber
\end{eqnarray}

This completes the proof of Theorem \ref{complex}.

\section{Proof of real local universality}\label{proof-real}

In this section, we will prove Theorem \ref{real}. 

As before, we can assume without loss of generality that $\tilde \xi_i$ has Gaussian distribution for all $i$.

Let $r_0 = 10^{-2}/2$. In the following, we prove \eqref{h6}; the same proof works for $Q$ in place of $P$ unless otherwise noted.
As before, we reduce the problem to showing (\ref{h6}) for functions $G$ of the form 
\[G(y_1,\dots, y_k, w_1,\dots, w_l) = F_1(y_1)\dots F_k(y_k)G_1(w_1)\dots G_l(w_l),
\]
where $F_i:\mathbb{R}\to\mathbb{C}$ and $G_j:\mathbb{C}\to \mathbb{C}$ are smooth functions supported on $[-r_0, r_0]$ and $B(0, r_0)$ respectively, such that
 \[|{\triangledown^{a}F_i}(x)|, |{\triangledown^{a}G_j}(z)|\le 1
\]
for all $1\le i\le k, 1\le j\le l$, $x\in \R$, $z\in \C$, and $0\le a\le 3$.

Then, by the inclusion-exclusion argument and the symmetry of zeros of $P$ about the $x$-axis, we can further reduce the problem to showing that
\begin{eqnarray}
\ab{\E \left(\prod_{j=1}^{k}X_{\check{x}_i, F_i, \mathbb{R}}^{P}\right)\left(\prod_{j=1}^{l}X_{\check{z}_j, G_j, \mathbb{C}_+}^{P}\right)-\E \left(\prod_{j=1}^{k}X_{\check{x}_i, F_i, \mathbb{R}}^{\tilde P}\right)\left(\prod_{j=1}^{l}X_{\check{z}_j, G_j, \mathbb{C}_+}^{\tilde P}\right)}\le C\delta^{c},\label{du6}
\end{eqnarray}
where $X_{\check{x}_i, F_i, \mathbb{R}}^{P} = \sum_{\zeta_j^{\check P}\in\mathbb{R}}F_i(\zeta_j^{\check P}-\check x_i)$ and $X_{\check{z}_j, G_j, \mathbb{C}_+}^{P}= \sum_{\zeta_i^{\check P}\in\mathbb{C}_+}G_j(\zeta_i^{\check P}-\check z_j)$.

Since the proof of Theorem \ref{complex} (and in particular, \eqref{du5}) hardly changes if we replace $I(\delta)$ by $I(\delta) + (-10^{-6}\delta, 10^{-6}\delta)$, we conclude that there exists a positive constant $c$ for which
\begin{eqnarray}
\ab{\E \left(\prod_{j=1}^{m}X_{\check{w}_j, H_j}^{P}\right)-\E \left(\prod_{j=1}^{m}X_{\check{w}_j, H_j}^{P}\right)}\le C\delta^{c},\label{du2}
\end{eqnarray}
where $1\le m\le k+l$, $|w_j|\in I(\delta)+ (-10^{-4}\delta, 10^{-4}\delta)$, $H_j:\mathbb{C}\to \mathbb{C}$ is a smooth function supported in $B(0, 2r_0)$ and $|{\triangledown^aH_{j}}|\le 1, \forall 0\le a\le 3$, and 
$X_{\check{w}_j, H_j} = \sum_{i=1}^{n}H_j(\zeta_i^{\check P}-\check w_j)$. For the rest of the proof, we will write, for example, $X_{\check{w}_j, H_j}$ when it can be either $X_{\check{w}_j, H_j}^{P}$ or $X_{\check{w}_j, H_j}^{\tilde P}$.

We shall reduce \eqref{du6} to \eqref{du2} by first showing that the number of complex zeros near the real axis is small with high probability. This is the key lemma for this proof. We make use of a more classical tool, the Rouch\'e's theorem, together with some elegant arguments in \cite{HKPV1} and \cite{PV1}.

\begin{lemma}\label{sun6}
Let $c$ be as in \eqref{du2}. Let $\gamma = \delta^{c_2}$ where $c_2 = \min\{\frac{c}{100}, \frac{c}{3k+3l+1}, \frac{\rho+1/2}{4}\}$. There exists a constant $C$ such that for all $\frac{1}{20n}\le \delta\le \frac{1}{C}$, one has
$$\P  \left(N_{\check P}{B(\check x,\gamma)}\ge 2\right) \le C\gamma^{3/2},
\qquad\text{for all } x\in \R \mbox{ with } |x|\in  I(\delta) + (-10^{-4}\delta, 10^{-4}\delta).$$
When $\delta \ge \frac{1}{10n}$, the same statement holds for $Q$ in place of $P$.
\end{lemma}
The power $3/2$ in the above lemma is not critical, we only need something strictly greater than 1.

\begin{proof} 
We will prove the Lemma for $P$. The same arguments also work for $Q$ unless otherwise noted. The strategy is using Theorem \ref{complex} to reduce to Gaussian case. Let $H$ be a non-negative smooth function supported on $B(0, 2\gamma)$, which equals 1 on $B(0, \gamma)$ and is at most 1 everywhere else, and  $|\triangledown ^a H|\le C\gamma^{-a}$ for all $0\le a\le 8$. In particular, one can take $H(z) = \phi\left(\frac{z}{\gamma}\right)$ where $\phi$ is any smooth function supported in $B(0, 2)$ and equals $1$ on $B(0,1)$.

By Theorem \ref{complex}, we have
\begin{eqnarray}
\P (N_{\check P}{B(\check x, \gamma)}\ge 2)&\le&\E \sum_{i\neq j} H(\check \zeta_i-\check x)H(\check \zeta_j-\check x)\nonumber\\
&\le&\E \sum_{i\neq j} H(\check {\tilde\zeta}_i-\check x)H(\check {\tilde\zeta}_j-\check x)+ C\delta^c\gamma^{-8}\nonumber\\
&\le&\E \sum_{i\neq j} H(\check {\tilde\zeta}_i-\check x)H(\check {\tilde\zeta}_j-\check x)\textbf{1}_{ k\ge \delta^{-c_3}} + \E k(k-1)\textbf{1}_{ k< \delta^{-c_3}}+ C\gamma^{3/2}\nonumber\\
&\le&\E \sum_{i\neq j} H(\check {\tilde\zeta}_i-\check x)H(\check {\tilde\zeta}_j-\check x)\textbf{1}_{ k\ge \delta^{-c_3}} + \delta^{-2c_3}\P (k\ge 2)+ C\gamma^{3/2}.\nonumber
\end{eqnarray}

where $k = N_{\check {\tilde P}}{B(\check x, 2\gamma)}=N_{\tilde{P}}{B(x,2.10^{-3}\gamma\delta)}=: N_{\tilde{P}}{B(x,\eta)}$, and $c_3=c_2/10$. By Proposition \ref{nonclustering-1}, $k\le \delta^{-c_3}$ with probability at least $1-C\gamma(\delta)$. 

Using the result from Section \ref{negligible-set}, we have
\begin{equation}
\E \sum_{i\neq j} H(\check {\tilde\zeta}_i-\check x)H(\check {\tilde\zeta}_j-\check x)\textbf{1}_{ k\ge \delta^{-c_3}}  \le \E \left (\left(\sum_{i = 1}^n H(\check {\tilde\zeta}_i-\check x)\right)^2\textbf{1}_{ k\ge \delta^{-c_3}}\right) \le C\delta^{1/22}\gamma^{-8}\le C\gamma^{3/2}.
\end{equation}

Thus, it remains to show that $\P (k\ge 2) = \P(N_{\tilde{P}}{B(x,\eta)}\ge 2)\le C\delta^{2c_3}\gamma^{3/2}$. Having reduced the task to the Gaussian case, we will adapt the proofs of similar results in \cite{HKPV1} and \cite{PV2} to show it.

Consider $ g(z) = \tilde P(x) +\tilde P'(x)(z - x)$ and put $v_{z} = (c_iz^{i})_{i = 0}^{n} $. Let $p(z) =\tilde P(z) - g(z)$. Notice that for this Gaussian case, $\P(\tilde P(x) = 0)=0$ when $n$ is sufficiently large. Since $g$ is linear, it has at most one zero in $B(x,\eta)$. And hence, when $k\ge 2$, $\tilde P$ has more zeros than $g$ in that ball. If $|g(z)|>|p(z)|$ for all $z\in \partial B(x, \eta)$, then by Rouch\'e's theorem, $\tilde P$ and $g$ have the same number of zeros. Thus, for all $t>0$, we have
\begin{equation}
\P(k\ge 2)\le \P\left (\min_{z\in \partial B(x, \eta)}|g(z)|\le \max_{z\in \partial B(x, \eta)}|p(z)|\right ).\nonumber
\end{equation}
Let $A_1 = \{\omega: \min_{z\in \partial B(x, \eta)}|g(z)|\le \max_{z\in \partial B(x, \eta)}|p(z)|\}$. We will show that $\P(A_1)\le C\delta^{2c_3}\gamma^{3/2}$.

We have $p(z) = (\tilde \xi_i)_{i=0}^{n}(v_z - v_x - v'(x)(z-x))$ and 
\begin{eqnarray}
|(v_z - v_x - v'(x)(z-x))_{i}| &\le&\sup_{0\le \theta\le 1} \frac{1}{2}|c_i||z-x|^{2}i(i-1)|x + \theta z|^{i-2} \label{boundp}.
\end{eqnarray}
If $\delta\ge \frac{1}{10n}$, then for all $z\in \partial B(x, \eta)$ and $\theta\in [0,1]$, $|x + \theta z|\le 1-\frac{\delta}{2}$, and so by Condition 1,
\begin{eqnarray}
\Var (p(z))=|v_z - v_x - v'(x)(z-x)|^{2}\le \sum_{i=0}^{n}\eta^{4}i^{4} c_i^{2}\left (1-\frac{\delta}{2}\right )^{2i-4}
\le CL^{2\rho+1-4c_2}.\nonumber
\end{eqnarray}

Similarly, if $\frac{1}{20n}\le\delta\le \frac{1}{10n}$, then for all $z\in \partial B(x, \eta)$ and $\theta\in [0,1]$, $|x + \theta z|\le 1+\frac{3}{n}$, and so
\begin{eqnarray}
\Var (p(z))&\le& \sum_{i=0}^{n}\eta^{4}i^{4} c_i^{2}\left (1+\frac{3}{n}\right )^{2i-4} \le C\sum_{i=0}^{n}\eta^{4}n^{4+2\rho} \left (1+\frac{3}{n}\right )^{2n}\le CL^{2\rho+1-4c_2}.\nonumber
\end{eqnarray}
Thus, in any case, 
\begin{equation}
\Var (p(z))\le CL^{2\rho+1-4c_2}\label{varp}.
\end{equation}

(When proving the Lemma for $\tilde Q$, there are two cases: if $\rho \ge 0$ then observe from Condition 1 that $\left |\frac{d_i}{d_0}\right |\le C = Ci^0$ for all $i$, and so, by the same argument as above, for the function $p(z) = \tilde Q(z) - \tilde Q(x) - (z-x)\tilde Q'(x)$, one has $\Var (p(z))\le CL^{(2)(0) + 1-4c_2}=CL^{ 1-4c_2}$ which is similar to the case $\rho =0$ for $P$. Now, if $-\frac 1 2 <\rho <0$  we similarly have
\begin{eqnarray*}
\Var[p(z)] 
\le C\eta^4 \sum_{0\le i \le n/2} i^4  e^{-\delta i}  + C\eta^4 \sum_{n/2<i\le n} i^4 \frac{(n-i)^{2\rho}}{n^{2\rho}} e^{-\delta i}  
\le C L^{1-4c_2}
\end{eqnarray*}
which again is similar to the case $\rho=0$ for $P$.   We note that in all computation it is very important that $2\rho+1>0$ to ensure that the harmonic sum $\sum_{j\le M} j^{2\rho}$ is dominated by $M^{2\rho+1}$.)

We use the above estimate to prove that for every $t>0$, 
\begin{equation}
\P(\max_{z\in \partial B(x, \eta)} |p(z)-\E p(z)|\ge t)\le Ce^{-t^{2}/(CL^{2\rho+1-4c_2})}.\label{concenp}
\end{equation}

Indeed, let $\bar p(z) = p(z)-\E p(z)$, then for every $z\in \partial B(x, \eta)$, by Cauchy's integral formula, 
\begin{eqnarray}
|\bar p(z)|\le \int_0^{2\pi}\frac{|\bar p(x + 2\eta e^{\sqrt{-1}\theta})|}{|z - x - 2\eta e^{\sqrt{-1}\theta}|}2\eta\frac{d\theta}{2\pi}
\le\sqrt{CL^{2\rho+1-4c_2}} \int_0^{2\pi}\frac{|\bar p(x + 2\eta e^{\sqrt{-1}\theta})|}{\sqrt{\Var (\bar p(x + 2\eta e^{\sqrt{-1}\theta}))}}\frac{d\theta}{2\pi}\nonumber.
\end{eqnarray}

Hence, by Markov's inequality,
\begin{eqnarray}
\P(\max_{z\in \partial B(x, \eta)} |\bar p(z)|\ge t)
\le \E\left (\exp\left (\int_0^{2\pi}\frac{|\bar p(x + 2\eta e^{\sqrt{-1}\theta})|}{10\sqrt{\Var (\bar p(x + 2\eta e^{\sqrt{-1}\theta}))}}\frac{d\theta}{2\pi}\right )^{2}\right )e^{-t^{2}/(10^{2}CL^{2\rho+1-4c_2})}.\nonumber
\end{eqnarray}

Applying Jensen's inequality for convex functions $x\to x^{2}$ and $x\to e^{x}$ and Fubini's theorem gives 
\begin{eqnarray}
\E\left (\exp\left (\int_0^{2\pi}\frac{|\bar p(x + 2\eta e^{\sqrt{-1}\theta})|}{10\sqrt{\Var (\bar p(x + 2\eta e^{\sqrt{-1}\theta}))}}\frac{d\theta}{2\pi}\right )^{2}\right )
\le \int_0^{2\pi}\E\exp\left (\frac{|\bar p(x + 2\eta e^{\sqrt{-1}\theta})|^{2}}{100 {\Var (\bar p(x + 2\eta e^{\sqrt{-1}\theta}))}}\right )\frac{d\theta}{2\pi}\nonumber.
\end{eqnarray}

Let $z = x + 2\eta e^{\sqrt{-1}\theta}$ then the real part and imaginary part of $\frac{\bar p(z)}{\sqrt{\Var(\bar p(z))}}=: X_z+\sqrt{-1} Y_z$ are normally distributed with mean 0 and variance at most 1. Hence, by Cauchy-Schwartz inequality,
\begin{eqnarray}
\E e^{10^{-2}|X_z+\sqrt{-1} Y_z|^{2}} = \E e^{10^{-2}X_z^{2}}e^{10^{-2}Y_z^{2}}\le\E e^{2.10^{-2}X_z^{2}} + \E e^{2.10^{-2}Y_z^{2}}\le C.\nonumber 
\end{eqnarray}

That proves \eqref{concenp}.

Set 
\begin{equation}
 t = L^{\rho + 1/2-2c_2 + c_3}, \label{valuet}
 \end{equation} then \eqref{concenp} becomes 
\begin{equation}
\P\left (\max_{z\in \partial B(x, \eta)} |p(z)-\E p(z)|\ge \frac{1}{2}t\right )\le Ce^{-t^{2}/(4CL^{2\rho+1-4c_2})}\le \delta^{2c_3}\gamma^{3/2}.\label{concenp1}
\end{equation}
(To prove Lemma \ref{sun6} for $Q$, we set $t = L^{1/2-2c_2 + c_3}$.)

Let $A_2 = \{\omega:\max_{z\in \partial B(x, \eta)} |p(z)-\E p(z)|\ge \frac{1}{2}t\}$.

Now, since $g$ is a linear function with real coefficients, $P(x)$ and $P'(x)$, one has
$$\min_{z\in \partial B(x, \eta)} |g(z)| = \min |g(x \pm \eta)|.$$ 
And so, 
\begin{equation}
\P(\min_{z\in \partial B(x, \eta)} |g(z)|\le t)\le \P(|g(x+\eta)|\le t) + \P(|g(x-\eta)|\le t).\nonumber
\end{equation}

Since $g(x\pm \eta)$ is normally distributed, 
\begin{equation}
\P\left (|g(x\pm \eta)|\le t\right )\le \P(|g(x\pm\eta)-\E g(x\pm\eta)|\le t)\le \frac{t}{\sqrt{\Var(g(x\pm\eta))}}=\frac{t}{|v_x \pm \eta v'_x|}.\label{conceng}
\end{equation}

Using Condition 1, we have
\begin{eqnarray}
\eta |v'_x|\le C\eta\sqrt{\sum_{i=0}^{n}i^{2\rho + 2}x^{2i}} \le C\eta L^{\rho+3/2} = CL^{\rho+1/2-c_2},\label{bound-varg}
\end{eqnarray}

and 
\begin{eqnarray}
|v_x|\ge \frac{1}{C}\sqrt{\sum_{i=L/40}^{L/20}i^{2\rho}x^{2i}} \ge \frac{1}{C} L^{\rho+1/2} ,\label{winter}
\end{eqnarray}
which together give $|v_x-\eta v'_x|\ge \frac{1}{C} L^{\rho+1/2}$.

(To prove Lemma \ref{sun6} for $Q$ observe that $\left |\frac{d_i}{d_0}\right |\ge \frac{1}{C}$ for all $i\le \frac{n}{2}$ and hence for all $i\le L/20$; therefore by the same argument as above, for the vector field $v_z = \left (\frac{d_i}{d_0}z^{i}\right )_{i=0}^{n}$, one has $|v_x|\ge \frac{1}{C}L^{1/2}$, which is again similar to the case $\rho = 0$ for $P$; now for $\eta|v'_x|$ we similarly have
\begin{eqnarray*}
\eta |v_x'| &\le& C \eta \sqrt{\sum_{i=0}^n i^2 [(n-i)/n]^{2\rho} x^{2i} } 
\le C\eta \sqrt{\sum_{i\le n/2} i^2 x^{2i} + n^{2-2\rho} x^{n/2}\sum_{i>n/2}(n-i)^{2\rho}}\\
&\le& C \eta \sqrt{L^3+n^3 x^{n/2}} \le C \eta L^{3/2} = C L^{1/2-c_2}
\end{eqnarray*}
which is similar to the case $\rho=0$ for $P$. ) 

And so, by \eqref{valuet}, the bound \eqref{conceng} becomes
\begin{equation}
\P(|g(x\pm \eta)|\le t)\le \P(|g(x\pm\eta)-\E g(x\pm\eta)|\le t)\le CL^{-2c_2 + c_3} \le CL^{-3/2c_2-2c_3} = C\delta^{2c_3}\gamma ^{3/2}.\label{conceng'}
\end{equation}

Hence, 
\begin{equation}
\P(\min_{z\in \partial B(x, \eta)} |g(z)|\le t)\le C\delta^{2c_3}\gamma ^{3/2}.\label{min_g}
\end{equation}

Let $A_3 = \{\omega: \min_{z\in \partial B(x, \eta)} |g(z)|\le t\}$, and $A_4 = A_1 \setminus (A_2\cup A_3)$

If Condition 2 \eqref{cond-real'-1} holds, i.e., $\E\xi_i = 0$ for all $N_0\le i\le n$, then by \eqref{boundp}, $|\E p(z)|\le \eta^{2}\sum_{i=0}^{N_0} |\E\xi_i||c_i|i^{2}(1+3/n)^{n}\le C\eta^{2}\le \frac{t}{2}$ for every $z\in \partial B(x, \eta)$. This together with \eqref{concenp1} give $\P\left (\max_{z\in \partial B(x, \eta)} |p(z)|\ge t\right )\le \delta^{2c_3}\gamma^{3/2}$. And so $\P(A_1)\le \P(A_3) + \P\left (\max_{z\in \partial B(x, \eta)} |p(z)|\ge t\right )\le \delta^{2c_3}\gamma^{3/2}$ as desired.

(Similarly, for $Q$, one has $|\E p(z)|\le \eta^{2}\sum_{i=0}^{N_0} |\E\xi_i|\frac{|c_i|}{|c_n|} (n-i)^{2}(1-\frac{1}{2L})^{n-i}\le C\eta^{2}n^{2-\rho}e^{-n/2L}\le C\eta^{2}L^{2-\rho} = C L^{-\rho-2c_2}\le \frac{t}{2}$ for every $z\in \partial B(x, \eta)$ because $\rho>-1/2$.)

Similarly, if Condition 2 \eqref{cond-real'-2} holds, and $x< 0$, i.e., $x$ is in $-I(\delta) + (-10^{-4}\delta, 10^{-4}\delta)$. Recall that $\rho \ge 0$ under Condition 2 \eqref{cond-real'-2}. Then for every $z\in \partial B(x, \eta)$, 
\begin{eqnarray}	
|\E p(z)| &\le& C\eta^{2} + |\mu|\left|  \sum_{i=N_0}^{n}c_iz^{i} - \sum_{i=N_0}^{n}c_ix^{i} - (z-x)\sum_{i=N_0}^{n}ic_i x^{i-1}\right|\nonumber\\
&\le& C\eta^{2} + C\eta^{2} \max_{z' \in \partial B(x, \eta)} \left|  \sum_{i=0}^{n}\mathfrak P(i) i(i-1) z'^{i-2}\right|\nonumber,
\end{eqnarray}
in which we used the fact that the contributions of the sums from $i=0$ to $i = N_0-1$ are just $O(\eta^{2})$ as showed in the case of Condition 2 \eqref{cond-real'-1}.
Observe that $\mathfrak P(i)i(i-1) = \sum_{j=0}^{\rho+2}e_ji(i-1)\dots (i-j+1)$ for some constants $e_j$, and for each $0\le j\le \rho+2$, 
$$\left |\sum_{i=0}^{n} i(i-1)\dots (i-j+1)z'^{i-j}\right | = \left |\left (\frac{1-z'^{n+1}}{1-z'}\right )^{(j)}\right |\le Cn^{j}|z'|^{n-j+1}\le CL^{j}\frac{n^{j}}{L^{j}}e^{-n/L}\le CL^{j}\le CL^{\rho+2},$$
where in the first inequality, we used the bounds $|1-z'|\ge |1 - x|-|x-z'|\ge 1$.

This shows that $|\E p(z)|\le CL^{\rho -2c_2}|\mu|\le \frac{t}{2}$. From this, the same proof as for Condition 2 \eqref{cond-real'-1} applies.

(Similarly, for $Q$, one has 
\begin{eqnarray}	
|\E p(z)| \le C\eta^{2}L^{2} + C\eta^{2} \max_{z' \in \partial B(x, \eta)} \left|  \sum_{i=0}^{n}\frac{\mathfrak P(n-i)}{\mathfrak P(n)} i(i-1) z'^{i-2}\right|
= O(\eta^{2}L^{2})=O(L^{-2c_2})\le \frac{t}{2},\nonumber
\end{eqnarray}
in which, again, we used the fact that the contribution of the sums from $i = n-N_0$ to $i = n$ is bounded by $C\eta^{2}L^{2} $ as showed in the case of Condition 2 \eqref{cond-real'-1}.)

Now, if Condition 2 \eqref{cond-real'-2} holds, and $x\ge 0$, i.e., $x$ is in $I(\delta) + (-10^{-4}\delta, 10^{-4}\delta)$. Without loss of generality, assume that $\mu\ge 0$ and $c_i>0$ for all $i$ sufficiently large, say $i\ge N_0$(by replacing $c_i$ by $-c_i$ and $\xi_i$ by $-\xi_i$ if needed).

We have
\begin{equation}
\P(A_1)\le \P(A_2) + \P(A_3) + \P(A_4)\le C\delta^{2c_3}\gamma^{3/2} + \P(A_4).
\end{equation}

If $|\E p(z)|\le \frac{t}{2}$ for every $z\in \partial B(x, \eta)$, then as in the above case we also have $\P(A_1)\le \delta^{2c_3}\gamma^{3/2}$.

Otherwise, assume that there exists $z_0\in B(x, \delta)$ such that $|\E p(z_0)|> \frac{t}{2}$. Without loss of generality, we choose $z_0$ that maximizes $|\E p(z_0)|$ in that (closed) ball. Let $m(z) = \E P(z) = \sum_{i=0}^{n}c_i\E\xi_i z^{i}$. Then 
\begin{eqnarray}
|\E p(z_0)| &=& |m(z_0)-m(x)-m'(x)(z_0-x)| \le \frac {|z_0-x|^{2}}{2}\max_{z\in \partial B(x, \eta)}|m''(z)|\nonumber\\
&\le& o(t) + \mu\eta^{2}\sum_{i=0}^{i=n} c_i i(i-1)(x+\eta)^{i-2}\nonumber.
\end{eqnarray}
(For $P$ the $o(t)$ is $C\eta^2$, and for $Q$ the $o(t)$ is $C\eta^2 L^2$.)

Observe by a similar bound as in \eqref{boundp} that
\begin{eqnarray}
\left |\sum_{i=n\wedge 2(4+\rho)L\log L}^{n} c_i i(i-1)(x+\theta z_0)^{i-2}\right |&\le& C  L^{\rho+3}\int_{2(4+\rho)\log L}^{\infty} e^{-x/2}dx\le 1.\nonumber
\end{eqnarray}

Hence, 

$ \frac{t}{2}<|\E p(z_0)|\le o(t) + \eta^{2}\sum_{i=0}^{i=n\wedge 2(4+\rho)L\log L} c_i i^{2}(x+\eta)^{i-2}\le 2\eta^{2}\sum_{i=0}^{i=n\wedge 2(4+\rho)L\log L} c_i i^{2}(x+\eta)^{i-2}$.

Now, 
\begin{eqnarray}
m(x) &\ge&\mu \sum_{i=0}^{i=n\wedge 2(4+\rho)L\log L} c_i x^{i} - o(t) \ge \frac{1}{C}\mu L^{-2}\log^{-2} L\sum_{i=0}^{i=n\wedge 2(4+\rho)L\log L} c_i i^{2}(x+\eta)^{i}-o(t)\nonumber\\
&\ge& \frac{1}{C}\frac{L^{2c_2}}{\log^{2} L}\eta^{2} \qquad \sum_{i=0}^{i=n\wedge 2(4+\rho)L\log L} c_i i^{2}(x+\eta)^{i-2}-o(t)\nonumber\\
&\ge& \frac{1}{C}\frac{L^{2c_2}}{\log^{2} L}|\E p(z_0)| \quad \ge \quad \frac{1}{C}\frac{L^{2c_2}}{\log^{2} L}t \quad =\quad \frac{1}{C}\frac{L^{\rho+1/2+c_3}}{\log^{2} L}.\label{boundm}
\end{eqnarray}

Similarly, 
\begin{eqnarray}
\eta  m'(x)\le C+\eta\sum_{i=0}^{i=n\wedge 2(4+\rho)L\log L} c_i ix^{i-1}\le C + C\eta L (\log L) m(x)\le C\frac{\log L}{L^{c_2}}m(x)\le \frac{m(x)}{2}.\nonumber
\end{eqnarray}

Thus, 
\begin{equation}
\E g(x\pm \eta) = m(x)\pm \eta m'(x)\ge \frac{m(x)}{2}.\label{bound-mg}
\end{equation}

By this and \eqref{bound-varg} and its analog for $v_x$ show that 
\begin{equation}
\sqrt {\Var g(x\pm \eta)}\le CL^{\rho + 1/2} \le \frac{\E g(x\pm \eta)}{L^{c_3/2}}.\label{bound-vmg}
\end{equation}

On the event $A_4$, we know that $\min |g(x\pm\eta)|\le \max_{z\in \partial B(x, \eta)} |p(z)|$. Choose any $z$ in the closed ball $cl(B(x, \eta))$ that maximizes $|p|$. Then, $\min |g(x\pm \eta)|\le |p(z)|$. Since $A_4\cap A_2 = \emptyset$, $|p(z)|\le |\E p(z_0)| + t/2\le 2|\E p(z_0)|$. Then, by \eqref{boundm} and \eqref{bound-mg}, 
\begin{equation}
\min |g(x\pm \eta)|\le |p(z)|\le C\frac{\log ^{2}L}{L^{2c_2}}\min |\E g(x\pm \eta)|\le \frac{1}{2}\min \E g(x\pm \eta).\label{dduu}
\end{equation}

Finally, by \eqref{bound-vmg}, we have
\begin{eqnarray}
\P\left (|g(x\pm \eta)|\le \frac{1}{2}\E g(x\pm \eta)\right )&\le& \P\left (\frac{|g(x\pm \eta) - \E g(x\pm \eta)|}{\sqrt{\Var (g(x\pm \eta))}}\ge \frac{\E g(x\pm \eta)}{2\sqrt{\Var (g(x\pm \eta))}}\right )\nonumber\\
&\le& \P\left (|N(0, 1)|\ge \frac{L^{c_3/2}}{2}\right )\le \delta^{2c_3}\gamma^{3/2}.\nonumber
\end{eqnarray}
This proves \eqref{dduu} and thus $\P(A_4)\le C\delta^{2c_3}\gamma^{3/2}$. So is $A_1$.
\end{proof}

Now, for every $1\le i\le k$, consider the strip $S_i = [\check x_i - r_0, \check x_i + r_0]\times [-\gamma/4, \gamma/4]$. We can cover $S$ by $O(\gamma^{-1})$ balls of the form $B(\check x, \gamma)$ where $x \in [\check x_i-r_0, \check x_i+r_0]$. Using Lemma \ref{sun6}, we obtain
\begin{eqnarray}
\P (\text{there is at least 1 (or equivalently 2) root in } S_i\backslash \mathbb R )
&=&  O(\gamma^{-1}\gamma^{3/2}) = O(\gamma^{1/2})\label{du4}.
\end{eqnarray}

Consider 
$\hat F_i(z) = F_i(Re(z))\phi \left (\frac{4Im(z)}{\gamma}\right )$, 
where $\phi$ is a bump function on $\mathbb R$ that is supported on $[-1,1]$ and is $1$ at $0$. Then $\hat F_i$ is a smooth function supported on $S_i - \check x_i$ and $|\hat F_i|\le 1$, and $\big|\triangledown ^a \hat F_i\big| = O(\gamma^{-a})$ for $0\le a\le 3$.

 Set 
 $ X_{\check x_i, \hat F_i} = \sum_{j=1}^n \hat F_i(\zeta_j^{\check P} - \check x_i)\quad\mbox{ and } D_{\check x_i, F_i} =  X_{\check x_i,\hat F_i} - X_{x_i, F_i, \R} =\sum_{\zeta_i^{\check P}\notin \R} \hat F_i(\zeta_i^{\check P}- \check x_i)$.

Observe that $|D_{\check x_i, F_i}| \le N_{\check P}{B(\check x_i, 2r_0)}$, and from \eqref{du4}, $D_{\check x_i, F_i} = 0$ with probability at least $1 - O(\gamma^{1/2})$.

Let $\phi_0$ be a bump function supported on $B(0, 4r_0)$ that equals 1 on $B(0, 2r_0)$ and $|\triangle^{a} \phi_0|\le C$ for all $0\le a\le 3$, then
\begin{equation}
\max\{|X_{\check x_i, \hat F_i}|, | X_{x_i, F_i, \R}|, |D_{\check x_i, F_i}|\}\le \sum_{j=1}^{n} \phi_0(\zeta_j^{\check P} - \check x_i)=: X_{\check x_i, \phi_0}.\nonumber
\end{equation}

Let $c_4 = \frac{c_2}{4(k+l)^{2}}$. By Proposition \ref{nonclustering-1}, $N_{\check P}{B(\check x_i, 2r_0)} =N_{ P}{B( x_i, 2r_010^{-3}\delta)} \le \delta^{-c_4}$ with probability at least $1-C\gamma(\delta)$. And from Section \ref{negligible-set}, we have $\E\left (|X_{\check x_i, \phi_0}|^{k+l}\textbf{1}_{ N_{\check P}{B(\check x_i, 2r_0)} > \delta^{-c_4}}\right )\le C\delta^{1/22}$.

Hence,
\begin{eqnarray}
\E \ab{{X}_{\check x_i,\hat F_i}-X_{\check x_i, F_i, \R}}^{k+l}&=& \E (|D_{\check x_i, F_i}|^{k+l}\textbf{1}_{N_{\check P}{B(\check x_i, 2r_0)}\le \delta^{-c_4}}) + \E (|D_{\check x_i, F_i}|^{k+l}\textbf{1}_{ N_{\check P}{B(\check x_0, 2r_0)} > \delta^{-c_4}})\nonumber\\
&\le& C\delta^{-c_4(k+l)}\gamma^{1/2} + \E\left (|X_{\check x_i, \phi_0}|^{k+l}\textbf{1}_{ N_{\check P}{B(\check x_i, 2r_0)} > \delta^{-c_4}}\right )\le C\delta^{c_4(k+l)^{2}}\nonumber.
\end{eqnarray}

Moreover, by another application of Proposition \ref{nonclustering-1} and Section \ref{negligible-set}, we obtain
\begin{eqnarray}
&&\max\{\E |{X}_{\check x_i,\hat F_i}|^{k+l}, \E |{X}_{\check x_i, F_i, \R}|^{k+l}\}\nonumber\\
&=& \E (|X_{\check x_i, \phi_0}|^{k+l}\textbf{1}_{N_{\check P}{B(\check x_i, 4r_0)}\le \delta^{-c_4}}) + \E (|X_{\check x_i, \phi_0}|^{k+l}\textbf{1}_{ N_{\check P}{B(\check x_0, 4r_0)}> \delta^{-c_4}})\nonumber\\
&\le& C\delta^{-c_4(k+l)} + C\delta^{1/22}\le C\delta^{-c_4(k+l)}\nonumber.
\end{eqnarray}

Similarly, for each $1\le j\le l$, let $\hat G_j(z) = G_j(z)\eta(\text{Im}(z+\check z_j)/\gamma)$ where $\eta$ is a bump function on $\R$ supported on $[1/2, \infty)$ and equal 1 on $[1, \infty)$. And let $X_{\check z_j, \hat G_j} = \sum_{i=1}^{n}\hat G_j(\zeta_i^{\check P} - \check z_j)$. Then $\E|X_{\check z_j, \hat G_j} - X_{\check z_j,  G_j, \C_{+}}|^{k+l}\le C\delta^{c_4(k+l)^{2}}$ and $\max\{\E |{X}_{\check z_j,\hat G_j}|^{k+l}, \E |{X}_{\check z_j, G_j, \C_{+}}|^{k+l}\}\le C\delta^{-c_4(k+l)}$.

By telescoping the difference and applying H\H{o}lder's inequality, we obtain
\begin{equation}
\E\left |(\prod_{i=1}^{k}X_{\check x_i, F_i, \R})(\prod_{j=1}^{l}X_{\check z_j, G_j, \C_{+}}) - (\prod_{i=1}^{k}X_{\check x_i, \hat F_i})(\prod_{j=1}^{l}X_{\check z_j, \hat G_j} )\right |\le C\delta^{c_4}.\nonumber 
\end{equation}

Combining this with \eqref{du2} with $H_j$'s being $\hat F_i/O(\gamma^{-3})$ and $\hat G_j/O(\gamma^{-3})$, respectively, we get the desired result.

\section{Proof of Lemma \ref{boundedness} and Corollary \ref{mean}}\label{proof-mean}


\begin{proof}[Proof of Lemma \ref{boundedness}] 
If suffices to show that
\begin{equation}\label{d1}
\E N_{P_n}\left ([- 1 + \frac{1}{C}, 1-\frac{1}{C}]\right )\le M(C)\mbox{, and }
\E N_{Q_n}\left ([ -1 + \frac{1}{C}, 1-\frac{1}{C}]\right )\le M(C),
\end{equation}
for some constant $M(C)$.

Again, the proof for the second inequality is the same as the first. We follow the approach in \cite{IM2}.
By \eqref{dq}, we showed that there exist constants $d$ and $q$ such that $\P(|\xi_i|\le d)\le q< 1$ for all $i$.

For $k\ge N_0$, let $B_k = \{\omega: \ab{\xi_{N_0}}\le d,\dots, \ab{\xi_{k-1}}\le d, \ab{\xi_k}>d\}$.
Then $\P (B_k)\le q^{k-N_0}$.

By mean value theorem and Jensen's inequality, we have  
\begin{eqnarray}
N_{P}{[-1+\frac{1}{C}, 1-\frac{1}{C}]}\le k +N_{ P^{(k)}}[-1+\frac{1}{C}, 1-\frac{1}{C}]\le  k +\frac{\log \frac{M}{|P^{(k)}(0)|}}{\log \frac{R}{r}}\nonumber
\end{eqnarray}
where $R = 1 - {\frac{1}{2C}}, r = 1-\frac{1}{C}$, and $M = \sup _{|z| = R} \ab{P^{(k)}(z)}$. On $B_k$, we have 
\begin{eqnarray}
N_{P}{[-1+\frac{1}{C}, 1-\frac{1}{C}]}
&\le&k +\frac{\log \frac{\sum_{j=k}^{n}c_{jk}\frac{|c_j|}{|c_k|}|\xi_j|}{d}}{\log \frac{R}{r}}\nonumber,
\end{eqnarray}
where $c_{jk} = j(j-1)\dots(j-k+1)R^{j-k}/k!$. And so,
\begin{eqnarray}
\E N_{P}{[-1+\frac{1}{C}, 1 - \frac{1}{C}]}&\le& \sum_{k=N_0}^{n+1} k\P(B_k) + \frac{1}{\log\frac{R}{r}} \sum_{k=N_0}^{n+1} \int_{B_k} \log\left(\sum_{j=k}^{n}c_{jk}\frac{|c_j|}{|c_k|}|\xi_j|\right)\text{d}\P + \frac{\log 1/d}{\log\frac{R}{r}}\sum_{k=N_0}^{n+1} \P(B_k).\nonumber
\end{eqnarray}

Thus, to show \eqref{d1}, it suffices to show that
\begin{equation}\label{c1}
\sum_{k=N_0}^{n+1} \int_{B_k} \log\left(R_k\right)\text{d}\P\le C',
\end{equation}
for some constant $C' = C'(C)$, where $R_k = (\rho + 1+k)^{-\rho-1}\sum_{j=k}^{n}c_{jk}\frac{|c_j|}{|c_k|}|\xi_j|$. Then 
\begin{eqnarray}\label{c2}
\E R_{k} \le  C'(\rho + 1+k)^{-\rho-1}\sum_{j=k}^{n}{j \choose k} \frac{|c_j|}{|c_k|} R^{j-k}
\le \frac{C' k^{-\rho}}{(1-R)^{k+\rho+1}}.\nonumber
\end{eqnarray} 
%
%
%
%
%
%
Let $B_{ki} = \{\omega\in B_k: e^{i}\E R_{k}< R_k\le e^{i+1}\E R_{k}\}$. Then $\P(B_{ki})\le e^{-i}$ by Markov's inequality. Let $i_0 = \lfloor-\log q^{k}\rfloor$, then
\begin{eqnarray}
&&\int_{B_k}\log R_k\text{d}\P \le \P(B_k)\log \left(e^{i_{0}}\E R_{k}\right)+\sum_{i = i_{0}}^{\infty} \int_{B_{ki}}\log R_k\text{d}\P\le C'q^{k}(k+2+\rho-k\log q -\rho \log k)\nonumber.
\end{eqnarray}
%
%
%
%
This proves (\ref{c1}) and completes the proof.
\end{proof}

\begin{proof}[Proof of Corollary \ref{mean}] Let $C$ be the constant in Theorem \ref{real} with $k=1$. As a consequence of the above lemma, we only need to concentrate on the domain $\R\cap A(0, 1-\frac{1}{C}, 1+\frac{1}{C})$.

Let $c$ be the constant in Theorem \ref{real}, and let $\alpha = c/7$. We will prove that for every $\frac{1}{20n}\le \delta\le \frac{1}{C}$ and real number $x_0$ such that $|x_0|\in I(\delta)$, we have 
\begin{equation}
\left |\E N_P (x_0-10^{-7}\delta, x_0 + 10^{-7}\delta) - \E N_{\tilde P} (x_0-10^{-7}\delta, x_0 + 10^{-7}\delta) \right | = O(\delta^{\alpha/2}).\label{dd1}
\end{equation}
and when $\frac{1}{10n}\le \delta\le \frac{1}{C}$,
\begin{equation}
\left |\E N_P (x_0-10^{-7}\delta, x_0 + 10^{-7}\delta) - \E N_{\tilde P} (x_0-10^{-7}\delta, x_0 + 10^{-7}\delta) \right | = O(\delta^{\alpha/2}).\label{dd1h}
\end{equation}

From \eqref{dd1}, we can conclude that $\left |\E N_P (\pm I(\delta)) - \E N_{\tilde P} (\pm I(\delta)) \right | = O(\delta^{\alpha/2})$ for all $\frac{1}{20n}\le \delta\le \frac{1}{C}$. Letting $\delta = \frac{1}{20n}, \frac{1}{10n}, \dots, \frac{2^{m}}{20n}$ where $\frac{2^{m-1}}{20n}< \frac{1}{C}\le \frac{2^{m}}{20n}$ and applying triangle inequality, we obtain $\left |\E N_P (\pm (1-\frac{2^{m+1}}{n}, 1+\frac{1}{n})) - \E N_{\tilde P} (\pm (1-\frac{2^{m+1}}{n}, 1+\frac{1}{n}))  \right | = O(1)$. This together with the analogue for $Q$ give the desired result. (By definition of $Q$ we have $\E N_Q[a,b]= \E N_P[1/b,1/a]$ if $0\le a<b \le  \infty$ or $-\infty\le a<b \le 0$.)

As for the proof of \eqref{dd1}, let $\frac{1}{20n}\le \delta\le \frac{1}{C}$ and let $x_0$ be a real number with $|x|\in I(\delta)$.

Let $G$ be a smooth function supported on $[-10^{-4}-\delta^\alpha, 10^{-4}+\delta^\alpha]$ such that $0\le G\le 1$, $G = 1$ on $[-10^{-4}, 10^{-4}]$, and $\norm{\triangledown^a G} \le C\delta^{-6\alpha}$ for all $0\le a \le 6$. We have
\begin{eqnarray}
\E N_{P}{[x_0-10^{-7}\delta, x_0+10^{-7}\delta]}& =& \E N_{\check P}{[\check x_0-10^{-4}, \check x_0+10^{-4}]} \le \E \sum_{\zeta_i^{\check P}\in\mathbb{R}} G(\zeta_i^{\check P} - \check x_0)\nonumber\\
&\le& \E \sum_{\zeta_i^{\check {\tilde P}}\in\mathbb{R}} G(\zeta_i^{\check {\tilde P}} - \check x_0) + C\delta^{c-6\alpha}\qquad\text{by Theorem \ref{real}}\nonumber\\
&\le& \E \sum_{i=1}^{n} \mathbf{1}_{[-\delta^\alpha-10^{-4}, \delta^\alpha+10^{-4}]}(\zeta_i^{\check {\tilde P}} - \check x_0) + C\delta^{c-6\alpha}\nonumber\\
&\le & \E N_{\tilde P}{[x_0-10^{-7}\delta, x_0+10^{-7}\delta]}+ \mathcal I_{\tilde{P}} + C\delta^{\alpha}\nonumber,
\end{eqnarray}

where 
$\mathcal I_{\tilde{P}} = \E \sum_{i=1}^{n} \mathbf{1}_{\pm[10^{-7}\delta,10^{-7}\delta+10^{-3}\delta^{\alpha+1}]}(\zeta_i^{ {\tilde P}}-x_0)
$. We will show later that $\mathcal I_{\tilde P} = O(\delta^{\alpha/2})$.

Thus, \[\E N_{P}{[x_0-10^{-7}\delta, x_0+10^{-7}\delta]}\le \E N_{\tilde P}{[x_0-10^{-7}\delta, x_0+10^{-7}\delta]} + C\delta^{\alpha/2}.\]
By similar arguments with the function $G$ being replaced by one supported on $[-10^{-4}, 10^{-4}]$ such that $0\le G\le 1$ and $G = 1$ on $[-10^{-4}+\delta^\alpha, 10^{-4}-\delta^\alpha]$, we have
 \[\E N_{P}{[x_0-10^{-7}\delta , x_0+10^{-7}\delta]}\ge \E N_{\tilde P}{[x_0-10^{-7}\delta, x_0+10^{-7}\delta]} - C\delta^{\alpha/2}.\]

This gives \eqref{dd1} for $P$. Hence, to finish, we only need to prove the stated bound on $\mathcal I_{\tilde{P}}$ and $\mathcal I_{\tilde{Q}}$. Let $[a, b] = x_0 \pm[10^{-7}\delta,10^{-7}\delta+10^{-3}\delta^{\alpha+1}]$. By a Kac-Rice type formula (see, for instance, \cite[Theorem 2.5]{F1}), one has 
\begin{eqnarray}\label{nh1}
\E N_{\tilde P}[a, b] &\le& \int_a^{b}\sqrt{\frac{\mathcal S}{\mathcal P^{2}}}dt +\int_a^{b}\frac{|m'|\mathcal P + |m|\mathcal R}{\mathcal P^{3/2}}e^{-\frac{1}{2}\left (\frac{m}{\sqrt{\mathcal P}}\right )^{2}}dt,
\end{eqnarray}
for any $a\le b$, where $m(t) = \E\tilde{P}(t), \mathcal P=\Var (\tilde P)=\sum_{i=0}^{n} c_i^{2}t^{2i}$, $\mathcal Q=\Var(\tilde P')=\sum_{i=0}^{n} c_i^{2}i^{2}t^{2i-2}$,  $\mathcal R = \textbf{Cov}(\tilde P, \tilde P')=\sum_{i=0}^{n} c_i^{2}it^{2i-1}$, and $\mathcal S = \mathcal P\mathcal Q - \mathcal R^{2}=\sum_{i<j}(j-i)^{2}c_i^{2}c_j^{2}t^{2i+2j-2}$.


First, we will bound the second integral. By similar bounds as in \eqref{varp} and \eqref{winter}, we have for every $t\in [a,b]$,
\begin{equation}
\mathcal P\ge \frac{1}{C}\delta^{-2\rho-1},\nonumber
\end{equation}
and	
\begin{equation}
\mathcal R \le C\delta^{-2\rho-2}\le C\delta^{-1}\mathcal P.\nonumber
\end{equation}

Additionally, by the same argument as in the proof of Theorem \ref{real} (more precisely Lemma~\ref{sun6} near the estimate \eqref{boundm}), one can show that under Condition 2 \eqref{cond-real'}, for every $t\in [a, b]$, $|m'(t)|\le C\delta^{-\rho-1} + C\delta^{-1}\log\frac{1}{\delta}|m(t)|$. Thus, 
\begin{eqnarray}
\frac{|m'|\mathcal P + |m|\mathcal R}{\mathcal P^{3/2}}e^{-\frac{1}{2}\left (\frac{m}{\sqrt{\mathcal P}}\right )^{2}}\le C\delta^{-1/2}+C\delta^{-1}\log\frac{1}{\delta}\frac{|m|}{\sqrt{\mathcal P}}e^{-\frac{1}{2}\left (\frac{m}{\sqrt{\mathcal P}}\right )^{2}}\le C\delta^{-1}\log\frac{1}{\delta}\nonumber
\end{eqnarray}
where in the last inequality, we used the boundedness of the function $x\to xe^{-x^{2}/2}$ on $\R$. Since the length of the interval $[a, b]$ is $C\delta^{\alpha+1}$, the second integral in \eqref{nh1} is of order $O(\delta^{\alpha/2)}$ as desired.

Hence, it remains to bound the first integral in \eqref{nh1}. By symmetry, we may assume that $a>0$. We first reduce to the hyperbolic polynomials for which that integral is easier to handle. Consider the corresponding hyperbolic polynomials with coefficients $c_i^{hyper}=\sqrt{\frac{(2\rho+1)\dots (2\rho+i)}{i!}}$. A routine estimation shows that $\frac{1}{C}i^{\rho}\le c_i^{hyper}\le Ci^{\rho}$ for some constant $C$. And thus, by condition \eqref{cond-c}, $\frac{1}{C'}c_i^{hyper}\le |c_i|\le C'c_i^{hyper}$ for all $N_0\le i$, and so when $|t|\ge \frac{1}{2}$, one has $\mathcal S(t)\le C'\mathcal S^{hyper}(t)$ and $\mathcal P(t)\ge \frac{1}{C'}\mathcal P^{hyper}(t)$. Thus, $\sqrt{\frac{\mathcal S}{\mathcal P^{2}}}\le C'\sqrt{\frac{\mathcal S^{hyper}}{(\mathcal P^{hyper})^{2}}}$. 

If $\frac{1}{2}\le t\le 1-\frac{(100\rho+100)\log n}{n}$,  one has 
$\mathcal P^{hyper}(t) = \frac{1}{(1-t^2)^{2\rho+1}} - \sum_{i=n+1}^{\infty}\frac{(2\rho+1)\dots(2\rho+i)}{i!}t^{2i}$, and the last term is bounded from above by $\sum_{i=1}^{\infty}\frac{(2\rho+1)\dots(2\rho+i)}{i!}t^{2i}A_i$ where $A_i = \frac{(2\rho+i+1)\dots(2\rho+i+n)}{(i+1)\dots (i+n)}t^{2n} \le \frac{(i+n+1)\dots(i+n+\lceil 2\rho\rceil)}{(i+1)\dots(i+\lceil 2\rho\rceil)}t^{2n}=o( n^{-100\rho-100})$. Thus, $\mathcal P^{hyper} = \frac{1}{(1-t^{2})^{2\rho+1}}(1 + o( n^{-100\rho-100}))$.
Similar calculations for $\mathcal Q$ and $\mathcal R$ reveal that $\mathcal Q = \left (\frac{(2\rho+1)(2\rho+2)t^{2}}{(1-t^{2})^{(2\rho+3)}} + \frac{2\rho+1}{(1-t^{2})^{2\rho+2}}\right )[1 + o( n^{-100\rho-100})]$ and $\mathcal R = \frac{(2\rho+1)t}{(1-t^{2})^{2\rho+2}}[1+o( n^{-100\rho-100})]$, therefore
$$\sqrt{\frac{\mathcal S^{hyper}}{(\mathcal P^{hyper})^{2}}} 
= \frac{\sqrt{2\rho+1}}{1-t^{2}}\left (1+O(n^{-12\rho-12})\right )$$ 
Plugging into \eqref{nh1} with $[a, b]=x_0\pm[10^{-7}\delta,10^{-7}\delta+10^{-3}\delta^{\alpha+1}]$ gives the desired bound for $\delta\ge (200\rho+200)n^{-1}\log n$.

Next, if $1+\frac{2}{n}\ge t\ge 1-\frac{(500\rho+500)\log n}{n}$, we will prove that 
\begin{equation}\label{d2}
\frac{\mathcal S}{\mathcal P^{2}}\le \frac{O(n)}{|1-t|}.
\end{equation}
This together with \eqref{nh1} will give the desired bound for $\delta\le \frac{(200\rho+200)\log n}{n}$.

To prove \eqref{d2}, observe that $\mathcal S\le 4\sum_{0\le i<j\le n}(j-i)^{2}c_i^{2}c_j^{2}t^{2i+2j}$. Set $M = \frac{1}{|1-t|}$. We have
\begin{equation}
\sum_{0\le i<j\le n\wedge (i+\sqrt{nM})}(j-i)^{2}c_i^{2}c_j^{2}t^{2i+2j}\le \frac{n}{|1-t|}\mathcal P^{2}(t).\label{d4}
\end{equation}
And so, we only need to work on the summands corresponding to $0\le i\le i + \sqrt{nM}<j\le n$. In particular, we can assume that $M< n$. If $1-2/n\le t\le 1+2/n$, then $\frac{1}{|1-t|}\ge \frac{n}{2}$ and so, \eqref{d2} follows by a similar argument to \eqref{d4}. Thus, we can further assume that $t<1 - \frac{2}{n}$.

For each $\sqrt{nM}<j\le n$, we have from \eqref{cond-c},
\begin{eqnarray}
\sum_{i=0}^{\lfloor j-\sqrt{nM}\rfloor}(j-i)^{2}c_i^{2}c_j^{2}t^{2i+2j} &=& O\left (j^{2}c_j^{2}c_{\lfloor j-\sqrt{nM}\rfloor}^{2}t^{2j}\sum_{i=0}^{\infty}t^{2i}\right )  =O\left (n^{2\rho+2}c_{\lfloor j-\sqrt{nM}\rfloor}^{2}Mt^{2j}\right )\label{d5}.
\end{eqnarray}

We will now show that 
\begin{equation}
n^{2\rho+2}c_{\lfloor j-\sqrt{nM}\rfloor}^{2}Mt^{2j}= \frac{O(n)}{1-t}\mathcal P(t)c_{\lfloor j-\sqrt{nM}\rfloor}^{2}t^{2\lfloor j-\sqrt{nM}\rfloor},\label{d3}
\end{equation}

which is equivalent to $n^{2\rho+1}= O(1) \mathcal P(t)(1-\frac{1}{M})^{-2\sqrt{nM}}$ for some constant $C_3$. This is true because the right-hand side is at least $(\sum_{i=\lceil M/2\rceil}^{M}c_i^{2}(1-\frac{1}{M})^{2i})e^{2\sqrt{n/M}} \gg M^{2\rho+1}e^{2\sqrt{n/M}}\gg n^{2\rho+1}$ because we assumed that $M = \frac{1}{1-t}<n$.

From \eqref{d4}, \eqref{d5}, and \eqref{d3}, we obtain \eqref{d2}.

The proof for \eqref{dd1h} follows nearly the same lines with $\mathcal I_{\tilde{Q}} = \E \sum_{i=1}^{n} \mathbf{1}_{\pm[10^{-7}\delta,10^{-7}\delta+10^{-3}\delta^{\alpha+1}]}(\zeta_i^{ {\tilde Q}}-x_0)$ and as in the proof of Lemma \ref{sun6}, the estimates for $\tilde Q$ will be similar to the case $\rho=0$ for $\tilde P$. In particular the handling of the corresponding second integral of \eqref{nh1} is similar (exploiting ingredients from the proof of Lemma~\ref{sun6}), and for the first integral we may upper bound it by that of Kac polynomials 
by comparing $c_j$ with the hyperbolic coefficients and using Lemma~\ref{l.density-pw-outer}. This completes the proof.

%
\end{proof}

\section{Proof of complex local universality for series}\label{proof-complex-series}

\begin{proof}[Proof of Theorems \ref{complex-series} and \ref{real-series}]
First, let us make some observations about the series $P_{PS}$ under Condition 1. 
\begin{enumerate}
\item $P_{PS}$ converges uniformly in every compact set in $\D$ a.e.

Indeed, let $\Omega_n = \{\omega: |\xi_i(\omega)|\le n+i^{n},\forall i\ge 0\}$. Then $\Omega_1\subset \Omega_2\dots\subset \Omega_n\dots$, 
and  $\Omega = \bigcup_{n=1}^{\infty}\Omega_n$. 
On each $\Omega_n$, $P_{PS}$ converges uniformly on compact sets in $\D$.

Moreover, $P_{PS}$ does not extend analytically to any domain larger than the unit disk (see, for instance, \cite[Lemma 2.3.3]{HKPV1}).

\item By the Lebesgue's dominated convergence theorem, $\Var(P_{PS}(z)) = \sum_{n=0}^{\infty} |c_n|^{2}|z|^{2n}$.


\item For every $0<\delta\le 1-\frac{1}{C}$, $z\in A(0, 1-2\delta, 1-\delta]$, $k\ge 1$, one has 
\begin{equation}
 \E [N_{P_{PS}}(B(z, \delta/10))]^{k}<\infty.\nonumber
 \end{equation} 
This follows from Proposition \ref{nonclustering-1} by setting $\lambda = 2^{n}$ with $n = 1, 2, 3, \dots$ and shows that the integrals in the statements of Theorems \ref{complex-series} and \ref{real-series} are well-defined.
\end{enumerate}

Now, the proofs of Theorems \ref{complex-series} and \ref{real-series} for any $0<\delta\le \frac{1}{C}$ follow exactly the same lines as the proofs of Theorems \ref{complex} and \ref{real} for the case $\frac{\log^{2}n}{n}\le \delta\le \frac{1}{C}$ with the $n$ in the latter proofs being replaced by $\infty$.
\end{proof}

\begin{proof}[Proof of Corollary \ref{series-rotation}]
The Corollary follows from Theorem \ref{complex} with the two sequences of random variables $(\xi_n)$ and $(\xi_ne^{\sqrt{-1}n\theta})$.
\end{proof}

\begin{proof}[Proof of Corollary \ref{series-isom}] 
Observe that by the change of variables formula, with respect to the rescaling formula \ref{rescale}, one has
\begin{equation}
\rho_{\check P}^{(k)}(w_1, \dots, w_k) = (10^{-3}\delta)^{2k}\rho_P^{(k)}(10^{-3}\delta w_1, \dots, 10^{-3}\delta w_k).\label{rescale-rho}
\end{equation}

Let $\tilde P_{PS}$ be the hyperbolic power series with $\xi$'s being iid standard complex Gaussian. By Theorem \ref{complex-series} and 
we have
\begin{eqnarray}\label{isom1}
\bigg|&&\int_{\mathbb{C}^{k}}G(w_1,\dots, w_k)(10^{-3}\delta_0)^{2k}\rho_{ P_{PS}}^{(k)}( z_1+10^{-3}\delta_0 w_1,\dots,  z_k+ 10^{-3}\delta_0 w_k)\text{d}w_1\dots\text{d}w_k\nonumber\\
&&-\int_{\mathbb{C}^{k}}G(w_1,\dots, w_k)(10^{-3}\delta_0)^{2k}\rho_{ {\tilde {P}}_{PS}}^{(k)}( z_1+10^{-3}\delta_0 w_1,\dots,  z_k+ 10^{-3}\delta_0 w_k)\text{d}w_1\dots\text{d}w_k\bigg|\le C'\delta_0^c.
\end{eqnarray}

As proven in Proposition 2.3.4 in \cite{HKPV1}, the zero set of $\tilde P_{PS}$ is invariant in distribution under the transformations $\phi$. 
Thus, \begin{eqnarray}\label{isom3}
&&\int_{\mathbb{C}^{k}}G(w_1,\dots, w_k)(10^{-3}\delta_0)^{2k}\rho_{ {\tilde {P}}_{PS}}^{(k)}( z_1+10^{-3}\delta_0 w_1,\dots,  z_k+ 10^{-3}\delta_0 w_k)\text{d}w_1\dots\text{d}w_k\nonumber\\
&=&\int_{\mathbb{C}^{k}}H(w_1,\dots, w_k)(10^{-3}\delta_1)^{2k}\rho_{ { \tilde {P}}_{PS}}^{(k)}( {t_1}+10^{-3}\delta_1 w_1,\dots,  t_k+10^{-3}\delta_1 w_k)\text{d}w_1\dots\text{d}w_k.
\end{eqnarray}

Thus, it remains to show that 
\begin{eqnarray}\label{isom2}
\bigg|&&\int_{\mathbb{C}^{k}}H(w_1,\dots, w_k)(10^{-3}\delta_1)^{2k}\rho_{ { {P}}_{PS}}^{(k)}( {t_1}+10^{-3}\delta_1 w_1,\dots,  t_k+10^{-3}\delta_1 w_k)\text{d}w_1\dots\text{d}w_k\nonumber\\
&&-\int_{\mathbb{C}^{k}}H(w_1,\dots, w_k)(10^{-3}\delta_1)^{2k}\rho_{ { \tilde {P}}_{PS}}^{(k)}( {t_1}+10^{-3}\delta_1 w_1,\dots,  t_k+10^{-3}\delta_1 w_k)\text{d}w_1\dots\text{d}w_k\bigg|\le C'\delta_1^c.
\end{eqnarray}

Recall that the hyperbolic area is defined by $Area(B) := \int_B \frac{dm(z)}{(1-|z|^{2)^{2}}}$ for every Borel set $B\subset \D$. By the change of variables formula, one can prove that if $\phi$ is a hyperbolic transformation then $\phi$ preserves the hyperbolic area, i.e., $Area(B) = Area(\phi(B))$. Moreover, $\phi$ maps circles in $\D$ into circles in $\D$ (see, for instance, \cite[Section 14.3]{Ru}).

Now, since $\phi$ maps $z_j$ to $t_j$ with $|z_j|\in [1-2\delta_0, 1-\delta_0]$ and $t_j \in [1-2\delta_1, 1-\delta_1]$, one has 
\begin{equation}
 \phi(\D(z_j, \delta_0/s))\subset \D(t_j, 10\delta_1/s)\label{hyp0}
 \end{equation} 
for every $s\ge 25$. Indeed, assume that $t_j\in \phi(\D(z_j, \delta_0/s)) = \D(t, r)$. Then $Area(\D(z_j, \delta_0/s)) = Area(\D(t, r))$. We have
\begin{eqnarray}
Area(\D(z_j, \delta_0/s)) = \int_{\D(z_j, \delta_0/s)}\frac{dm(z)}{(1-|z|^{2})^{2}}\le \frac{\pi\delta_0^{2}}{s^{2}}\frac{1}{(\delta_0-\delta_0/s)^{2}}\le\frac{\pi}{(s-1)^{2}}.\nonumber
\end{eqnarray}
The radius $r$ cannot be larger than $1/3$ because otherwise, there exits some $t'$ between $t$ and $t_j$ such that $|t' - t_j| = \delta_1/2$. And so $\D(t', \delta_1/2)\subset \D(t, r)$, but then
\begin{eqnarray}
Area(\D(t', \delta_1/2))\ge \frac{\pi\delta_1^{2}}{4}\frac{1}{(1 - (1-3\delta_1)^{2})^{2}}\ge\frac{\pi}{144}> \phi(\D(z_j, \delta_0/s))\label{hyp2}
\end{eqnarray}
which is impossible. So, $r\le 1/3$, and hence, for every $z\in \D(t, r)$, $|z|\ge |t_j| - 2r\ge 1 - 2\delta_1-2r>\frac{1}{3} - 2\delta_1>0$. 
Therefore, $Area(\D(t, r))\ge \frac{\pi r^{2}}{16(\delta_1 + r)^{2}}$. Comparing this with \eqref{hyp0}, we conclude that $r\le \frac{4\delta_1}{s-5}$. Hence, $\D(t, r)\subset \D(t_j, \frac{8\delta_1}{s-5})\subset \D(t_j, \frac{10\delta_1}{s})$, proving \eqref{hyp0}.

From this and the assumption that $G$ is supported in $B(0, 10^{-4})^{k}$, one can deduce that $H$ is supported in $B(0, 10^{-3})^{k}$. The inequality \eqref{isom2} will then follow from Theorem \ref{complex} if we can show that 
$|\triangledown ^{a} H(z)|\le C$ for all $0\le a\le 2k+4$ and $z\in \C^{k}$, which in turn follows from the bounds  
%
\begin{equation}\label{hyp3}
|(\phi^{-1})^{(n)}(z)|\le C_n\frac{\delta_0}{\delta_1^{n}},\quad \forall n\ge 0, \forall z\in \D(t_i, 10^{-6}\delta_1),
\end{equation}
where $C_n$ is a constant depending on $n$.


Hence, it remains to show \eqref{hyp3}. Since $\phi^{-1}(t_j) = z_j$, there exists some $\theta\in [0, 2\pi)$ such that $\phi^{-1}(z) = \varphi_{-z_j}(e^{\sqrt{-1}\theta}\varphi_{t_j}(z))$  for all $z\in \D$ where $\varphi_{\alpha} = \frac{z-\alpha}{1-z\bar \alpha}$ (see, for instance, \cite[Sections 12.4, 12.5]{Ru}). Since $e^{\sqrt{-1}\theta}$ does not change the magnitudes of the derivatives, we can assume without loss of generality that $\theta = 0$. Now, by direct computation, we have
\begin{equation}\label{hyp4}
|\varphi_{t_j}^{(m)}(z)| = \left |\frac{m!(1-|t_j|^{2}){\bar t_j}^{m-1}}{(1-\bar t_j z)^{m+1}}\right |\le \frac{C_m\delta_1}{\delta_1^{m+1}}=\frac{C_m}{\delta_1^{m}} \quad\forall m\ge 0, \forall z\in  \D(t_i, 10^{-6}\delta_1).
\end{equation}


For $z\in \D(t_i, 10^{-6}\delta_1)$, set $w= \varphi_{t_j}(z) = \frac{z - t_j}{1 - \bar t_j z}\in \D(0, 10^{-5})$. And 
\begin{equation}\label{hyp5}
|\varphi_{-z_j}^{(m)}(w)| = \left |\frac{m!(1-|z_j|^{2}){\bar z_j}^{m-1}}{(1-\bar z_j w)^{m+1}}\right |\le C_m\delta_0 \quad\forall m\ge 0, \forall w\in  \D(0, 10^{-5}).
\end{equation}

Combining \eqref{hyp4} and \eqref{hyp5}, we obtain \eqref{hyp3} and complete the proof.
\end{proof}

\section{Proof of Theorem~\ref{t.Gaussian}, part I: reduction to the case $M=m=1$}\label{s.Gaussian}

We begin the proof of Theorem~\ref{t.Gaussian} in this section.  In this section,  $\xi_k$'s are i.i.d. normalized Gaussian and  $(c_k)_{k\ge0}$ is a sequence of deterministic real numbers satisfying the following assumptions. 
For some $N_0\ge 0$ and $0<m\le M<\infty$, it holds  that
$$m\sqrt{h(k)} \le  |c_k|  \le M\sqrt{h(k)} \ \ , \ \ N_0\le k \le n $$
$$\max_{0\le k < N_0} c_k^2 \le C_1 M.$$


Below, we let $N_n(I)$ be the number of real zeros of  $\displaystyle P_n(t)=\sum_{k=0}^n c_k \xi_k t^k$ that are inside $I$   for any $I \subset \R$. For brevity, we will sometimes write $N_n(a,b) = N_n((a,b))$, $N_n[a,b]=N_n([a,b])$, etc.

   Our main goal of this section  is to reduce the theorem to the simpler case $M=m=1$.

By Edelman--Kostlan \cite{EK1}, the density function for the distribution of the real zeros for $P_n(x)$ is
$\displaystyle \rho_n(t) = \frac 1 \pi\|\gamma'_n(t)\|$
where  $\gamma_n(t)$ is the unit vector in the direction of $v_n(t):=(c_0, c_1t, \dots, c_n t^n)$. It was shown in  \cite{EK1} that
\begin{equation}\label{e.edelman-kostlan}
\|\gamma'_n(t)\|^2 = (\frac{\|v'_n(t)\|}{\|v_n(t)\|})^2 - (\frac{v_n(t)\cdot v'_n(t)}{\|v_n(t)\|^2})^2.
\end{equation}
From \eqref{e.edelman-kostlan}, it follows that $\rho_n$ is an even function of $t$.

By elementary computation,  for any $n\ge 0$ and any sequence $(x_k)$ and $(y_k)$ we have:
$$(\sum_{k=0}^n x_k^2)(\sum_{k=0}^n y_k^2) - (\sum_{k=0}^n x_k y_k)^2 = \sum_{k, m} (x_k y_m - x_m y_k)^2.$$
It follows that
$$\rho_n(t)^2 = \frac 1 {\pi^2}\frac{\|v'_n(t)\|^2 \|v_n(t)\|^2 - [v'_n(t)\cdot v_n(t)]^2}{\|v_n(t)\|^4}$$
\begin{equation*}
= \frac 1 {\pi^2}\frac{\sum_{0\le k, m\le n} (m-k)^2 c_k^2 c_m^2 t^{2m+2k-2}}{(\sum_{k=0}^n c_k^2 t^{2k})^2}.
\end{equation*}

Thus, for $|t|$ comparable to $1$ we have $\rho(t) =O(n)$, therefore $\E N_n(1-\frac c n, 1+\frac c n))=O(1)$ for any absolute constant $c>0$. Furthermore,  by scaling invariant one sees that

\begin{corollary} \label{c.compareseq} Suppose that for $0<m\le M<\infty$ we have $m|b_k| \le |a_k|\le M|b_k|$ for every $k=0,\dots, n$. Let $N_n$ and $\widetilde N_n$ respectively count the real zeros of random polynomials associated with $a_0,\dots, a_n$ and $b_0,\dots, b_n$. Then
$$\frac {m^2}{M^2} \E\widetilde N_n \le \E N_n \le \frac{M^2}{m^2} \E \widetilde N_n.$$
\end{corollary}


Thanks to Corollary~\ref{c.compareseq}, it suffices to prove Theorem~\ref{t.Gaussian}  for $m=M=1$. We will free the symbols $m$ and $M$ so that they could be used for unrelated purposes later.  

We now describe the high-level overview of the rest of the proof of Theorem~\ref{t.Gaussian}. Thanks to Lemma~\ref{boundedness}, it remains to count the number of real zeros  near the critical points $x=-1$ and $x=1$. By symmetry  it suffices to consider a small neighborhood of $1$, which we will discuss in the next two section: Section~\ref{s.frameworknear1} will discuss estimates for the denstiy function near $1$ and Section~\ref{s.kac-inside} will use these results to estimate the average number of real zeros near $1$.

\section{Proof of Theorem~\ref{t.Gaussian}, part II:  estimates for the density function near $\pm 1$}\label{s.frameworknear1}

In this section, we  prove some estimates for $\rho_n$ near $\pm1$

Below, for $x\ge 0$ let $f_n(x)= \sum_{0\le k\le n} c_n^2 x^n$, clearly $\Var[P_n(t)] =f_n(t^2)$ so our notational convention is to think of $x$ as $t^2$.   

Our general framework for the analysis in this section will be under the heuristics that $f_n(x)$ converges fairly rapidly to some $f_\infty(x)$  as $n\to\infty$. This convergence essentially leads to  the convergence of  $\rho_n$ to some limit $\rho_\infty$.  The local average number of real zeros of $P_n$ is essentially decided by the  local  behavior  of  $\rho_\infty$ and the rate of the convergence $f_n\to f_\infty$.  
For instance, if  $P_n(t)=\sum_{k= 0}^n c_k \xi_k t^k$ where $\xi_k$ are iid normalized Gaussian  and $c_k$ are independent of $n$ then the natural choice for $f_\infty$ would be $f_\infty(x) = \sum_{k=0}^\infty c_k^2 x^k$, and the convergence $f_n\to f_\infty$ holds for $x$ inside the radius of convergence of $f_\infty$.  On the other hand,  our approach  is applicable  even if  $c_k$  depend on $n$, and does not require the polynomially growing assumptions on $c_k$.

To motivate the definition of $\rho_\infty$, we let $g_n(x)=\log f_n(x)$, and note that

\begin{lemma}\label{l.density} For every $n$ it holds that
\begin{equation}\label{e.density-formula}
\rho_n(t) = \frac 1\pi\Big(g'_n(t^2) + t^2 g''_n(t^2)\Big)^{1/2}
\end{equation}
\end{lemma}

\proof
Let $v_n(t)$ denote the vector $(c_0, c_1t,\dots, c_n t^n)$. Clearly,
\begin{eqnarray*}
\|v_n(t)\|^2 &=& \sum_{0\le k\le n} c_k^2 t^{2k} \qquad =\qquad  f_n(t^2)\\
v'_n(t)\cdot v_n(t) &=& \sum_{0\le k \le n} kc_k^2 t^{2k-1} =\frac 1 2 \frac{d}{dt}(\|v_n(t)\|^2) \qquad = 
\qquad tf'_n(t^2)\\
\|v'_n(t)\|^2 
&=& \sum_{0\le k \le n} k^2 c_k^2 t^{2k-2}  = \frac 1{4t} \frac{d}{dt}(t \frac{d}{dt}(\|v_n(t)\|^2)) \qquad = \qquad f'_n(t^2) + t^2 f''_n(t^2)
\end{eqnarray*}

The desired claim now follows from the Edelman--Kostlan formula \eqref{e.edelman-kostlan}
\begin{eqnarray*}
\pi^2 \rho_n(t)^2 
&=& (\frac{\|v'_n(t)\|}{\|v_n(t)\|})^2 - (\frac{v_n(t)\cdot v'_n(t)}{\|v_n(t)\|^2})^2 \quad = \quad \frac{f'_n(t^2)}{f_n(t^2)} + t^2\frac{f''_n(t^2)f_n(t^2) - [f'_n(t^2)]^2}{[f_n(t^2)]^2}\\
&=& g'_n(t^2) + t^2 g''_n(t^2)
\end{eqnarray*}
\endproof


Let $0\le \beta<\alpha<\infty$ such that for $x\in (\alpha^2,\beta^2)$ the limit $f_\infty(x):=\lim_{n\to\infty} f_n(x)$ exists and is continuously twice differentiable on this interval. Let $g_\infty(x) = \log f_\infty(x)$ and define
\begin{equation}\label{e.limiting-density}
\rho_\infty(t):= \frac 1\pi \sqrt{ g'_\infty(t^2) + t^2 g''_\infty(t^2)}
\end{equation}
Motivated by Lemma~\ref{l.density}, under some mild assumptions one expects that $\rho_n(t)$ converges to $\rho_\infty(t)$ for $\beta<|t|<\alpha$.  The precise estimates will be discussed below.

Note that   the current analysis is only directly applicable to count the number of real zeros   inside $(-\alpha,\alpha)$ near $\pm \alpha$. For $\R\setminus (-\alpha,\alpha)$,  we will pass to the reciprocal polynomial $\widetilde P_n(t) = \frac 1 {c_n} t^n P_n(\frac 1 t)$ and apply the  argument to $\widetilde P_n$, which is also a Gaussian random polynomial. 

\subsection{Convergence of $\rho_n$}

\begin{theorem}\label{t.density-compare}
Let $u_n(x) := \frac{f_n(x)}{f_\infty (x)}$. Assume that $I_n  \subset (\beta,\alpha)$ is an interval (whose endpoints may depend on $n$) such that $u_n(t^2) \ge c_0$ for $|t|\in I_n$ for some fixed constant $c_0>0$.  Then uniformly over $\{|t| \in I_n\}$ it holds that 
$$\rho_n(t) =  \rho_\infty(t) + O\Big(
|u'_n(t^2)|^{1/2}  + |u'_n(t^2)|  + |u''_n(t^2)|^{1/2}\Big).$$
\end{theorem}

\proof Let $D_n(x)=\log f_n(x) -\log f_\infty(x)$. Using Lemma~\ref{l.density} we have
\begin{eqnarray*}
|\rho_n(t) - \rho_\infty(t)| 
&\le& |\rho_n(t)^2-\rho_\infty(t)^2|^{1/2} \\
&\le& |D_n'(t^2)|^{1/2} + \alpha |D''_n(t^2)|^{1/2} \ \ .
\end{eqnarray*}
On the other hand, let $x=t^2$ where $t\in I_n$, then $u_n(x) \ge c_0>0$, therefore
\begin{eqnarray*}
D'_n(x) &=& \frac{u_n'(x)}{u_n(x)} = O(u'_n(x)) \\
D''_n(x) &=& \frac{u''_n(x)}{u_n(x)} - (\frac {u'_n(x)}{u_n(x)})^2 \quad =\quad O(D'_n(x)^2 + |u_n''(x)|) 
\end{eqnarray*}
and the desired estimate immediately follows.
\endproof

We remark that the assumption $u_n(t^2) \ge c_0>0$ uniformly over $|t|\in I_n$  in Theorem~\ref{t.density-compare} is fairly mild, since one has  $u_n(x)\to 1$ as $n\to\infty$.

\subsection{Blowup nature of $\rho_\infty$}

It follows from Theorem~\ref{t.density-compare} that the leading asymptotics of $\rho_n$ on $I_n$ is determined by two factors: the size  of $u_n=f_n/f_\infty$ (and its first two derivatives), and the possible blowup  of $\rho_\infty$, which typically could happen near the endpoint of $I_n$. By \eqref{e.limiting-density}  depends on the blowup nature of $f_\infty$ there. For the polynomially growing setting of Theorem~\ref{t.Gaussian} (and with the normalization $M=m=1$) one expects that $f_\infty$ blows up  polynomially near the endpoints of its convergence interval. This will  lead to a simple pole for $\rho_\infty$, as proved in the following lemma.

\begin{lemma}\label{l.limiting-density} Let $0\le \beta<\alpha <\infty$ and $\gamma \ge 0$.  Assume that $\log f_\infty(x) +  \gamma \log |x-\alpha^2|$
has two uniformly  bounded derivatives for $x\in (\beta^2, \alpha^2)$. Then the following holds uniformly over $|t| \in (\beta,\alpha)$: 
$$\rho_\infty(t) =\frac{\alpha\sqrt{\gamma}}{\pi|t^2-\alpha^2|} + O(1).
$$
\end{lemma}

\proof Recall that $g_\infty = \log f_\infty$. For $|t| \in (\beta,\alpha)$, by the given assumption we have
$$g'_\infty(t^2) = -\frac{\gamma}{t^2-\alpha^2} + O(1) \qquad , \qquad g''_\infty(t^2) = \frac{\gamma}{(t^2-\alpha^2)^2} + O(1).$$
Using \eqref{e.limiting-density} we obtain
\begin{eqnarray*}
\rho_\infty(t)^2 
&=& \frac{1}{\pi^2} (g'_\infty(t^2) + t^2g''_\infty(t^2)) \\
&=&\frac{1}{\pi^2}(-\frac{\gamma}{t^2-\alpha^2} + \frac{t^2 \gamma}{(t^2-\alpha^2)^2}) + O(1)\\
&=& \frac{1}{\pi^2}\frac{\gamma \alpha^2}{(t^2-\alpha^2)^2}  + O(1)  \ \ .
\end{eqnarray*}
Since $\rho_\infty\ge 0$, the desired conclusion follows immediately.
\endproof

\section{Proof of Theorem~\ref{t.Gaussian}, part III: counting real zeros near $\pm 1$}\label{s.kac-inside}
Recall  that  $h(k)=\sum_{j=0}^d \alpha_j L_j(L_j+1)\dots(L_j+k-1)/k!$ with nonzero coefficients,   and for some fixed $N_0\ge 0$ the following holds:
\begin{itemize}
\item for every $N_0\le k\le n$ it holds that $|c_k| =\sqrt{h(k)}$. 
\item for some $C_1$ fixed we have $\max_{0\le k<N_0} |c_k| < C_1$.
\end{itemize}

Without loss of generality assume that $\alpha_d=1$.

To count the real zeros near $\pm 1$ of $P_n(t)=\sum_{k=0}^d c_k \xi_k t^k$,  we separate the treatment of the inside and outside into two results, Lemmas~\ref{l.kac-inside} and \ref{l.kac-outside} below. In the following two results, the implicit constants may depend on $N_0$, $C_1$, and $h$.

\begin{lemma}\label{l.kac-inside}
For some $\beta \in (0,1)$ that depends only on $h$, $N_0$, $C_1$, it holds that
$$\E N_n(\{\beta \le |t| \le 1\}) = \frac {\sqrt{\deg(h)+1}}{\pi}\log n + O(1)$$
\end{lemma}

\begin{lemma}\label{l.kac-outside}
It holds that
$$\E N_n ([-2,-1]\cup [1,2]) = \frac {\log n}{\pi} + O(1)$$
\end{lemma}

Remark: It is clear that Theorem~\ref{t.Gaussian} follows from Lemma~\ref{boundedness}, Lemma~\ref{l.kac-inside}, Lemma~\ref{l.kac-outside}. Therefore this section completes the proof  of Theorem~\ref{t.Gaussian}.


For convenience of notation, in the rest of the section let 
\begin{eqnarray}\label{e.gdefn}
g(x) = \sum_{k=0}^{N_0-1} [c_k^2-h(k)]x^k\ \ . 
\end{eqnarray}
It follows that $f_n(x) = g(x) + \sum_{k=0}^n h(k) x^k$.
Furthermore $g(x)$ and its derivatives are uniformly bounded on any compact subset of $\R$ with bounds depending on $C_1$ and $N_0$ and $h$. This fact will be used implicitly below.

For any $L\in \R$ we also let
$$f_{n,L}(x)=\sum_{k=0}^n b_{k,L} x^k \ \ , \ \  b_{k,L_d}:=L_d\dots (L_d+k-1)/k! \ \ .$$

\subsection{Proof of Lemma~\ref{l.kac-inside}}
Clearly, $f_n(x)\to f_\infty(x)$ for $|x|<1$, and $f_\infty(x)=g(x) +   \sum_{k= 0}^\infty h(k)x^k$. 

Using the binomial expansion of $(1-x)^{-L}$ we obtain
$$f_\infty(x) =g(x)+ \sum_{m=0}^d \alpha_m (1-x)^{-L_m}$$ for every $x\in [-1,1)$. Since $\alpha_d>0$ and $L_d>\dots > L_0>0$, it follows that
$$\log f_\infty(x) + L_d \log (1-x)$$
is  bounded uniformly over $x\in [\beta^2,1)$ for some $\beta \in (0,1)$ depending only on $N_0$, $C_1$ and $h$. We furthermore choose $\beta \in (0,1)$ to be sufficiently close to $1$ such that $|(\frac{d}{dx})^{j}f_\infty(x)|\approx (1-x)^{-L_d-j}$\footnote{We say that $f\approx g$ if there exist constants $c, C$ such that $cf\le g\le Cf$.} uniformly over $x\in [\beta^2,1)$ where $j=0, 1, 2$.
Now, for $c_0=\min(L_d, \min_j (L_{j}-L_{j-1}))>0$ it is clear  that the $j$th derivative of $\log f_{\infty}(x) + L_d \log (1-x) = \log [(1-x)^{L_d} f_\infty(x)]$ is bounded above by $O((1-x)^{c_0-j})$. Using \eqref{e.limiting-density} and argue as in the proof of Lemma~\ref{l.limiting-density} it follows that
$$\rho_\infty(t) = \frac{\sqrt{L_d}}{\pi(1-t^2)} + O((1-t^2)^{\frac {c_0}2 -1})$$
uniformly over $|t|\in [\beta,1)$. Now, recall that $L_d\equiv \deg(h)+1$. By the symmetry of the real zeros distribution, we have
\begin{eqnarray}\label{e.ENnbeta1}
\qquad \qquad \E N_n (\{\beta \le |t| \le 1\}) = 2  \int_{\beta}^{1-\frac {c}n} \rho_n(t) dt + O(1) \ \ .
\end{eqnarray}
where $c>0$ is any fixed constant. We will use the above estimate for $\rho_\infty$ to show that

\begin{lemma}\label{l.density-pw-inner} For some fixed $c$ sufficiently large, it holds uniformly over $|t|\in (\beta,1-\frac c n)$ that
$$\rho_n(t) = \frac{\sqrt{\deg(h)+1}}{2\pi(1-|t|)} + O(1)+O((1-|t|)^{\frac {c_0}2-1}) + O\Big(\frac{[n(1-t^2)]^{(L_d +1)/2} |t|^{n} +  [n(1-t^2)]^{L_d} |t|^{2n}}{1-|t|} \Big)$$
\end{lemma}

We first show that this lemma implies the desired estimate for Lemma~\ref{l.kac-inside}. Indeed, notice that for every $\alpha >0$ we have $\frac{\alpha \dots (\alpha+k-1)}{k!}  \approx k^{\alpha-1}$ and $\sum_{k=1}^n k^{\alpha-1} \approx n^\alpha$, therefore we obtain the following uniform estimates (over $0\le x\le 1$):
$$n^\alpha x^n  \quad \le \quad C \sum_{k=0}^n \frac{\alpha \dots (\alpha+k-1)}{k!} x^k \le \frac C {(1-x)^{\alpha}} \ \ .$$
Combining this with Lemma~\ref{l.density-pw-inner}, we obtain
the uniform estimate 
\begin{eqnarray*}
\rho_n(t) 
&=& \frac {\sqrt{\deg(h)+1}}{2\pi(1-t)}  + O((1-|t|)^{\frac {c_0}2-1})  + O(\frac 1{n(1-|t|)^2})
\end{eqnarray*}
over $t\in [\beta, 1-\frac cn]$ where $c>0$ is any fixed large constant. Together with \eqref{e.ENnbeta1}, we obtain
\begin{eqnarray*}
\E N_n (\{\beta \le |t|\le 1\}) 
&=& \frac {\sqrt{\deg(h)+1}}{\pi}\log n + O(1) \ \ ,
\end{eqnarray*}
as stated in Lemma~\ref{l.kac-inside}.

We now prove Lemma~\ref{l.density-pw-inner}. The proof of this Lemma relies on the following estimates for $f_n$:

\begin{lemma}\label{l.fnfinfty} For each $j=0,1,2$, it holds uniformly over $x\in [\beta^{2}, 1)$ that
$$(\frac{d}{dx})^j  \Big(f_n(x)-f_\infty(x) \Big)= O(\frac{(1+[n(1-x)]^{L_d+j-1})x^{n+1}}{(1-x)^{L_d+j}})$$
and it holds uniformly over $x\in [-1,0]$ that $f_n(x) = O((1+x)^{-(L_d-1)})$.
\end{lemma}

We first prove Lemma~\ref{l.density-pw-inner} using Lemma~\ref{l.fnfinfty}. Let $u_n=f_n(x)/f_\infty(x)$. By Lemma~\ref{l.fnfinfty}, uniformly over $x\in (\beta^2, 1)$ and $j=0,1,2$ it holds that
\begin{eqnarray*}
(\frac d{dx})^j (u_n(x)-1)  &=& O((1-x)^{-j}(1+[n(1-x)]^{L_d+j-1})x^n)
\ \ .
\end{eqnarray*}
In particular for $c$ large and $\beta^{2}\le x \le 1-\frac cn$ we have $u_n(x) = 1+ O(x^{n/2}) = 1 + O(e^{-c/2})$ and thus $u_n(x) \in [\frac 1 2, \frac 32]$. Therefore Theorem~\ref{t.density-compare} is applicable, and we obtain the desired estimate of Lemma~\ref{l.density-pw-inner}.

\proof[Proof of Lemma~\ref{l.fnfinfty}] Consider $x\in [\beta^{2},1)$. It suffices to show that for every $L>0$ and each  $0\le j\le 2$ the following holds uniformly:
 \begin{eqnarray}\label{e.tailseries}
\qquad (\frac{d}{dx})^j \Big(- \frac{1}{(1-x)^{L}} + f_{n,L}(x) \Big) =  O_L(\frac{(1+[n(1-x)]^{L+j-1})x^{n+1}}{(1-x)^{L+j}})
\end{eqnarray}
Similarly,  for $x\in [-1,0]$ it suffices to show that for any $L> 0$
\begin{eqnarray}\label{e.tailseries2}
f_{n,L}(x)  =  O((1+x)^{L-1}).
\end{eqnarray}

Observe that
\begin{eqnarray*}
 \frac d{dx} \Big(- \frac{1}{(1-x)^{L}} + f_{n,L}(x) \Big) 
&=& L\Big(  - \frac{1}{(1-x)^{L+1}}  + f_{n-1,L+1}(x)\Big) \\
\end{eqnarray*}
therefore in \eqref{e.tailseries} we may assume that $j=0$.

 Now, for $0\le x<1$ we have
 \begin{eqnarray}
\nonumber - \frac{1}{(1-x)^{L}}   + \sum_{k=0}^n  \frac{L\dots (L+k-1)}{k!}  x^k 
&=&  \sum_{k=n+1}^\infty  \frac{L\dots (L+k-1)}{k!}  x^k\\
\label{e.shiftindex} &=&x^{n+1}\sum_{k=0}^\infty  \frac{L\dots (L+k+n)}{(n+1+k)!}  x^{k}
\end{eqnarray}
Now, we will use the standard asymptotic estimate for generalized binomial coefficients
$$\frac{L(L+1)\dots (L+k-1)}{k!} \approx C  k^{L-1}$$
as $k\to \infty$ where $C$ depends on $L$. It follows that
\begin{eqnarray*} 
\frac{L(L+1)\dots (L+k+n)}{(n+1+k)!}  
&\le& C \frac{L(L+1)\dots (L+k-1)}{k!}   (\frac{n+k+1}k)^{L-1} \\
&\le& C   \frac{L(L+1)\dots (L+k-1)}{k!} (1+ \frac{(n+1)^{L-1}}{k^{L-1}})\\
&\le& C \frac{L(L+1)\dots (L+k-1)}{k!} + (n+1)^{L-1}
\end{eqnarray*}
(in the last estimate we use the asymptotic for generalized binomial coefficients again). Using \eqref{e.shiftindex} and the binomial expansion, it follows that
\begin{eqnarray*} - \frac{1}{(1-x)^{L}}   + \sum_{k=0}^n  \frac{L(L+1)\dots (L+k-1)}{k!}  x^k 
&\le& x^{n+1}\Big[  (1-x)^{-L} + (n+1)^{L-1}  (1-x)^{-1}\Big] \\
&\le& C \Big((1+[n(1-x)]^{L-1}) (1-x)^{-L} x^{n+1}\Big)
\end{eqnarray*}
giving \eqref{e.tailseries}. 

For $x\in [-1,0]$ we will use the following recursive formulas. 
\begin{lemma}\label{l.recursive-f}
For any $x\ne 1$ it holds that
$$f_{n,L}(x) = \frac{f_{n,L-1}(x)}{1-x} - \frac{L\dots (L+n-1)}{n!}\frac{x^{n+1}}{1-x}.$$
\end{lemma}
\proof We have
\begin{eqnarray*}
f_{n,L}(x) &=& 1+ Lx + \frac{L(L+1)}2x^2 + \dots + \frac{L(L+1)\dots (L+n-1)}{n!}x^n\\
xf_{n,L}(x) &=&   x + Lx^2 + \dots + \frac{L(L+1)\dots (L+n-2)}{(n-1)!}x^n + \frac{L\dots (L+n-1)}{n!}x^{n+1} \\
(1-x)f_{n,L}(x) &=&  1+ (L-1)x + \frac{L(L-1)}2x^2 + \dots + \frac{L(L+1)\dots (L+n-2)(L-1)}{n!}x^n\\
&& - \frac{L\dots (L+n-1)}{n!}x^{n+1} \\
&=& f_{n,L-1}(x) - \frac{L\dots (L+n-1)}{n!}x^{n+1}
\end{eqnarray*}
and the desired claim follows.
\endproof
For $x\in [-1,0]$ it is clear that $\frac{L\dots (L+n-1)}{n!}\frac{x^{n+1}}{1-x} = O(n^{L-1}|x|^n) = O(\frac 1 {(1+x)^{L-1}})$. Thus, without loss of generality we may assume that $0< L\le 1$. For this $L$, for $x\in [-1,0]$ it is clear that $f_{n,L}$ is an alternating sum whose terms have decreasing modulus,  and could be easily bounded  by $O(1)$ uniformly over $x\in [-1,0]$.
\endproof

\subsection{Proof of Lemma~\ref{l.kac-outside}}

Thanks to the symmetry of the distribution of the real zeros, we have 
\begin{eqnarray}\label{e.ENnouter}
\E N_n(\{1\le |t|\le 2\}) = 2 \E \widetilde N_n(\frac1 2, 1) = 2\int_{\frac 1 2}^{1-\frac cn} \widetilde \rho_n(t)dt + O(1)
\end{eqnarray}
where $\widetilde N_n$ and $\widetilde \rho_n$ are respectively  the number of real zeros and the density of the real zeros distribution for the normalized reciprocal polynomial $$\widetilde P_n(t) = \sum_{k=0}^n \frac{c_{n-k}}{c_n}\xi_k t^k \ \ .$$
We note that $|c_n| = \sqrt{h(n)}$ so $c_n \ne 0$ for $n$ sufficiently large, so $\widetilde P_n$ is well-defined.

 Let $\widetilde f_n(x)$ denote the corresponding variance function
$$\widetilde f_n(x) = \sum_{k=0}^n \frac{c_{n-k}^2}{c_n^2}x^k \equiv\frac{x^n f_n (1/x)}{c_n^2}$$
As we will see, for any $0\le x<1$ the sequence $\widetilde f_n(x)$ converges to $\widetilde f_\infty(x) := \frac 1{1-x}$, which suggests that $\widetilde \rho_n(t)$ is asymptotically $\frac 1{2\pi (1-t)}$ for $t\in [\frac 12, 1)$. In fact, we will show that

\begin{lemma} \label{l.density-pw-outer} Suppose that $c>0$ is a sufficiently large fixed constant. Then uniformly over $t\in [\frac 1 2, 1-\frac cn]$ it holds that
$$\widetilde \rho_n(t) = \frac 1{2\pi(1-t)} (1+O(n^{L_{d-1}-L_d}))+ O(1) + O(\frac 1{n(1-t)^2} + \frac{1}{\sqrt{n(1-t)^3}}).$$
\end{lemma}

From the following computation, Lemma~\ref{l.density-pw-outer} and \eqref{e.ENnouter} imply the desired estimate for Lemma~\ref{l.kac-outside}:
\begin{eqnarray*}
\E N_n (\{1\le |t|\le 2\}) 
&=& 2 \int_{\frac 1 2}^{1-\frac cn} \frac{1}{2\pi (1-t)} dt + O(1) + O(\int_{\frac 1 2}^{1-\frac cn}  \frac 1{n(1-t)^2} + \frac 1{n^{1/2}(1-t)^{3/2}}dt) \\
&=& \frac {\log n}{\pi} + O(1) \ \ .
\end{eqnarray*}


To prove Lemma~\ref{l.density-pw-outer},  we reduce the problem to the hyperbolic setting. As we will see, $\widetilde f_n(x)$ converges to $\widetilde f_\infty(x)=\frac 1{1-x}$ for every $x\in [0,1)$ sufficiently close to $1$, say $x\in [1/2,1)$. Our proof will make use of the density comparison results developed in the previous section, Theorem~\ref{t.density-compare} and Lemma~\ref{l.limiting-density}, relying on various estimates for $\widetilde f_n(x)/\widetilde f_\infty(x)$ and its first two derivatives. It is clear that modulo the contribution of $g$ (defined in \eqref{e.gdefn}) which will be shown to be very small, $\widetilde f_n(x)$ is a linear combination of $\widetilde f_{n,L_j}$ where the linear coefficient for $\widetilde f_{n,L_d}$ is $1+O(n^{-c})$ and the linear coefficients of other terms are $O(n^{-c})$ where $c=L_d-L_{d-1}$. Thus it suffices to consider the setting when $f_n = g+f_{n,L_d}$, which we assume below.

We first establish some basic estimates for $f_{n,L}$.
\begin{lemma}\label{l.reciprocal-fractional} Let $L\in \R \setminus \{0,-1,-2,\dots\}$. Then uniformly over $0\le x <1$ it holds that
\begin{eqnarray}\label{e.reciprocal-fractional}
\widetilde f_{n,L}(x) = \frac {1}{1-x} \Big[1+ O(\frac{1}{n(1-x)})\Big]
\end{eqnarray}
the implicit constant  depends only on $L$. Furthermore, if $L\ge 1$ then uniformly over $x\in [-1,0]$ it holds that $\widetilde f_{n,L}(x)=O(1)$.
\end{lemma}

\proof For every $x$ we have
\begin{eqnarray}
\nonumber \widetilde f_{n,L}(x) &=&  \sum_{k=0}^n \frac{L\dots (L+n-k-1)n!}{L\dots(L+n-1)(n-k)!}x^k \\
\nonumber &=&  \sum_{k=0}^n  \frac{(n-k+1)\dots n}{(L+n-k)\dots (L+n-1)}x^k \\
\label{e.monotoneform} &=&  \sum_{k=0}^n  \frac{x^k}{(1+\frac{L-1}{n-k+1})\dots (1+\frac {L-1}{n})}.
\end{eqnarray}
Now, it is clear that if $x\in [-1,0]$ and $L\ge 1$ then \eqref{e.monotoneform} is an alternating sum where the terms have decreasing modulus, thus is clearly bounded above by $O(1)$.

Now we consider $L\in \mathbb R \setminus \{0,-1,\dots,\}$ and $x\in [0,1)$. Notice that for $0\le k \le n/2$ (and $n$ large) it holds that $0< 1-\frac {2|L-1|}{n} \le 1+\frac{L-1}{n-k+1} \le 1+\frac{2|L-1|}{n}$. It follows that
 $(1+\frac{L-1}{n-k+1})\dots (1+\frac {L-1}{n}) \approx 1$,
therefore by a telescoping argument we obtain
$$\frac{1}{(1+\frac{L-1}{n-k+1})\dots (1+\frac {L-1}{n})} = 1 + O(\frac {k}{n})$$
(the implicit constant depends on $L$). Consequently the sum of the first $n/2$ terms of $\widetilde f_{n,L}$ satisfies
\begin{eqnarray*}
\sum_{0\le k\le n/2}  \frac{x^k}{(1+\frac{L-1}{n-k+1})\dots (1+\frac {L-1}{n})} 
&=& \sum_{0\le k \le n/2} x^k + \frac 1 n O(\sum_{k\ge 0} k x^k) \\
&=& \frac 1{1-x} + O(\frac 1 {n(1-x)^2}) \ \ .
\end{eqnarray*}
For the other terms, we use the classical estimate
$$C_0 k^{L-1} \le |\frac{L(L+1)\dots (L+k-1)}{k!}| \le C_2 k^{L-1}$$
for some $C_0, C_2>0$ depending only on $L$ (this estimate requires $L \not\in \{0,-1,-2\dots\}$. It follows that
\begin{eqnarray*}
|\sum_{n/2<k\le n} \frac{L(L+1)\dots (L+n-k-1)/(n-k)!}{L(L+1)\dots (L+n-1)/n!} x^k| 
&\le& C n^{1-L} x^{n/2} \sum_{n/2<k\le n} (n-k)^{L-1} \\
&\le& C n^{1-L} x^{n/2} n^{L}\\
&\le& C \frac 1{n(1-x)^2}.
\end{eqnarray*}
This completes the proof of the lemma.
\endproof

Now,  recall the definition of $g$ in \eqref{e.gdefn}, we obtain \begin{eqnarray}\label{e.hyperbolic-reduction}
\widetilde f_n(x) = \frac 1 {b_{n,L_d}} x^n g(\frac 1 x) +  \widetilde f_{n,L_d}(x)
\end{eqnarray}
and we have the crude estimate (which holds  uniformly over $x=O(1)$)
$$|\frac 1 {b_{n,L_d}} x^n g(1/x)| \le C n^{1-L_d} (x^n + x^{n-N_0}) \le C \frac{1}{n^{L_d+1}(1-x)^2}$$ 
therefore using  Lemma~\ref{l.reciprocal-fractional}, we obtain the following corollary: 
\begin{corollary}\label{c.reciprocal-genKac} Uniformly over $0\le x <1$ it holds that
\begin{eqnarray*}
\widetilde f_n(x) 
&=& \frac {1}{1-x} \Big[1+ O(\frac{1}{n(1-x)})\Big]  \ \ .
\end{eqnarray*}
\end{corollary}
Since $L_d>0$, it follows that for every fixed $x\in [0,1)$ we have $\lim_{n\to\infty} \widetilde f_n(x) = \frac 1 {1-x} \equiv \widetilde f_\infty(x)$
as claimed earlier. Furthermore, from Corollary~\ref{c.reciprocal-genKac}, it follows that for any fixed $c>0$, if $x\in [0,1-\frac cn]$ then
$$\widetilde u_n(x):=\frac{\widetilde f_n(x)}{\widetilde f_\infty(x)} = 1+O(\frac 1c)$$
for therefore by choosing $c>0$ sufficiently large we could ensure that $\widetilde u_n(x) \ge 1/2$ for every $x\in [0,1-\frac cn]$ and every $n$ sufficiently large. Thus, by Theorem~\ref{t.density-compare} and Lemma~\ref{l.limiting-density}, we obtain the following estimate, uniformly over $t\in [\frac 1 2, 1-\frac cn]$:
\begin{eqnarray*}
\widetilde \rho_n(t) = \frac 1{2\pi (1-t)} +O(1) + O(|\widetilde u_n'(t^2)|^{1/2} + |\widetilde u_n'(t^2)| + |\widetilde u''_n(t^2)|^{1/2}).
\end{eqnarray*}
Thus, to complete the proof of Lemma~\ref{l.density-pw-outer}, it remains to show the following estimates uniformly over $x\in [\frac 1 4, 1-\frac cn]$:
\begin{eqnarray}
\label{e.un-der}
\widetilde u'_n(x) &=& O(\frac 1{n(1-x)^2}) \\
\label{e.un-der2} 
\widetilde u''_n(x) &=& O(\frac{1}{n(1-x)^3}) \ \  . 
\end{eqnarray}

Recall from \eqref{e.hyperbolic-reduction} that
$$\widetilde f_n(x) = \frac 1 {b_{n,L_d}} x^n g(\frac 1 x) +  \widetilde f_{n,L_d}(x)$$
Using the definition \eqref{e.gdefn} for $g$ and using $L_d\ge 0$ it follows that the first term $g_1(x) := \frac 1{b_{n,L_d}} x^n g(\frac 1 x)$ satisfies
$$\frac d{dx} g_1(x) = O(n^{2-L_d} x^n) = O(\frac 1{n^{L_d+1} (1-x)^3}) = O(\frac 1{n(1-x)^2}\widetilde f_\infty(x)) \ \ ,$$
$$(\frac d {dx})^2 g_1(x) = O(n^{3- L_d} x^n) = O(\frac 1{n^{L_d+1}(1-x)^4}) = O(\frac 1{n(1-x)^3}\widetilde f_\infty(x)) \ \ ,$$
uniformly over $x\in [1/2,1)$.

Therefore, it suffices to show \eqref{e.un-der} and \eqref{e.un-der2} for $f_n=f_{n,L_d}$.  We will use the following analogue of Lemma~\ref{l.recursive-f}:

\begin{lemma}\label{l.recursive}
For $L\not\in  \{0,-1,-2,\dots\}$ It holds that
$$\widetilde f_{n,L}(x) = \frac{1}{1-x} -\frac x{1-x}\frac{L-1}{L+n-1}  \widetilde f_{n,L-1}(x).$$
\end{lemma}

\proof This follows from Lemma~\ref{l.recursive-f} using the definition of $\widetilde f$.
Alternatively, we could directly compute
\begin{eqnarray*}
(1-x)\widetilde f_{n,L}(x) 
&=& 1 + \sum_{k=1}^n \frac{b_{n-k,L}-b_{n-k+1,L}}{b_{n,L}} x^k - \frac 1{b_{n,L}}x^{n+1} \\
&=& 1 - \sum_{k=1}^n \frac{b_{n-k+1,L-1}}{b_{n,L}} x^k - \frac 1{b_{n,L}}x^{n+1} \\
&=& 1 - \frac{L-1}{L+n-1}\sum_{k=1}^n \frac{b_{n-k+1,L-1}}{b_{n,L-1}} x^k - \frac{L-1}{L+n-1}\frac 1{b_{n,L-1}}x^{n+1} \\
&=& 1 -\frac{L-1}{L+n-1} x\widetilde f_{n,L-1}(x)
\end{eqnarray*}
giving the desired claim.
\endproof

Using Lemma~\ref{l.reciprocal-fractional} and Lemma~\ref{l.recursive} it follows that if $L\not\in \{1,0,-1,\dots\}$ then
\begin{eqnarray}\label{e.finer}
\widetilde f_{n,L}(x)
=\frac 1 {1-x} - \frac {L-1}{L+n-1}\frac{x}{(1-x)^2} + O(\frac 1{n^2(1-x)^3}). 
\end{eqnarray}
On the other hand, if $L=1$ then this estimate holds trivially via explicit computation from $\widetilde f_{n,1}=(1-x^{n+1})/(1-x)$. Thus \eqref{e.finer} holds for any $L\not\in \{0,-1,-2,\dots\}$.

Now, using \eqref{e.finer} and Lemma~\ref{l.recursive} again we obtain the following corollary:
\begin{corollary}\label{c.finer}
For $x\in [1/2, 1-c/n]$, if $L\not\in \{0,-1,-2\dots\}$ then
\begin{eqnarray*}
\widetilde f_{n,L}(x)
&=& \frac 1 {1-x} - \frac {L-1}{L+n-1}\frac{x}{(1-x)^2} + O(\frac 1{n^2(1-x)^3}) \\
&=& \frac 1 {1-x} - \frac {(L-1)}{(L+n-1)}\frac{x}{(1-x)^2} + \frac{(L-1)(L-2)}{(L+n-1)(L+n-2)} \frac{x^2}{(1-x)^3} + O(\frac 1{n^3(1-x)^4}).
\end{eqnarray*}
\end{corollary}
(Again, the case $L=1$ of the second estimate in Corollary~\ref{c.finer} does not follow from \eqref{e.finer} and Lemma~\ref{l.recursive} and one checks this case  separately  using explicit computation.)

Now, we show the desired estimate \eqref{e.un-der} for $\widetilde u'_n$. As remarked earlier it suffices to assume $f_n=f_{n,L}$ for some  $L>1$. We have
\begin{eqnarray*}
\widetilde u'_n(x) &=&  (\widetilde f_{n,L}(x) (1-x))' = (1-x) \frac d{dx}\widetilde f_{n,L}(x)  -\widetilde f_{n,L}(x)  \ \ ,\\
\frac{d}{dx}\widetilde f_{n,L}(x) &=& \frac{1}{b_{n,L}}\Big[nx^{n-1}f_{n,L}(\frac 1 x ) -  x^{n-2}f'_{n,L}(\frac 1 x) \Big].  
\end{eqnarray*}
It is clear that 
$$f'_{n,L}(x)=L \sum_{k=0}^{n-1} \frac{(L+1)\dots (L+k)}{k!} x^k = L f_{n-1,L+1}(x)$$
therefore  
\begin{eqnarray}\label{e.ftilde-der}
\frac{d}{dx}\widetilde f_{n,L}(x) = \frac{n}{x}[\widetilde f_{n,L}(x) -  \widetilde f_{n-1, L+1}(x)]. 
\end{eqnarray}

Recall that $L_d > 0$.
Thus, by Corollary~\ref{c.finer},  we have
\begin{eqnarray*}
\widetilde  u'_n(x) &=& \frac{n(1-x)}x [\widetilde f_{n,L_d}(x) - \widetilde f_{n-1, L_d+1}(x)]  - \widetilde f_{n,L_d}(x) \\
&=&   \frac{n(1-x)}{x}\Big[(\frac 1 {1-x}-\frac {L_d-1}{L_d+n-1} \frac{x}{(1-x)^2})  - ( \frac 1 {1-x} - \frac{L_d}{L_d+n-1} \frac{x}{(1-x)^2})\Big] +\\
&&+ (-\frac 1 {1-x})  + O(\frac 1{n(1-x)^2})  \\
&=&O(\frac 1{n(1-x)^2}) 
\end{eqnarray*}
uniform over $x\in [1/2, 1-c/n]$, thus proving \eqref{e.un-der}. 

For \eqref{e.un-der2}, we use \eqref{e.ftilde-der} to obtain
\begin{eqnarray*}
\widetilde  u''_n(x) &=&  \frac{d}{dx}\Big[\frac{n(1-x)}{x}(\widetilde f_{n,L}(x) - \widetilde f_{n-1,L+1}(x))\Big]  - \frac{d}{dx}\widetilde f_{n,L}(x) \\
&=& \frac{n(1-x)}x \Big[\frac n x (\widetilde f_{n,L}(x) - \widetilde f_{n-1, L+1}(x)) - \frac{n-1}x (\widetilde f_{n-1, L+1}(x) - \widetilde f_{n-2, L+2}(x))\Big] + \\
&& + (-\frac n{x^2} - \frac n x)[\widetilde f_{n,L}(x) - \widetilde f_{n-1,L+1}(x)]. 
\end{eqnarray*}

Using Corollary~\ref{c.finer} again, we have
\begin{eqnarray*}
\widetilde f_{n,L}(x) - \widetilde f_{n-1,L+1}(x) 
&=& \frac {x}{(L+n-1)(1-x)^2} + \frac{-2(L-1)}{(L+n-1)(L+n-2)} \frac{x^2}{(1-x)^3} + O(\frac 1 {n^3(1-x)^4}) .
\end{eqnarray*}
Therefore
\begin{eqnarray*}
\widetilde  u''_n(x) &=&    \frac{n(1-x)}{x}\Big[\frac {1}{(L_d+n-1)(1-x)^2}] + \\
&&+ \frac{n(1-x)}{x}\Big[\frac {2L_d(n-1)- 2(L_d-1)n}{(L_d+n-1)(L_d+n-2)} \frac{x}{(1-x)^3}\Big] +\\
&&+ (-1)\frac{n+nx}{x(L_d+n-1)(1-x)^2}   + O(\frac 1{n(1-x)^3})\\
&=& \frac{n(1-x)}{x}\Big[\frac {1}{n(1-x)^2} + O(\frac{1}{n^2(1-x)^2})\Big] + \\
&&+ \frac{n(1-x)}{x}\Big[\frac {2L_d n- 2(L_d-1)n}{n^2} \frac{x}{(1-x)^3} + O(\frac{1}{n^2(1-x)^3})\Big] +\\
&&+ (-1)\frac{n+nx}{nx(1-x)^2}   + O(\frac 1{n(1-x)^2}) + O(\frac 1{n(1-x)^3})\\
&=& \frac {1}{x(1-x)} + \frac {2n}{n(1-x)^2} - \frac{n+nx}{x(1-x)^2} + O(\frac 1{n(1-x)^3}) \\
&=& O(\frac 1{n(1-x)^3}) 
\end{eqnarray*}
thus proving \eqref{e.un-der2}. This completes the proof of Lemma~\ref{l.kac-outside}.

\section{Proof of Theorem~\ref{t.Gaussian2}}\label{s.Gaussian2}
In this section we count the average number of real zeros for $P_n(t)=\sum_{j=0}^n c_j \xi_j t^j$ where for $j\ge N_0$ two conditions hold: $c_j=\mathfrak P(j)$ for some fixed classical polynomial $\mathfrak P$ of degree $\rho$ when $j\ge N_0$ and is bounded when $j\le N_0$, and $\xi_j$ are independent Gaussian with mean $\mu \ne 0$ and variance $1$. Without loss of generality we may assume that the leading coefficient of $\mathfrak P$ is $1$, i.e. $c_j = j^\rho + \dots$.

Thanks to Lemma~\ref{boundedness} the average number of real zeros outside    $[-1-b_1,-1+b_1]$ and $[1-b_1, 1+b_1]$  (for any fixed  $b_1>0$) is bounded.

We will show that on average there are a bounded number of real zeros in  $[1-b_1,1+b_1]$ and $\frac {1+\sqrt{2\rho+1}}{2\pi} \log n + O(1)$ real zeros in $[-1-b_1, -1+b_1]$.

As in the proof of Corollary~\ref{mean}, let $m(t)=\E P(t)$, $\mathcal P = \Var[P_n] = \sum_{j=0}^n c_j^2 t^{2j}$, $\mathcal Q = \Var[P'_n(t)] = \sum_{j=0}^n c_j j^2 t^{2j-2}$, and $\mathcal R = \textbf{Cov}[P_n, P'_n] = \sum_{j=0}^n j c_j^2 t^{2j-1}$, and $\mathcal S = \mathcal P \mathcal Q - \mathcal R^2$.

We will use the following generalization of the Kac-Rice formula in \cite[Corollary 2.1]{F1}, which gives
\begin{eqnarray*}
\E N_n[a,b] &=& I_1(a,b) + I_2 (a,b) \ \ , \\
I_1 (a,b) &:=& \int_{a}^b  \frac{\mathcal S^{1/2}}{\pi \mathcal P}  \exp(-\frac{m^2 \mathcal Q + m'^2 \mathcal P - 2 mm' \mathcal R}{2\mathcal S}) dt  \ \ , \\
I_2(a,b) &:=& \int_{a}^b \frac{\sqrt 2 |m' \mathcal P  - m\mathcal R|}{\pi \mathcal P^{3/2}}  \exp(-\frac {m^2}{2\mathcal P}) erf(\frac{|m'\mathcal P - m \mathcal R|}{\sqrt {2\mathcal P} \mathcal S}) dt, \\
erf(x) &:=& \int_0^x e^{-t^2}dt.
\end{eqnarray*}
We note that in $I_1$ the first factor $\mathcal S^{1/2}/(\pi \mathcal P)$ is exactly the density of the real roots for $P_n$ in the mean zero case, namely $\rho_n$ in the notation of  Lemma~\ref{l.density}, and there is an extra exponential factor in $I_1$. Our plan is, essentially, to show that near $1$ the exponential decay of   the extra factor in $I_1$ will cancel out  the pole singularity of $\rho_n$ and near $-1$ the extra factor in $I_1$ is essentially $1$. This would lead to $I_1(a,b) = O(1)$ if $a,b$ are close to $1$ and $I_1(a,b)=\int_a^b \rho_n(t)dt + O(1)$ if $a,b,$ are close to $-1$, thus  allowing us to reduce the proof to the mean zero case. For $I_2$ we will show that  $I_2(a,b)=O(1)$ for both cases.

We now separate the neighborhood into four intervals: $[1-b_1,1]$, $[-1, -1+b_1]$, $[1,1+b_1]$, and $[-1-b_1, -1]$,  where $b_1>0$ is a sufficiently small fixed constant.

\underline{The interval $[1-b_1,1]$.} We wil show that this interval contributes $O(1)$ to $\E N_n$. Using \eqref{e.tailseries}, for $1-b_1\le t <1$ we have
\begin{eqnarray*}
m(t) &=& \mu \sum_{j=0}^n c_j t^j =  \sum_{j< N_0} \mu [c_j -\mathfrak P(j)]t^j + \frac{\mu \rho!}{(1-t)^{\rho+1}} \Big(1+ O([1+n(1-t)]^{\rho} t^{n+1})\Big)\\
&=& O(1) + \frac{\mu \rho!}{(1-t)^{\rho+1}} \Big(1+ O([1+n(1-t)]^{\rho} t^{n+1})\Big)\\
&=& O(1) + \frac{\mu \rho!}{(1-t)^{\rho+1}} \Big(1+ O(t^{n/2})\Big)
\end{eqnarray*}
here we have used the fact that $n^L s^n =O((1-s)^{-L})$, applied to $s=t^{1/2}$. Similarly,
\begin{eqnarray*}
m'(t) &=& \mu \sum_{j=0}^{n-1} c_{j+1} (j+1) t^{j} = O(1) + \frac{\mu (\rho+1)!}{(1-t)^{\rho+2}}\Big(1+ O(t^{n/2})\Big)\\
\mathcal P &=& O(1) + \frac{(2\rho)!}{(1-t^2)^{2\rho+1}} \Big(1+ O(t^{n})\Big)\\
\mathcal Q &=&  O(1) + \frac{(2\rho+2)!}{(1-t^2)^{2\rho+3}} \Big(1+ O(t^{n})\Big)\\
\mathcal R &=&  O(1) + \frac{(2\rho+1)! t}{(1-t^2)^{2\rho+2}} \Big(1+ O (t^{n})\Big).
\end{eqnarray*}
Note that by choosing $c>0$ sufficiently large we could ensure that $ t^{n/2} \ll 1$ for $|t| \le 1- \frac c n$, and by choosing $b_1>0$ sufficiently small we could ensure that $\frac 1 {1-t^2} \gg 1$ for $t\in [1-b_1, 1)$. It follows that  
\begin{eqnarray*}
m^2 \mathcal Q + m'^2 \mathcal P - 2mm' \mathcal R & \ge& C_\rho \frac 1 {(1-t)^{4\rho+5}}   \ge C'_\rho \frac{\mathcal P^2}{(1-t)^3} 
\end{eqnarray*}
for some positive constants $C'_\rho, C_\rho$ depending only on $\rho$ and $\mu$. On the other hand, by Lemma~\ref{l.density-pw-inner} we have 
\begin{eqnarray*}
\frac{\mathcal S^{1/2}}{\pi \mathcal P} = \rho_n \sim \frac 1 {1-|t|} \ \ . 
\end{eqnarray*}
Consequently, uniformly over $t\in [1-b_1,1]$ we have
\begin{eqnarray*}
\frac{m^2 \mathcal Q + m'^2 \mathcal P - 2mm' \mathcal R}{2\mathcal S} & \ge& C''_\rho \frac 1 {1-t} 
\end{eqnarray*}
therefore
\begin{eqnarray*} 
I_1 (1-b_1, 1-\frac c n) &=& O\Big(\int_{1-b_1}^{1} \frac 1 {1-t} \exp(-\frac {C''_\rho}{1-t}) dt\Big) = O(1) \\
I_1(1-\frac cn, 1) &\le& \int_{1-\frac c n}^1 \rho_n(t) dt = O(1).
\end{eqnarray*}

Now, for $I_2$ we similarly have, for $t\in [1-b_1, 1-\frac cn]$,
\begin{eqnarray*}
\frac {m^2}{2\mathcal P} &\ge& C_\rho \frac 1 {1-t}  \ \  ,\\
|m'\mathcal P - m\mathcal R| &=& O(\frac 1 {(1-t)^{3\rho+3}}) = O(\frac {\mathcal P^{3/2}}{(1-t)^{3/2}})
\end{eqnarray*}
therefore
\begin{eqnarray*}
I_2(1-b_1, 1-\frac c n)  &=& O\Big(\int_{1-b_1}^{1-\frac cn} \frac 1 {(1-t)^{3/2}} \exp(-C_\rho\frac 1 {1-t})dt\Big) = O(1)  \ \ .
\end{eqnarray*}
On the other hand, the integrand of $I_2$ is   bounded above by $O(n)$ for $t\in [1-\frac cn, 1]$, for any fixed $c>0$. To see this, first note that for some absolute constant $n_0$ the coefficients $c_j$ are of the same sign  and $|c_j| \ge j^\rho$ for $j\ge n_0$. It follows that the main contribution to $m$ and $m'$ comes from the tail $j\ge n_0$. For instance,
\begin{eqnarray*}
|m(t)| &=& O(1) + |\sum_{n_0\le j \le n} c_j t^j| \ \ , \\ 
|\sum_{n_0\le j \le n} c_j t^j| &\ge& \frac 1 C \sum_{n_0\le j \le n} j^\rho  \ge  \frac 1 {C'} n^{\rho+1} \gg 1 \ \ ,
\end{eqnarray*}
and for $m'$ we could argue similarly. Since $c_j$ are of the same sign for $j\ge n_0$, it follows immediately that $|m'(t)| =O(n |m(t)|)$, and consequently
\begin{eqnarray*}
\frac{|m' \mathcal P|}{\mathcal P^{3/2}} \exp(-\frac {m^2}{2\mathcal P}) &=&   nO\Big(\frac{|m|}{\mathcal P^{1/2}} \exp(-\frac {m^2}{2\mathcal P})\Big)= O(n).
\end{eqnarray*}
using the boundedness of $xe^{-x^2}$. We also have 
\begin{eqnarray*}
\frac{|m\mathcal R|}{\mathcal P^{3/2}}  \exp(-\frac{m^2}{2\mathcal P})  &=& O(\frac{\mathcal R}{\mathcal P}) = O(n). 
\end{eqnarray*}
It follows that $I_2(1-\frac cn, 1) = O(1)$, so $I_2(1-b_1, 1)=O(1)$.

\underline{The interval $[-1, -1+b_1]$.} We will show that this interval contributes $(\sqrt{2\rho+1}\log n)/(2\pi) + O(1)$ to $\E N_n$. The analysis of this interval is fairly similar to the analysis of $[1-b_1, 1]$, the main difference is that $m(t)$ and $m'(t)$ are less singular near $-1$, in fact they are bounded by $O((1+t)^{-\rho})$ and $O((1+t)^{-(\rho+1)})$ respectively (by using \eqref{e.tailseries2} for $L=1,2,\dots, \rho$ and expanding the polynomial defining $c_j$ into
the linear basis of binomial polynomials). It follows that
\begin{eqnarray*}
m^2 \mathcal Q + m'^2 \mathcal P - 2mm'\mathcal R = O(\frac 1 {(1+t)^{4\rho+3} }) = O(\frac{\mathcal P^2}{1+t}) = O((1+t)\mathcal S)
\end{eqnarray*}
therefore for $c>0$ sufficiently large and $b_1>0$ sufficiently small
\begin{eqnarray*}
I_1(-1+\frac cn, -1+b_1) &=& \int_{-1+\frac cn}^{-1+b_1} \rho_n(t)dt + O\Big(\int_{-1}^{-1+b_1} \rho_n(t)(1+t)dt\Big) = \frac {\sqrt{2\rho+1}}{2\pi}\log n + O(1)\\
I_1(-1,-1+\frac cn)  &\le& \int_{-1}^{-1+\frac cn}\rho_n(t)dt = O(1).
\end{eqnarray*}
For $I_2$, similarly we only need to show that $I_2(-1+\frac cn, -1+b_1)=O(1)$. This follows from
\begin{eqnarray*}
\frac{|m'\mathcal P- m\mathcal R|}{\mathcal P^{3/2}} = O(\frac{(1+t)^{-(2\rho+2)} } {(1+t)^{-3\rho+\frac 32}}) = O((1+t)^{\rho-\frac 1 2}) = O((1+t)^{-\frac 1 2}).
\end{eqnarray*}

\underline{The interval $[1,1+b_1]$.} We will show that this interval contributes $O(1)$ to $\E N_n$. To analyze this interval, we will consider the reciprocal polynomial $\widetilde P_n(t) = \sum_{j=0}^n \widetilde c_j \xi_{j} t^j$ where $\widetilde c_j = c_{n-j}/c_n$. For convenience of notation, let $\widetilde \rho_n$, $\widetilde I_1$, $\widetilde I_2$, $\widetilde {\mathcal P}$, $\widetilde {\mathcal Q}$, $\widetilde {\mathcal R}$,  $\widetilde {\mathcal S}$, $\widetilde m$, and $\widetilde m'$ be the corresponding quantities, and similarly it suffices to show that $\widetilde I_1(1-b_1,1-\frac cn),\widetilde I_2(1-b_1,1 -\frac cn) =O(1)$ where $c>0$ is a fixed large constant. 

Let $\widetilde f_n(t)= \widetilde{\mathcal P}(t)$, in other words $\widetilde f_n(t^2) =\Var[\widetilde P_n(t)]$ as in the proof of Theorem~\ref{t.Gaussian}. Recall from the proof of Lemma~\ref{l.density} that $\widetilde {\mathcal Q} = \sum_{j=0}^n j^2 \widetilde c_j^2 t^{2j-2}= \widetilde f_n'(t^2) + t^2 \widetilde f''_n(t^2)$, and $\widetilde {\mathcal R} = \sum_{j=0}^n j \widetilde c_j^2 t^{2j-1} = t \widetilde f_n'(t^2)$.

Recall that $\widetilde u_n(x)= \widetilde f_n(x) (1-x)$. From Corollary~\ref{c.reciprocal-genKac}, \eqref{e.un-der}, and \eqref{e.un-der2}, for $x\in [1-b_1, 1-\frac cn]$ with $c>0$  sufficiently large we have
\begin{eqnarray*}
\widetilde f_n(x) &=& \frac 1{1-x} [1 + O(\frac 1{n(1-x)})]   \ \ ,\\
\widetilde f'_n(x) &=& \frac{\widetilde u'_n(x)  +  \widetilde f_n(x)}{1-x} = \frac{\widetilde f_n(x)}{1-x} + O(\frac 1 {n(1-x)^3})  = \frac 1{(1-x)^2} + O(\frac 1 {n(1-x)^3}) \ \ ,  \\
\widetilde f''_n(x) &=& \frac{\widetilde u''_n(x) + 2\widetilde f'_n(x)}{1-x} = \frac 2{(1-x)^3} + O(\frac 1{n(1-x)^4}). 
\end{eqnarray*}
It follows that for $t\in [1-b_1,1-\frac cn]$ we have
\begin{eqnarray*}
\widetilde{\mathcal P}(t) &=&   \frac{1}{1-t^2} +O(\frac 1{n(1-t^2)^2}) \  \  , \\ \widetilde{\mathcal Q}(t) &=&    \frac {2}{(1-t^2)^3} + O(\frac{1}{n(1-t^2)^4}) + O(\frac 1{(1-t^2)^2})\ \  , \\ 
\widetilde{\mathcal R}(t) &=&   \frac{t}{(1-t^2)^2} + O(\frac{1}{n(1-t^2)^3})
\end{eqnarray*}
and using Lemma~\ref{l.density-pw-outer} and the Edelman--Kostlan formula we have
\begin{eqnarray*}
\widetilde{\mathcal S} &=& \widetilde \rho_n (t)^2  \pi^2 \widetilde {\mathcal P}^2(t) = O(  \frac 1 {(1-t^2)^2} \widetilde{\mathcal P}^2 ).
\end{eqnarray*}
On the other hand, for $t\in [1-b_1,1]$,   using  Corollary~\ref{c.reciprocal-genKac} we have
\begin{eqnarray*}
\widetilde m(t) &=& \mu\sum_{j=0}^n \widetilde c_{j} t^j = \frac {\mu}{1-t}[1+O(\frac 1{n(1-t)})] \ \ .
\end{eqnarray*}
Let $d_j=c_j j$ which is a polynomial of $j$ (for $j\ge N_0$) of degree $\rho+1$. Then for $j\le n - N_0$ we have $\widetilde d_j = \frac{d_{n-j}}{d_n} = \widetilde c_j - \frac j n \widetilde c_j$. 
We obtain
$$\widetilde m'(t) = \mu\sum_{j=0}^n \widetilde c_{j}jt^{j-1} =  n\mu \sum_{j=0}^n \widetilde c_j t^{j-1}  - n\mu \sum_{j=0}^n \widetilde d_j t^{j-1} \ \ .$$
To evaluate $ \sum_{j=0}^n \widetilde d_j t^{j-1}$ and $ \sum_{j=0}^n \widetilde c_j t^{j-1}$, we use Corollary~\ref{c.finer} together with an expansion of the polynomials defining $c_j$ and $d_j$ into the linear basis of binomial polynomials $\frac{L\dots (L+j-1)}{j!}$ with $L= 1,2, \dots$ (as in the proof of Corollary~\ref{c.reciprocal-genKac}). It follows that
\begin{eqnarray*}
\widetilde m'(t) &=& \mu\sum_{j=0}^n \widetilde c_{j}jt^{j-1} =  n\mu \sum_{j=0}^n \widetilde c_j t^{j-1}  - n\mu \sum_{j=0}^n \widetilde d_j t^{j-1} \\
&=& n\mu \Big[\frac 1{1-t} (1+O(\frac 1 n))- \frac{\rho}{\rho+n} \frac{t}{(1-t)^2} + O(\frac 1 {n^2 (1-t)^2})] \\
&& \quad -\quad  n\mu \Big[\frac 1{1-t} (1+O(\frac 1 n)) - \frac{\rho+1}{\rho+1+n} \frac{t}{(1-t)^2} + O(\frac 1 {n^2 (1-t)^2})] \\
&=& \frac{\mu}{(1-t)^2} +  O(\frac{1}{1-t}). 
\end{eqnarray*}
Note that by choosing $b_1$ small and $c$ large we know that $1-t \ll 1$ and $\frac 1 {n(1-t)} \ll 1$. Thus, 
\begin{eqnarray*}
\widetilde m^2 \widetilde {\mathcal Q}+ \widetilde m'^2 \widetilde {\mathcal P} - 2\widetilde m \widetilde m' \widetilde{\mathcal R}  &\ge&   C^{-1} \frac{1}{(1-t)^5} \ge \widetilde {\mathcal P}^2 (1-t)^{-3} \ge  C^{-1} \widetilde {\mathcal S} (1-t)^{-1}
\end{eqnarray*}
and the rest of the proof is similar to the prior treatment for (the case $\rho=0$ of) $\E N_n[1-b_1, 1]$. In particular, to show that $\widetilde m'(t) = O(n|\widetilde m(t)|)$ for $t\in [1-\frac cn, 1]$ (in the treatment of $\widetilde I_2(1-\frac cn,1)$) we similarly observe that the main contributions to $|\widetilde m|$ and $|\widetilde m'|$ come from $0\le j\le n-n_0$, and for these indices we have  $\widetilde c_j>0$.

\underline{The interval $[-1-b_1, -1]$.} We will show that this interval contributes $(\log n)/(2\pi)+O(1)$ to $\E N_n$. As before we also consider the reciprocal polynomial $\widetilde P_n$ and count the number of real roots in $[-1, -1+b_1]$ for this polynomial. The analysis is similar to the treatment for the interval $[1,1+b_1]$;  the only modification is in the estimate for $\widetilde m$ and $\widetilde m'$ near $-1$, and unlike the last interval here   these two terms are bounded above by $O(1)$ (via applications of Lemma~\ref{l.reciprocal-fractional} together with an expansion of the polynomial defining $c_j$ into the linear basis of binomial polynomials). The rest of the proof is entirely similar to the prior treatment for (the case $\rho=0$ of)   $\E N_n [-1,-1+b_1]$.

\section{Appendix}
In this section, we provide the proof of Theorem \ref{log-com16-1}. For the proof of Theorem \ref{complex-series}, we need an analog of this theorem when $P$ is a power series of the form \eqref{series}. A proof for series in fact runs along the same line with the following proof for polynomials, except some minor modifications that we shall notify the reader.  

We first prove the following lemma.
\begin{lemma}\label{log-com33} Let $P$ be the random polynomial of the form \eqref{poly} where the $\xi_i$ are independent random variables with variance 1 and $\sup_{i\ge 0} \E|\xi_i|^{2+\ep}\le \tau_2$ for some constant $\tau_2$. And let $\tilde P = \sum_{i=0}^{\infty}c_i\tilde \xi_i z^{i}$ be the corresponding polynomial with Gaussian random variables $\tilde \xi_i$. Assume that $\tilde \xi_i$ matches moments to second order with $\xi_i$ for every $i\in \{0, \dots, n\}\setminus I_0$ for some subset $I_0$ (may depend on $n$) of size bounded by some constant $N_0$ and that $\sup_{i\ge 0} \E|\tilde \xi_i|^{2+\ep}\le \tau_2$. 

Then there exists a constant $C_2$ such that the following holds true. Let $\alpha_1\ge C_2 \alpha_0>0$ and $C>0$ be any constants. Let $\delta \in (0, 1)$ and $m\le \delta^{-\alpha_0}$ and $z_1,\dots, z_m\in \mathbb{C}$ be complex numbers such that 
\begin{equation}\label{log-com-e11}
\frac{|c_i||z_j|^i}{\sqrt{V(z_j)}}\le C\delta^{\alpha_1}, \forall i=0,\dots, n, j=1, \dots, m,
\end{equation} 
where $V(z_j) = \sum _{i=\{0, \dots, n\}\setminus I_0}|c_i|^{2}|z_j|^{2i}$.
Let $H:\mathbb{C}^m\to \mathbb{C}$ be any smooth function such that $\norm{\triangledown^a H} \le \delta^{-\alpha_0}$ for all $0\le a\le 3$, then 
\[\left|\E H\left(\frac{P(z_1)}{\sqrt{V(z_1)}}, \dots, \frac{P(z_m)}{\sqrt{V(z_m)}}\right)-\E H\left(\frac{\tilde P(z_1)}{\sqrt{V(z_1))}}, \dots, \frac{\tilde P(z_m)}{\sqrt{V(z_m)}}\right) \right|\le \tilde{C}\delta^{\alpha_0},
 \]
 where $\tilde{C}$ is a constant depending only on $\alpha_0, \alpha_1, C, N_0, \tau_2$ and not on $\delta$. 
\end{lemma}


\begin{proof} 
Our proof works for any subset $I_0$ of size bounded by $N_0$, but for notation convenience, we assume that $I_0 = \{0, \dots, N_0-1\}$. We use the Lindeberg swapping argument. Let $P_{i_0} = \sum_{i=0}^{i_0-1}c_i\tilde{\xi}_iz^i + \sum_{i=i_0}^n c_i\xi_iz^i$. Then $P_0 = P$ and $P_{n+1} = \tilde{P}$. Put
 \[I_{i_0} = \left|\E  H\left(\frac{P_{i_0}(z_1)}{\sqrt{V(z_1)}},\dots, \frac{P_{i_0}(z_m)}{\sqrt{V(z_m)}}\right)-\E  H\left(\frac{P_{i_0+1}(z_1)}{\sqrt{V(z_1)}},\dots, \frac{P_{i_0+1}(z_m)}{\sqrt{V(z_m)}}\right)\right|.
 \]
 Then\footnote{For power series, to have $I\le \sum_{i_0=0}^{\infty}I_{i_0}$, we need to show that 
\[\E  H\left(\frac{P_{0}(z_1)}{\sqrt{V(z_1)}},\dots, \frac{P_{0}(z_m)}{\sqrt{V(z_m)}}\right)-\E  H\left(\frac{\tilde P(z_1)}{\sqrt{V(z_1)}},\dots, \frac{\tilde P(z_m)}{\sqrt{V(z_m)}}\right)\]\[ = \sum_{i_0=0}^{\infty}\left(\E  H\left(\frac{P_{i_0}(z_1)}{\sqrt{V(z_1)}},\dots, \frac{P_{i_0}(z_m)}{\sqrt{V(z_m)}}\right)-\E  H\left(\frac{P_{i_0+1}(z_1)}{\sqrt{V(z_1)}},\dots, \frac{P_{i_0+1}(z_m)}{\sqrt{V(z_m)}}\right)\right),
\]
i.e., $\E  H\left(\frac{P_{n}(z_1)}{\sqrt{V(z_1)}},\dots, \frac{P_{n}(z_m)}{\sqrt{V(z_m)}}\right)\to\E  H\left(\frac{\tilde P(z_1)}{\sqrt{V(z_1)}},\dots, \frac{\tilde P(z_m)}{\sqrt{V(z_m)}}\right)$ as $n\to \infty$. This follows from the fact that $P_n(z_i)\to \tilde P(z_i)$ a.e., the continuity and boundedness of $H$, and the dominated convergence theorem.
},
\[I: = \big|\E H\big(\frac{P(z_1)}{\sqrt{V(z_1)}}, \dots, \frac{P(z_m)}{\sqrt{V(z_m)}}\big)-\E H\big(\frac{\tilde P(z_1)}{\sqrt{V(z_1)}}, \dots, \frac{\tilde P(z_m)}{\sqrt{V(z_m)}}\big) \big|\le \sum_{i_0=0}^{n} I_{i_0}. \]
 Fix $i_0\ge N_0$ and let $Y_j=\sum_{i=0}^{i_0-1}\frac{c_i\tilde{\xi}_iz_j^i}{\sqrt{V(z_j)}} + \sum_{i=i_0+1}^n \frac{c_i{\xi}_iz_j^i}{\sqrt{V(z_j)}}$ for $j$ from 1 to $n$. Then, $\frac{P_{i_0}(z_j)}{\sqrt{V(z_j)}} = Y_j + \frac{c_{i_0}\xi_{i_0}z_j^{i_0}}{\sqrt{V(z_j)}}$ and $\frac{P_{i_0+1}(z_j)}{\sqrt{V(z_j)}} = Y_j + \frac{c_{i_0}\tilde{\xi}_{i_0}z_j^{i_0}}{\sqrt{V(z_j)}}$.
 Fix $\xi_i$ when $i<i_0$ and $\tilde{\xi}_i$ when $i>i_0$ and the $Y_j$'s are fixed. Put $G = G_{i_0}(w_1,\dots, w_m) := H(Y_1+w_1,\dots, Y_m+w_m)$. Then $\norm{\triangledown^a G}_\infty \le C\delta^{-\alpha_0}$ for all $0\le a\le 3$.
 Then we need to estimate
 \[d_{i_0} := \left| \E _{\xi_{i_0}, \tilde{\xi}_{i_0}}G\left(\frac{c_{i_0}\xi_{i_0}z_1^{i_0}}{\sqrt{V(z_1)}},\dots, \frac{c_{i_0}\xi_{i_0}z_m^{i_0}}{\sqrt{V(z_m)}}\right) - \E _{\xi_{i_0}, \tilde{\xi}_{i_0}}G\left(\frac{c_{i_0}\tilde{\xi}_{i_0}z_1^{i_0}}{\sqrt{V(z_1)}},\dots, \frac{c_{i_0}\tilde{\xi}_{i_0}z_m^{i_0}}{\sqrt{V(z_m)}}\right)\right|
 \]
Let $a_{i, i_0} = \frac{c_{i_0}z_i^{i_0}}{\sqrt{V(z_i)}}$ and $a_{i_0} = (\sum_{i=1}^m |a_{i, i_0}|^2)^{1/2}$. Taylor expanding $G$ around $(0,\dots, 0)$ gives
\begin{equation}
G\left(a_{1, i_0}\xi_{i_0},\dots, a_{m, i_0}\xi_{i_0}\right) = G(0) + G_1 + \text{err}_1,\label{jul1}
\end{equation}
where \[G_1 = \frac{\text{d}G\left(a_{1, i_0}\xi_{i_0}t,\dots, a_{m, i_0}\xi_{i_0}t\right)}{\text{d}t}\bigg|_{t=0} = \sum_{i=1}^m \frac{\partial G(0)}{\partial Re(w_i)}Re(a_{i, i_0}\xi_{i_0})+\sum_{i=1}^m \frac{\partial G(0)}{\partial Im(w_i)}Im(a_{i, i_0}\xi_{i_0})\] and
\begin{eqnarray}
|\text{err}_1| &\le& \sup_{t'\in [0, 1]} \left| \frac{1}{2}\frac{\text{d}^2 G\left(a_{1, i_0}\xi_{i_0}t,\dots, a_{m, i_0}\xi_{i_0}t\right)}{\text{d} t^2}\right|_{t = t'} \nonumber\\
&=&\sup_{t'\in [0, 1]} \left|\frac{1}{2}\sum_{h, k \in \{Re, Im\}, i, j\in\{1,\dots, m\}}\frac{\partial^2 G}{\partial h(w_i)\partial k(w_j)}h(a_{i, i_0}\xi_{i_0})k(a_{j, i_0}\xi_{i_0})\right| \nonumber\\
&\le& \tilde C\delta^{-\alpha_0}|\xi_{i_0}|^{2}\sum_{i, j = 1}^m|a_{i, i_0}||a_{j, i_0}|\le \tilde C\delta^{-\alpha_0}|\xi_{i_0}|^2\left(\sum_{i=1}^m|a_{i, i_0}|\right)^2\nonumber\\
&\le& \tilde C\delta^{-\alpha_0}|\xi_{i_0}|^2m\left(\sum_{i=1}^m|a_{i, i_0}|^2\right)= \tilde C\delta^{-2\alpha_0}|\xi_{i_0}|^2a_{i_0}^2\nonumber.
\end{eqnarray}
Similarly,
\begin{equation}
G\left(a_{1, i_0}\xi_{i_0},\dots, a_{m, i_0}\xi_{i_0}\right) = G(0) + G_1 +\frac{1}{2}G_2+ \text{err}_2,\label{jul2}
\end{equation}
where $G_2 = \frac{\text{d}^2G(a_{1, i_0\xi_{i_0}}t,\dots, a_{m, i_0}\xi_{i_0}t)}{\text{d}t^2}\bigg|_{t=0}$ and
\begin{eqnarray}
|\text{err}_2| &
\le& \tilde C\delta^{-\frac{5}{2}\alpha_0}|\xi_{i_0}|^3a_{i_0}^3.\label{mm}
\end{eqnarray}
Also, we have $|\text{err}_2| = \left|\text{err}_1 - \frac{1}{2}G_2\right| \le \tilde C\delta^{-2\alpha_0}|\xi_{i_0}|^2a_{i_0}^2\le \delta^{-\frac{5}{2}\alpha_{0}}|\xi_{i_0}|^2a_{i_0}^2$.  
Interpolation gives
\begin{eqnarray}
|\text{err}_2| &
\le& \tilde C\delta^{-\frac{5}{2}\alpha_0}|\xi_{i_0}|^{2+\epsilon}a_{i_0}^{2+\epsilon}.\nonumber
\end{eqnarray}

The expression \eqref{mm} also holds for $\tilde{\xi}$ in place of $\xi$. Subtracting and taking expectations and using the matching moments give
\[d_{i_0} = \left|\E \text{err}_2\right| \le \tilde C\delta^{-\frac{5}{2}\alpha_0}a_{i_0}^{2+\epsilon}\left(\E |\xi_{i_0}|^{2+\epsilon}+\E |\tilde\xi_{i_0}|^{2+\epsilon}\right)\le \tilde C\delta^{-\frac{5}{2}\alpha_0}a_{i_0}^{2+\epsilon}.
\]
Taking expectation with respect to the random variables $\tilde{\xi}_i$ where $i<i_0$ and $\xi_i$ where $i>i_0$ gives
\[I_{i_0} \le \tilde C\delta^{-\frac{5}{2}\alpha_0}a_{i_0}^{2+\epsilon},
\]
for all $i_0\ge N_0$.

For $0\le i_0<N_0$, instead of \eqref{jul1} and \eqref{jul2}, we use mean value theorem to get the rough bound
\begin{equation}
G\left(a_{1, i_0}\xi_{i_0},\dots, a_{m, i_0}\xi_{i_0}\right) = G(0) + O(m ||\triangledown G||_{\infty}|\xi_{i_0}|\sum_{i = 1}^{m} |a_{i, i_0}|),\nonumber
\end{equation}
which by the same arguments as above gives
\[I_{i_0} \le \tilde C\delta^{-\frac{5}{2}\alpha_0}a_{i_0}.
\]

Thus,
\[I\le \tilde C\delta^{-\frac{5}{2} \alpha_0}\sum_{i_0 = 0}^n a_{i_0}^{2+\epsilon} + \tilde C\delta^{-\frac{5}{2} \alpha_0}\sum_{i_0 = 0}^{N_0} a_{i_0}.
\]
Note that since $a_{i_0}^2 = \sum_{i=1}^m |c_{i_0}|^{2}\frac{|z_i|^{2i_0}}{V(z_i)}$, $\sum_{i_0=0}^n a_{i_0}^2 = m + \sum_{i = 1}^{m}\sum_{i_0=0}^{N_0}\frac{|c_{i_0}|^{2}|z_i|^{2i_0}}{V(z_i)} = m + O(m\delta^{2\alpha_1}) = O(m)$. Moreover, since $\frac{|c_{i_0}||z_i|^{i_0}}{\sqrt{V(z_i)}}\le \tilde C\delta^{\alpha_1}$, $a_{i_0}^2\le m\tilde C^2\delta^{2\alpha_1}\le \tilde C^2\delta^{2\alpha_1-\alpha_0}$. Hence,
\begin{eqnarray}
I\le  \tilde C\delta^{-\frac{5}{2}\alpha_0+\epsilon(\alpha_1 - \frac{\alpha_0}{2})}\sum_{i_0 = 0}^n a_{i_0}^2  + \tilde C \delta^{\alpha_1 - 3\alpha_0}\le \tilde C\delta^{-\frac{5}{2}\alpha_0+\epsilon(\alpha_1 - \frac{\alpha_0}{2})}\delta^{-\alpha_0}+ \tilde C \delta^{\alpha_1 - 3\alpha_0}\le \tilde C\delta^{\alpha_0}.\nonumber
\end{eqnarray}
\end{proof}

Now we proceed to the proof of Theorem \ref{log-com16-1}. 
\begin{proof}
Consider $\bar F(w_1,\dots, w_m) = F(w_1 + \frac{1}{2}\log |V(z_1)|,\dots, w_m+\frac{1}{2}\log |V(z_m)|)$. Then, we still have $\norm{\triangledown ^a \bar F}_\infty \le C\delta^{-\alpha_0}$ for all $0\le \alpha\le 3$, and we want to show that 
\[\big|\E \bar F\big(\log\frac{|P(z_1)|}{\sqrt{V(z_1)}}, \dots, \log\frac{|P(z_m)|}{\sqrt{V(z_m)}}\big)-\E \bar F\big(\log\frac{|\tilde P(z_1)|}{\sqrt{V(z_1)}}, \dots, \log\frac{|\tilde P(z_m)|}{\sqrt{V(z_m)}}\big) \big|\le \tilde{C}\delta^{\alpha_0},
 \]

Let \[\Omega_1 = \{(w_1,\dots, w_m)\in \mathbb{R}^m: \min_{i=1,\dots, m} w_i<-M\}\] and \[\Omega_2 = \{(w_1,\dots, w_m)\in \mathbb{R}^m: \min_{i=1,\dots, m} w_i>-M-1\}\] where $M$ is to be defined. Then $\Omega_1\cup\Omega_2=\mathbb{R}^m\subset \mathbb{C}^m$, and since we only look at \[\bar F\left(\log \frac{|P(z_1)|}{\sqrt{V(z_1)}},\dots, \log \frac{|P(z_m)|}{\sqrt{V(z_m)}}\right),\] we can restrict $\bar F$ to $\mathbb{R}^m\subset\mathbb{C}^m$ and think about $\bar F$ as a function from $\mathbb{R}^m\to \mathbb{C}$. We can further assume that $\bar F:\mathbb{R}^m\to \mathbb{R}$ by considering the real and imaginary parts of $\bar F$ separately. \\\\
There exists a smooth function $\psi:\mathbb{R}^m\to \mathbb{R}$ such that $\psi$ is supported in $\Omega_2$ and $\psi=1$ on the complement of $\Omega_1$ and $\norm{\triangledown^a\psi}_\infty\le m^{C_2}$ for all $0\le a\le 3$ and $C_2$ is some constant. 

Indeed, there exists a function $\rho:\mathbb{R}\to \mathbb{R}$ such that $\rho$ is supported in $[-M-1, \infty)$, $\rho = 1$ on $[-M, \infty)$, $0\le \rho\le 1$, and $\rho$ has bounded derivatives of all orders. This function $\rho$ can be constructed by convolution of the indicator of $[-M-1/2,\infty)$ with a mollifier. Now, let $\psi(x_1,\dots, x_m) = \rho(x_1)\dots\rho(x_m)$. Then clearly, $\psi$ satisfies the required conditions.

Now, put $\phi= 1 - \psi$, $F_1 = \bar F.\phi$, and $F_2 = \bar F.\psi$. Then $\bar F =  F_1 + F_2$, and both $F_1, F_2$ are smooth functions with $\text{supp } F_1 \subset\bar \Omega_1, \text{supp } F_2\subset\bar \Omega_2$.
We have \[\norm{\triangledown F_1}=\norm{\triangledown \bar F.\phi + \bar F\triangledown \phi}\le\norm{\triangledown \bar F}\norm{\phi}+\norm{\bar F}\norm{\triangledown\phi} \le \tilde{C}\delta^{-C_2\alpha_0}.\] And similarly for higher derivatives and for $F_2$, we then get $\norm{\triangledown^aF_i}\le \tilde{C}\delta^{-C_2\alpha_0}$ for $i=1, 2$ and $0\le a\le 3$. 

We now show that the contribution from $F_1$ is negligible. We show this by first showing that there exists a smooth function $H_1:\mathbb{R}^m\to \mathbb{R}$ such that
$\ab{F_1(\log |w_1|,\dots,\log |w_m|)}\le H_1(w_1,\dots, w_m)$, $\norm{\triangledown^aH_1}\le \tilde{C}\delta^{-C_2\alpha_0}$
and $\text{supp} H_1\subset \{(w_1,\dots, w_m)\in \mathbb{R}^m: \min_{i=1,\dots, m}|w_i|\le e^{-M}\}$.
Indeed, since we have $\norm{\bar F}_\infty\le C\delta^{-\alpha_0}$, let $\tilde{F_1} = C\delta^{-\alpha_0}\phi$ then $\ab{F_1}\le \tilde{F_1}$ and $\norm{\triangledown^a\tilde{F}_1}\le \tilde{C}\delta^{-C_2\alpha_0}$. Then let \[H_1(w_1,\dots,w_m) = \tilde{F}_1(\log |w_1|,\dots, \log |w_m|).\] Since $\tilde{F}_1$ is constant on $\Omega_2^c$, $H_1$ is smooth. We have $\norm{H_1}\le \tilde{C}\delta^{-\alpha_0}$ and for all $a\ge 1$, $\triangledown^aH_1 = 0$ on $(\log |w_1|,\dots, \log |w_m|)\in \text{Int}(\Omega_2^c)\cup\text{Int}(\Omega_1^c)$. In the remaining domain $(\log |w_1|,\dots, \log |w_m|)\in \bar{\Omega_2}\cap\bar{\Omega_1}$, we have
\begin{eqnarray}
\ab{\frac{\partial H_1}{\partial w_1}}&=&\ab{\frac{\partial \tilde F_1}{\partial w_1}\frac{1}{\ab{w_1}}}\le  \tilde{C}\delta^{-\alpha_0}\ab{\frac{\partial \phi}{\partial w_1}}\frac{1}{|w_1|}\le\tilde{C}\delta^{-C_2\alpha_0}\frac{1}{\ab{w_1}}\nonumber,
\end{eqnarray}
where our constant $C_2$ can, as always, change from one line to another.
Since $\log \ab{w_1}\ge -M-4$, $\ab{w_1}\ge e^{-M-4}$. Thus, $\ab{\frac{\partial H_1}{\partial w_1}}\le\tilde{C}\delta^{-C_2\alpha_0}e^M$.
Similarly for higher derivatives, we get that $\norm{\triangledown^a H_1}\le \tilde{C}\delta^{-C_2\alpha_0}e^{3M}$. Choose $M = \log \left(\delta^{-3\alpha_0}\right)$ then $\norm{\triangledown^a H_1}\le \tilde{C}\delta^{-C_2\alpha_0}$ for all $0\le a\le 3$. 
Applying Lemma \ref{log-com33} to $\alpha_1$ and $C_2\alpha_0$
\begin{eqnarray}
\E \ab{F_1\left(\log \frac{|P(z_1)|}{\sqrt{V(z_1)}},\dots, \log \frac{|P(z_m)|}{\sqrt{V(z_m)}}\right)}&\le&\E {H_1\left( \frac{|P(z_1)|}{\sqrt{V(z_1)}},\dots,  \frac{|P(z_m)|}{\sqrt{V(z_m)}}\right)}\nonumber\\
&\le& \E {H_1\left( \frac{|\tilde P(z_1)|}{\sqrt{V(z_1)}},\dots,  \frac{|\tilde P(z_m)|}{\sqrt{V(z_m)}}\right)}+\tilde{C}\delta^{C_2\alpha_0}.\nonumber
\end{eqnarray}
Since $H_1 = 0$ if $(\log \ab{w_1},\dots, \log \ab{w_m})\notin \Omega_1$, one has
\begin{eqnarray}
\E {H_1\left( \frac{|\tilde P(z_1)|}{\sqrt{V(z_1)}},\dots,  \frac{|\tilde P(z_m)|}{\sqrt{V(z_m)}}\right)}&\le& 
\tilde{C}\delta^{-\alpha_0}\P \left(\exists i\in\{1,\dots, m\}:\frac{|\tilde P(z_i)|}{\sqrt{V(z_i)}}\le e^{-M} = \delta^{3\alpha_0}\right)\nonumber\\
&\le&\tilde{C}\delta^{-\alpha_0}m\delta^{3\alpha_0}\le \tilde{C}\delta^{\alpha_0}\nonumber.
\end{eqnarray}
Thus, $\E \ab{F_1\left(\log \frac{|P(z_1)|}{\sqrt{V(z_1)}},\dots, \log \frac{|P(z_m)|}{\sqrt{V(z_m)}}\right)}\le\tilde{C}\delta^{\alpha_0}$.
Finally, we will show that 
\[\ab{\E {F_2\left(\log \frac{|P(z_1)|}{\sqrt{V(z_1)}},\dots, \log \frac{|P(z_m)|}{\sqrt{V(z_m)}}\right)} - \E {F_2\left(\log \frac{|P(z_1)|}{\sqrt{V(z_1)}},\dots, \log \frac{|P(z_m)|}{\sqrt{V(z_m)}}\right)}}\le \tilde{C} \delta^{\alpha_0}.
\]
Define $H_2:\mathbb{R}^m\to \mathbb{R}$ by $H_2(w_1,\dots,w_m) = F_2(\log |w_1|,\dots, \log |w_2|)$. Since $\text{supp }F_2\subset \bar{\Omega_2}$, $\text{supp }H_2\subset \{(w_1,\dots, w_m): \log |w_i|\ge -M-4 \forall i\} = \{(w_1,\dots,w_m): |w_i|\ge \tilde{C}\delta^{3\alpha_0} \forall i\}$. Thus, $H_2$ is well-defined and smooth on $\mathbb{R}^m$. By a similar argument to the part about $H_1$, $\norm{\triangledown^aH_2}\le \tilde{C}\delta^{-C_2\alpha_0}$ for all $0\le a\le 3$. We can increase $C_2$ to have $C_2\ge 1$. Applying Lemma \ref{log-com33} to $\alpha_1$ and $C_2\alpha_0$ gives 
\begin{eqnarray}
&&\ab{\E {F_2\left(\log \frac{|P(z_1)|}{\sqrt{V(z_1)}},\dots, \log \frac{|P(z_m)|}{\sqrt{V(z_m)}}\right)} - \E {F_2\left(\log \frac{|P(z_1)|}{\sqrt{V(z_1)}},\dots, \log \frac{|P(z_m)|}{\sqrt{V(z_m)}}\right)}}\nonumber\\
&=&\ab{\E {H_2\left( \frac{|P(z_1)|}{\sqrt{V(z_1)}},\dots,  \frac{|P(z_m)|}{\sqrt{V(z_m)}}\right)} - \E {H_2\left( \frac{|P(z_1)|}{\sqrt{V(z_1)}},\dots,  \frac{|P(z_m)|}{\sqrt{V(z_m)}}\right)}}\nonumber\\
&\le& \tilde{C} \delta^{C_2\alpha_0}\le \tilde{C} \delta^{\alpha_0}.
\end{eqnarray}
This completes the proof.
\end{proof}

\bibliographystyle{plain}
\bibliography{longref}

\end{document}